\pdfoutput=1
\documentclass[11pt]{article}

\usepackage[utf8]{inputenc}
\usepackage{graphicx}
\usepackage[svgnames]{xcolor}
\usepackage{array,fullpage,microtype,enumitem,aliascnt}
\usepackage[leqno]{amsmath}
\usepackage{amssymb,amsthm,mathtools,bm,stmaryrd,tikz-cd,tensor}
\usepackage[unicode,bookmarksnumbered,bookmarksopen,psdextra,colorlinks=true,allcolors=DarkBlue]{hyperref}
\usepackage{bookmark}
\usepackage[nottoc]{tocbibind}
\usepackage[capitalize,nosort]{cleveref}

\usepackage{mymacros}

\makeatletter

\crefformat{footnote}{#2\footnotemark[#1]#3}

\def\yesnumber{\refstepcounter{equation}\tag{\theequation}}

\def\matcheq#1{%
    \expandafter\ifx\csname r@#1@cref\endcsname\relax%
        \@latex@warning{Reference `#1' on page \thepage \space undefined}%
    \else%
        \cref@getcounter{#1}{\@tempa}%
        \setcounter{equation}{\@tempa-1}%
    \fi%
}

\let\defn=\textbf
\let\displaystyle=\textstyle
\let\Dia=\Diamond
\def\supp{\mathrm{supp}}

\mathlig{=>}{\mathbin{\Rightarrow}}
\mathlig{<=>}{\mathbin{\Leftrightarrow}}
\mathlig{<<}{\ll}
\mathlig{<<<}{\lll}
\mathlig{>>>}{\ggg}

\newlist{eqenum}{enumerate}{1}
\let\c@eqenumi=\@undefined
\newaliascnt{eqenumi}{equation}
\setlist[eqenum]{label=(\thesubsection.\arabic*),align=left,leftmargin=4em,labelindent=0pt,labelwidth=!,resume}
\pretocmd\eqenum{\edef\enit@resume@eqenum{\noexpand\c@eqenumi\the\c@eqenumi}}{}{}
\crefname{eqenumi}{}{}

\makeatother

\begin{document}

\title{Structural, point-free, non-Hausdorff topological realization of Borel groupoid actions}
\author{Ruiyuan Chen}
\date{}
\maketitle

\begin{abstract}
We extend the Becker--Kechris topological realization and change-of-topology theorems for Polish group actions in several directions.
For Polish group actions, we prove a single result that implies the original Becker--Kechris theorems, as well as Sami's and Hjorth's sharpenings adapted levelwise to the Borel hierarchy; automatic continuity of Borel actions via homeomorphisms; and the equivalence of ``potentially open'' versus ``orbitwise open'' Borel sets.
We also characterize ``potentially open'' $n$-ary relations, thus yielding a topological realization theorem for invariant Borel first-order structures.
We then generalize to groupoid actions, and prove a result subsuming Lupini's Becker--Kechris-type theorems for open Polish groupoids, newly adapted to the Borel hierarchy, as well as topological realizations of actions on fiberwise topological bundles and bundles of first-order structures.

Our proof method is new even in the classical case of Polish groups, and is based entirely on formal algebraic properties of category quantifiers; in particular, we make no use of either metrizability or the strong Choquet game.
Consequently, our proofs work equally well in the non-Hausdorff context, for open quasi-Polish groupoids, and more generally in the point-free context, for open localic groupoids.
\let\thefootnote=\relax
\end{abstract}

\tableofcontents

\section{Introduction}
\label{sec:intro}

The interaction between topological and Borel structure is a central theme in analysis, topology, dynamics, and logic.
For instance, it is a well-known classical result that in a ``nice'' topological space, every Borel set can be made open in a finer topology that is still ``nice''.
Thus, the Borel $\sigma$-algebra remembers very little of the original topology.
Here ``nice'' can be taken for instance to mean \emph{Polish}, i.e., second-countable and completely metrizable.
See \cite[13.1]{Kcdst}.

The situation is markedly different in the presence of a group structure, where Pettis's automatic continuity theorem shows that in a Polish group, the topology can be fully recovered from the Borel structure together with the group structure.
See \cite[9.10]{Kcdst}, as well as \cite{Rautcts} for a detailed survey of automatic continuity phenomena.

The Becker--Kechris topological realization theorem \cite[5.2.1]{BKgrp} interpolates between these two extreme behaviors, by characterizing the topological information encoded in the Borel structure of a Polish group \emph{action}.
We now state one formulation of the Becker--Kechris theorem.
Not all parts below commonly appear in the literature in this form, although they are all easy consequences of \cite{BKgrp}.
We include a proof in this paper, as \cref{thm:grp-realiz-borel} (see also \cref{rmk:grp-realiz-0d-reg}).

\begin{theorem}[Becker--Kechris]
\label{thm:intro-becker-kechris}
Let $G$ be a Polish group, $X$ be a standard Borel $G$-space.
For any Borel set $A \subseteq X$, the following are equivalent:
\begin{enumerate}[label=(\roman*)]
\item \label{thm:intro-becker-kechris:open}
$A$ is open in some compatible Polish topology on $X$ making the action continuous.
\item
$A$ is a countable union of \defn{Vaught transforms}
\begin{equation*}
U_i * A_i := \{x \in X \mid \exists^* g \in G\, (g \in U_i \AND x \in gA_i)\},
\end{equation*}
where each $U_i \subseteq G$ is open and $A_i \subseteq X$ is Borel.
\item
The preimage of $A$ under the action map $G \times X -> X$ is a countable union of Borel rectangles.
\item
$A$ is \defn{orbitwise open}, i.e., its restriction to each orbit $G \cdot x$ is open in the quotient topology induced from the group topology on $G$ via the action $G ->> G \cdot x$.
\item
There are countably many Borel sets in $X$ generating all $G$-translates $g \cdot A$ under union.
\end{enumerate}
Moreover, countably many $A$ obeying these conditions may be made open as in \cref{thm:intro-becker-kechris:open} at the same time.
\end{theorem}

Here $\exists^*$ is the Baire category quantifier ``there exist non-meagerly many''.
The Vaught transform, denoted $U * A$ above, is more commonly denoted $A^{\triangle U^{-1}}$; see e.g., \cite{Kcdst}, \cite{BKgrp}, \cite{Gidst}.
The above ``multiplicative'' notation, reminiscent of the product set $U \cdot A$ of which it is the Baire-categorical analog, will be more convenient for our purposes in this paper.

Note that the conditions in the above statement clearly hold if $A$ is invariant.
The last sentence in the above statement also yields the change-of-topology theorem of \cite[5.1.8]{BKgrp}, that if $X$ is already a Polish $G$-space, then there is a finer Polish $G$-space topology making $A$ open.

In \cref{thm:grp-realiz}, we give a stronger version of \cref{thm:intro-becker-kechris} that allows us to place an upper bound on the resulting topology on $X$; see there for the precise statement.
For example, we recover as special cases Sami's~\cite{Sbeckec} and Hjorth's~\cite{Hbeckec} finer change-of-topology theorems adapted to each level of the Borel hierarchy, as \cref{thm:grp-realiz-dis} (see also \cref{thm:grp-realiz-dis-0d-reg}):

\begin{theorem}[Sami, Hjorth]
\label{thm:intro-sami-hjorth}
Let $G$ be a Polish group, $X$ be a Polish $G$-space, and $\xi \ge 2$ be a countable ordinal.
Then any countably many $G$-invariant $\*\Sigma^0_\xi$ sets may be made open in a finer Polish topology contained in $\*\Sigma^0_\xi(X)$ for which the action is still continuous.

Moreover, if $G$ is non-Archimedean, then the new topology may be taken to be zero-dimensional.
\end{theorem}

For later reference, we also state here another classical result \cite[9.16(i)]{Kcdst}, generalizing Pettis's automatic continuity theorem to actions, that will follow from \cref{thm:grp-realiz}.
This result is perhaps not usually viewed as a ``topological realization theorem''; however, we can (somewhat perversely) regard it as saying that we can find a topological realization ``compatible with'' (i.e., equal to) the preexisting topology.
See \cref{thm:grp-top-realiz}.

\begin{theorem}[classical]
\label{thm:intro-grp-top-realiz}
Let $G$ be a Polish group, $X$ be a Polish space with a Borel action of $G$ via homeomorphisms of $X$.
Then the action is jointly continuous.
\end{theorem}

\subsection{Topological realization of relations and structures}
\label{sec:intro-struct}

We now describe the main new results of this paper, which generalize in several directions the Becker--Kechris \cref{thm:intro-becker-kechris} as well as the related results described above.

\Cref{thm:intro-becker-kechris} characterizes the Borel sets $A \subseteq X$ in a Borel $G$-space which are ``potentially open'' in some topological realization.
In \cref{sec:grp-struct}, we consider more generally Borel relations $R \subseteq X^n$ of arbitrary finite arity $n \in \#N$, and more generally relations between different $G$-spaces.
The following generalizes \cref{thm:intro-becker-kechris} to characterize ``potentially open'' relations, and is part of \cref{thm:grp-binary-realiz-borel} (see also \cref{thm:grp-nary-realiz-borel,rmk:grp-realiz-0d-reg}).

\begin{theorem}[characterization of ``potentially open'' relations]
\label{thm:intro-grp-nary-realiz-borel}
Let $G$ be a Polish group, $X_i$ be countably many standard Borel $G$-spaces.
For an $n$-ary Borel relation $R \subseteq X_{i_1} \times \dotsb \times X_{i_n}$, the following are equivalent:
\begin{enumerate}[label=(\roman*)]
\item \label{thm:intro-grp-nary-realiz-borel:open}
$R$ is open in the product topology for some compatible Polish topologies on each $X_i$ making the action continuous.
\item
$R$ is a countable union of Vaught transforms $U_j * (A_{j,1} \times \dotsb \times A_{j,n})$ of Borel rectangles by Borel (or open) sets $U_j \subseteq G$, under the diagonal action $G \curvearrowright X_{i_1} \times \dotsb \times X_{i_n}$.
\item
The preimage of $R$ under the diagonal action map $G \times X_{i_1} \times \dotsb \times X_{i_n} -> X_{i_1} \times \dotsb \times X_{i_n}$ is a countable union of Borel rectangles.
\item
The preimage of $R$ under the diagonal action map is a countable union of rectangles $U_j \times A_{j,1} \times \dotsb \times A_{j,n}$ where $U_j \subseteq G$ is open and each $A_{j,k} \subseteq X_{i_k}$ is Borel orbitwise open.
\end{enumerate}
Moreover, countably many such $R$ (of varying arities) may be made open as in \cref{thm:intro-grp-nary-realiz-borel:open} at once.
\end{theorem}

Again, these conditions clearly hold if $R$ is invariant and a countable union of Borel rectangles.
In other words, we obtain a topological realization theorem for \emph{(multi-sorted) Borel relational structures}, in the sense of first-order logic, equipped with a $G$-action via automorphisms.

As in the unary case, we in fact prove a stronger version of the above result that takes an upper bound on the topologies; see \cref{thm:grp-binary-realiz}.
We may apply this to obtain a change-of-topology result for relations, generalizing \cref{thm:intro-sami-hjorth}; see \cref{thm:grp-nary-realiz-dis} (and \cref{rmk:grp-realiz-0d-reg}).

\begin{theorem}[change of topology for relations]
\label{thm:intro-grp-nary-realiz-dis}
Let $G$ be a Polish group, $X_i$ be countably many Polish $G$-spaces.
Then countably many invariant relations between them of various arities, each of which is a countable union of $\*\Sigma^0_\xi$ rectangles, may be made open in the products of finer Polish topologies on each $X_i$ contained in $\*\Sigma^0_\xi(X_i)$ for which the action is still continuous.
\end{theorem}

\subsection{Quasi-Polish $G$-spaces}
\label{sec:intro-qpol}

In \cite{dBqpol}, de~Brecht introduced a natural non-Hausdorff generalization of Polish spaces, called \defn{quasi-Polish spaces}, and proved that they obey nearly all of the basic descriptive set-theoretic properties of Polish spaces.
Moreover, several additional techniques are available for quasi-Polish spaces, and not Polish ones, which are particularly useful when working with Polish group actions.
(A quasi-Polish \emph{group} is automatically Polish, because topological groups are uniformizable.)

For instance, one equivalent characterization of quasi-Polish spaces is that they are precisely the continuous open $T_0$ quotients of Polish spaces; see \cref{it:qpol-openquot}.
From this, one may deduce that quasi-Polish spaces are precisely the $T_0$ quotients of spaces of orbits of Polish group actions on Polish spaces, also known as \emph{topological ergodic decompositions} of such actions; see \cite{Cbermd}.

For another instance, another fundamental result of Becker--Kechris \cite[2.6.1]{BKgrp} shows that $\@F(G)^\#N$ is a \emph{universal Borel $G$-space}, where $\@F(G)$ is the Effros Borel space of closed subsets of $G$.
The Becker--Kechris topological realization \cref{thm:intro-becker-kechris} then implies that $\@F(G)^\#N$ can be made into a \emph{Polish} $G$-space.
There are also various other known explicit examples of natural Polish $G$-spaces which are universal as Borel $G$-spaces, typically shown by embedding $\@F(G)^\#N$; see \cite{Gidst}, \cite{Knotes}.
In the quasi-Polish context, this picture is simplified: the Effros Borel space (of any quasi-Polish space) can be equipped with a canonical quasi-Polish topology to form the \defn{lower powerspace}; then $\@F(G)^\#N$ becomes a universal \emph{quasi-Polish} $G$-space.
See \cref{thm:grp-lowpow-univ}.

The above topological realization theorems are all equally valid for quasi-Polish $G$-spaces.
In fact, their proofs (as described in \cref{sec:intro-loc} below) naturally take place in the quasi-Polish context, with the Polish results stated above obtained via an additional argument at the very end; that is the point of the \cref{rmk:grp-realiz-0d-reg} referenced repeatedly above.

\subsection{Groupoid actions}
\label{sec:intro-gpd}

A \defn{groupoid} $G$ is a generalization of a group, where the elements $g \in G$, now called \emph{morphisms}, have a specified source or ``domain'' as well as target or ``codomain'' from among a set of \emph{objects} $G_0$, and composition is only defined for adjacent morphisms.
A \defn{groupoid action} on a family of sets $(X_x)_{x \in G_0}$, one for each object, has each morphism $g : x -> y \in G$ mapping from $X_x$ to $X_y$.
We may represent the family $(X_x)_{x \in G_0}$ formally as a \emph{bundle} $p : X := \bigsqcup_{x \in G_0} X_x -> G_0$, where each $X_x$ is recovered as the fiber $p^{-1}(x)$; this allows us to make sense of ``continuous actions'', ``Borel actions'', etc.
See \cref{sec:gpd-prelim} for the precise definitions.

Groupoids and their actions appear naturally in many contexts in dynamics and logic; see e.g., \cite{Rgpd}, \cite{SWgpd}, \cite{Lgpd}, \cite{GLgpd}, \cite{Cscc}, \cite{Cgpd}, \cite{Bgpd}.
Most relevantly for this paper, Lupini in \cite{Lgpd} developed analogs of much of the theory in \cite{BKgrp} for open%
\footnote{A topological groupoid is \defn{open} if the source map $G -> G_0$ is an open map; this is a standard assumption in a large part of the theory of topological groupoids.}
Polish groupoid actions, including the topological realization and change-of-topology theorems, as well as the result on universal actions mentioned in the preceding subsection.

In this paper, we generalize further to open \emph{quasi-Polish} groupoids $G$ and quasi-Polish $G$-spaces, for which we prove versions of all results aforementioned in this Introduction.
This is a substantial leap over \cite{Lgpd}, due to the pervasive use of metrizability in the classical theory.
Indeed, in \cite{Lgpd}, Lupini (following the earlier \cite{Rgpd}) already considered a slight generalization of Polish groupoids, allowing the space of morphisms to be \emph{$\sigma$-locally} Polish; such spaces still admit many of the classical metric techniques, such as the \emph{strong Choquet game} central to the original proof of the Becker--Kechris theorem \cite[\S5.2]{BKgrp}.
By contrast, the proofs in this paper look quite different from those in \cite{BKgrp}, and have a more abstract, ``algebraic'' flavor, as explained in the next subsection.

We will not restate, here in this Introduction, the versions for quasi-Polish groupoids of all aforementioned results, which largely consist of substituting ``group'' with ``groupoid'' everywhere and inserting some technical assumptions; see \cref{thm:gpd-realiz,thm:gpd-realiz-dis,thm:gpd-realiz-borel,thm:gpd-binary-realiz-borel,thm:gpd-nary-realiz-borel,thm:gpd-nary-realiz-dis}.
However, some new features of the groupoid setting are worth mentioning.

Whereas \cref{thm:intro-grp-top-realiz} may not appear much related to topological realization, the following generalization to groupoids is clearly an instance thereof, and is in fact an application of the stronger (upper-bounded) form of the groupoid version of \cref{thm:intro-becker-kechris}.
In \cref{def:fib-bor-qpol} we introduce the notion of a \defn{standard Borel bundle of quasi-Polish spaces} $f : X -> Y$ over a standard Borel base space $Y$, which intuitively means that each fiber $f^{-1}(y)$ is equipped with a quasi-Polish topology ``in a Borel way as $y$ varies''.
We then have the following; see \cref{thm:gpd-top-realiz}.

\begin{theorem}[topological realization of Borel $G$-bundles of spaces]
Let $G$ be an open quasi-Polish groupoid, $p : X -> G_0$ be a standard Borel bundle of quasi-Polish spaces equipped with a Borel action of $G$ via fiberwise homeomorphisms.
Then there is a compatible quasi-Polish topology on $X$ making $p$ and the action continuous, which also restricts to the originally given topology on each fiber.
\end{theorem}

We may combine this with the groupoid version of \cref{thm:intro-grp-nary-realiz-borel}, namely \cref{thm:gpd-nary-realiz-borel}, to obtain as part of \cref{thm:gpd-top-nary-realiz}:

\begin{theorem}[topological realization of Borel $G$-bundles of topological structures]
Let $G$ be an open quasi-Polish groupoid, $p_i : X_i -> G_0$ be countably many standard Borel bundles of quasi-Polish spaces, $R \subseteq X_{i_1} \times_{G_0} \dotsb \times_{G_0} X_{i_n}$ be an invariant $n$-ary Borel fiberwise (over $G_0$, in the fiber product topology) open relation.
Then there are compatible quasi-Polish topologies on each $X_i$ making $p_i$ and the action continuous and restricting to the topology on each fiber, such that $R \subseteq X_{i_1} \times_{G_0} \dotsb \times_{G_0} X_{i_n}$ becomes open.
Moreover, countably many such $R$ (of varying arities) may be made open at once.
\end{theorem}

As an application, in \cref{thm:gpd-etale-realiz} we rederive and generalize the core result of \cite[1.5]{Cscc}, a topological realization theorem for $G$-bundles of \emph{countable} structures as \emph{étale} $G$-bundles, which was originally proved in that paper (for bundles without any structure) using \emph{ad hoc} methods.

\subsection{Proof strategy: point-free topology}
\label{sec:intro-loc}

Over the past decade or so, it has become known that the topologies and Borel structures occurring in descriptive set theory can be usefully regarded as purely \emph{algebraic} structures.
Consider a topology $\@O(X)$ (of open sets) on a set $X$: it is a poset under inclusion, and is equipped with the operations of finite meets $\cap$ and arbitrary joins $\bigcup$, where $\cap$ distributes over $\bigcup$.
An abstract poset equipped with such operations is called a \defn{frame}.
The frames which are \emph{countably presented}, i.e., have a presentation $\ang{G \mid R}$ with countably many generators $G$ and relations $R$, are precisely the topologies of quasi-Polish spaces; imposing regularity yields the Polish topologies.
See \cite{Hqpol}.

Many basic constructions in descriptive set theory have conceptually simple descriptions from this algebraic perspective.
For instance, the lower powerspace of closed subsets $\@F(X)$ (mentioned in \cref{sec:intro-qpol} above) corresponds to forgetting about finite meets in $\@O(X)$ and then reintroducing them freely; see \cite{Vlog}, \cite{Spowloc}.
And the process of refining the topology to make Borel sets clopen corresponds to freely adjoining complements for existing elements of $\@O(X)$, thereby approaching the free Boolean $\sigma$-algebra generated by $\@O(X)$ which is the Borel $\sigma$-algebra $\@B(X)$; see \cite{Jnegopen}.

Moreover, such ``algebraic'' constructions tend to generalize straightforwardly to quasi-Polish spaces.
This is because, with the focus now on the (open, Borel, etc.)\ \emph{sets}, rather than \emph{points}, the usual sequential metric approximations that pervade classical arguments become quite unnatural.
In fact, if we forget about points altogether, then the resulting arguments often do not depend on countability at all.
A \defn{locale} $X$ is a ``topological space without points'', which is just to say, the same thing as an \emph{abstract} frame $\@O(X)$, whose elements by convention we call ``open sets of $X$''; see \cite[C1.1--1.2]{Jeleph}, \cite{Jstone}, \cite{PPloc}.
Over the past 40 years, much of classical descriptive set theory has been generalized to arbitrary locales, without any countability restrictions; see \cite{Cborloc}.

The work in this paper was originally motivated by attempting to find a point-free ``algebraic'' proof of the Becker--Kechris theorems, that would generalize to quasi-Polish $G$-spaces and more generally $G$-locales.
Now, the original proof in \cite[\S5.2]{BKgrp} is already point-free to a large extent, at least modulo some superficially point-based reasoning that can be relatively easily ``algebraicized''.
However, there is a glaring exception: the last step, \cite[\S5.2, Proof of 5.2.1, Claim 4]{BKgrp}, shows that the topology obtained is Polish via the \emph{strong Choquet game}, which is inextricably a sequential argument, and more subtlely, is best-suited to Hausdorff spaces (see \cite[\S10]{dBqpol}, \cite[\S11]{Cqpol}).

\subsection{Organization of paper}
\label{sec:intro-contents}

The plan of this paper, therefore, is to first carefully reprove the classical Becker--Kechris theorem for actions of Polish groups $G$ in a point-free manner that does not depend on countability, metrizability, or the Hausdorff axiom.
Such a proof will then work essentially verbatim for quasi-Polish $G$-spaces and more generally $G$-locales.
The generalization to groupoids is only slightly more involved, with some extra bookkeeping to keep track of fibers.
The results for $n$-ary relations described in \cref{sec:intro-struct,sec:intro-gpd} will also follow easily from our proof method.

Our proof is based entirely on formal algebraic properties of the Baire category quantifier $\exists^*$, in a fairly general context, namely with respect to a ``Borel bundle of spaces'' $f : X -> Y$ as in \cref{sec:intro-gpd}.
We recall and develop in \cref{sec:prelim} the theory of such bundles, which is well-known in some special cases (e.g., a product bundle $\pi_1 : Y \times Z -> Y$), but does not appear to have been written down before in the generality which we need.

In \cref{sec:grp}, we prove in full detail all of the topological realization theorems for Polish group actions described above.
This entails recalling/redeveloping the basic theory of Vaught transforms, again in a point-free manner, in \cref{sec:grp-vaught}.
While large parts of \cref{sec:grp} cover well-known results (slightly generalized to the quasi-Polish setting), our proofs are different from the standard ones, and will be reused in the following sections.

In \cref{sec:gpd}, we prove the various topological realization theorems for open quasi-Polish groupoid actions.
This section assumes familiarity with the previous one, to which we refer for identical proofs.
The main focus of \cref{sec:gpd} is on the new ``fiberwise'' subtleties and variations that arise.

In \cref{sec:loc}, we generalize everything to localic group(oid) actions.
The bulk of this section is devoted to developing point-free ``fiberwise'' topology and Baire category as in \cref{sec:prelim}, given which the arguments from \cref{sec:grp,sec:gpd} simply work verbatim.

We hope that our approach of first working out the details in the classical context of Polish group actions, and then indicating the tweaks needed for the more general contexts, will help to make this paper more accessible.
Familiarity with basic descriptive set theory is assumed throughout the paper.
In \cref{sec:loc} only, we additionally assume basic familiarity with category theory, lattices, and Boolean algebras.
A quick review of the needed locale theory is provided, although some conceptual familiarity here would be helpful as well.

\paragraph*{Acknowledgments}

I would like to thank Aristotelis Panagiotopoulos for some enlightening discussions regarding Sami's change-of-topology theorem \cite{Sbeckec}, as well as Matthew de~Brecht, Anush Tserunyan, and the anonymous referee for several helpful corrections and suggestions.
Research partially supported by NSF grant DMS-2224709.

\section{Topological preliminaries}
\label{sec:prelim}

\subsection{Topologies, Borel structures, and $\sigma$-topologies}
\label{sec:top}

In this paper, we will be dealing extensively with different topologies and Borel structures.
We therefore begin by carefully fixing some basic conventions.

For a topological space $X$, we will denote its \defn{topology}, i.e., lattice of open sets, by
\begin{equation}
\label{eq:top}
\@O(X) \subseteq \@P(X).
\end{equation}
We will occasionally need to consider multiple topologies on the same underlying set, which will be denoted $\@S, \@T, \@T', \dotsc \subseteq \@P(X)$; we reserve the notation $\@O(X)$ for a distinguished topology that $X$ is considered to be ``equipped'' with, for which we are willing to abuse notation as usual and denote the topological space by $X$ instead of $(X,\@O(X))$.

By a \defn{Borel space}, we will mean what is commonly called a \emph{measurable space}, i.e., a set equipped with an arbitrary $\sigma$-algebra of subsets (not assumed to be induced by any topology), which are called \defn{Borel}.
Similarly to \cref{eq:top}, we will denote the \defn{Borel $\sigma$-algebra} of a Borel space $X$ by
\begin{equation}
\@B(X) \subseteq \@P(X).
\end{equation}
For a topological space $X$, we equip it by default with the $\sigma$-algebra generated by $\@O(X)$, as usual.
A \defn{standard Borel space} is one whose $\sigma$-algebra is generated by a Polish topology.

We will be working with non-metrizable spaces, for which we use the modified \defn{Borel hierarchy} due to Selivanov~\cite{Sdst}.
The main difference from the usual definition is in level 2:
\begin{align}
\label{eq:top-bor}
A \in \*\Sigma^0_2(X)  &\coloniff  A = \bigcup_i (U_i \setminus V_i), &
A \in \*\Pi^0_2(X)  &\coloniff  A = \bigcap_i (U_i => V_i)
\end{align}
for countably many open $U_i, V_i \in \@O(X) =: \*\Sigma^0_1(X)$, where
$(U_i => V_i) := \{x \in X \mid x \in U_i \implies x \in V_i\}$.
For higher countable ordinals $\xi > 2$, we may define $\*\Sigma^0_\xi$ in the same way but taking $U_i, V_i \in \*\Sigma^0_{\zeta_i}(X)$ for $\zeta_i < \xi$ (and $\*\Pi^0_\xi$ to be the complements of $\*\Sigma^0_\xi$ sets); but here, it is enough to take $U_i = X$ as in the usual definition for metrizable spaces.
As usual, $\*\Delta^0_\xi := \*\Sigma^0_\xi \cap \*\Pi^0_\xi$.


The following simple facts take the place of the $T_1$ or Hausdorff axioms in many arguments in the non-Hausdorff context, and will be freely used without mention:
\begin{eqenum}
\item
Points in $T_0$ first-countable spaces are $\*\Pi^0_2$.
\item
The equality relation in a $T_0$ second-countable space $X$ is $\*\Pi^0_2$ in $X^2$.
\end{eqenum}

The following common generalization of topologies and $\sigma$-algebras will allow us to unify some analogous statements between the topological versus Borel contexts:

\begin{definition}
\label{def:stop}
A \defn{$\sigma$-topology} on a set $X$ is a collection of subsets $\@S \subseteq \@P(X)$ closed under finite intersections and \emph{countable} unions, whose elements are called \defn{$\sigma$-open}.

Note that a second-countable $\sigma$-topology is the same thing as a second-countable topology, since an arbitrary union reduces to a countable union of basic opens.

The definition of the Borel hierarchy makes sense also in $\sigma$-topological spaces.
If $X$ is a $\sigma$-topological space, then $\*\Sigma^0_\xi(X)$ is a finer $\sigma$-topology for each $\xi < \omega_1$, as is their union $\@B(X)$.
\end{definition}

\begin{definition}
\label{def:ostar}
If $\star$ is a binary operation on sets, then for two families of sets $\@S$ and $\@T$, we write
\begin{equation*}
\textstyle
\@S \ostar \@T := \{\bigcup_{i \in I} (A_i \star B_i) \mid I \text{ countable} \AND A_i \in \@S \AND B_i \in \@T\}.
\end{equation*}
(Typically $\@S, \@T$ are closed under countable unions, over which $\star$ distributes.)

For instance,
if $\@S \subseteq \@P(X)$ and $\@T \subseteq \@P(Y)$ are $\sigma$-topologies, then $\@S \otimes \@T \subseteq \@P(X \times Y)$ is the \defn{product $\sigma$-topology}, consisting of all countable unions of rectangles.
Thus if $\@S, \@T$ are second-countable topologies, then $\@S \otimes \@T$ is the product topology.
See also \cref{def:grp-vaught,def:gpd-vaught}.
\end{definition}

We now recall some basic notions and facts surrounding Baire category.
In a topological space $X$, a subset $A \subseteq X$ is \defn{comeager} if it contains a countable intersection of dense open sets, and \defn{meager} if its complement $\neg A$ is comeager.
We write
\begin{align}
\label{eq:exists*}
\exists^* A  &\coloniff  \exists^* x \in X\, (x \in A)  \coloniff  A \text{ is nonmeager}, \\
\label{eq:forall*}
\forall^* A  &\coloniff  \forall^* x \in X\, (x \in A)  \coloniff  A \text{ is comeager}  \iff  \neg \exists^* \neg A.
\end{align}
We say $X$ is a \defn{Baire space} if the Baire category theorem holds in it, i.e., every comeager set is dense, or equivalently every nonempty open set is nonmeager; and $X$ is \defn{completely Baire} if every closed subspace $Y \subseteq X$ is Baire, or equivalently every $\*\Pi^0_2$ subspace is Baire (see e.g., \cite[7.2]{Cqpol}).
For $A, B \subseteq X$, we have the relations of \defn{containment and equality mod meager}:
\begin{align}
\label{eq:subset*}
A \subseteq^* B  &\coloniff  A \setminus B \text{ is meager}  \iff  \forall^*(A => B), \\
\label{eq:=*}
A =^* B  &\coloniff  A \subseteq^* B \subseteq^* A  \iff  A \triangle B \text{ is meager}.
\end{align}
We say that $A \subseteq X$ has the \defn{Baire property} if it is $=^*$ to an open set; such sets include all open sets and form a $\sigma$-algebra, hence include all Borel sets.
Explicitly, we have the following formulas which show how to inductively construct, for each Borel set $A \subseteq X$, an open set $U_A =^* A$:
\begin{align}
\label{eq:bp-union}
U_{\bigcup_i A_i} &:= \bigcup_i U_{A_i}, \\
\label{eq:bp-diff}
U_{A \setminus B} &:= U_A \cap (\neg U_B)^\circ = \bigcup_{W \in \@W; W \cap U_B = \emptyset} (W \cap U_A)
\end{align}
where $\@W$ is any open basis for $X$ (and $(-)^\circ$ denotes interior).
These formulas will be particularly useful for working uniformly with ``bundles'' of spaces; see \cref{sec:fib}.

\subsection{Quasi-Polish spaces}
\label{sec:qpol}

The main topological setting of this paper is de~Brecht's \defn{quasi-Polish spaces} \cite{dBqpol} (see also \cite{Cqpol}), which can be defined via \cref{it:qpol-pi02} or \cref{it:qpol-openquot} below.
We list here some basic properties of quasi-Polish spaces we will freely use; for proofs, see the aforementioned references.
\begin{eqenum}

\item \label{it:qpol-t0-cb-baire}
All quasi-Polish spaces are $T_0$, second-countable, and completely Baire.

\item \label{it:qpol-pol}
A topological space is Polish iff it is quasi-Polish and regular ($T_3$).

\item \label{it:qpol-s}
The \defn{Sierpinski space} $\#S = \{0, 1\}$, where $\{1\}$ is open but not closed, is quasi-Polish.

\item \label{it:qpol-dis}
If $X$ is a quasi-Polish space, and $A_i \in \*\Sigma^0_\xi(X)$ are countably many sets, then there is a finer quasi-Polish topology containing each $A_i$ and contained in $\*\Sigma^0_\xi(X)$.
In more detail,
\begin{enumerate}[label=(\alph*)]
\item \label{it:qpol-dis:clopen}
adjoining a single $\*\Delta^0_2$ set to the topology of $X$ preserves quasi-Polishness;
\item \label{it:qpol-dis:join}
if the intersection of countably many quasi-Polish topologies contains a quasi-Polish topology, then their union generates a quasi-Polish topology.
\end{enumerate}

\item \label{it:qpol-sbor}
A quasi-Polish space can be made zero-dimensional Polish by adjoining countably many closed sets to the topology, hence is in particular standard Borel (in the usual sense).

\item \label{it:qpol-prod}
Countable products of quasi-Polish spaces are quasi-Polish.

\item \label{it:qpol-loc}
A space with a countable cover by open quasi-Polish subspaces is quasi-Polish.

\item \label{it:qpol-pi02}
A subspace of a quasi-Polish space is quasi-Polish iff it is $\*\Pi^0_2$ (in the sense of \cref{eq:top-bor}).
In fact, quasi-Polish spaces are precisely the $\*\Pi^0_2$ subspaces of $\#S^\#N$, up to homeomorphism.

\item \label{it:qpol-openquot}
A continuous open $T_0$ quotient of a quasi-Polish space is quasi-Polish.
In fact, nonempty quasi-Polish spaces are precisely the continuous open $T_0$ quotients of $\#N^\#N$.

\end{eqenum}
This last property ultimately underlies all of the topological realization results in this paper.

\begin{definition}
\label{def:qpol-compat}
As usual, for a standard Borel space $X$, we say that a quasi-Polish topology $\@O(X)$ on $X$ is \defn{compatible} (with the Borel structure) if $\@O(X) \subseteq \@B(X)$; it then follows (by the Lusin--Suslin theorem) that $\@O(X)$ generates $\@B(X)$ as a $\sigma$-algebra.

We also say that a $\sigma$-topology $\@S$ on $X$ is \defn{compatible} (with the Borel structure) if every countable subset of $\@S$ is contained in a compatible quasi-Polish topology contained in $\@S$.
It follows that $\@S \subseteq \@B(X)$, that $\@S$ generates $\@B(X)$, and that $\@S$ contains at least one quasi-Polish topology.
\end{definition}

\begin{example}
If $\@S$ is second-countable, then to say that $\@S$ is a compatible $\sigma$-topology is the same as to say that $\@S$ is a compatible quasi-Polish topology.
\end{example}

\begin{example}
\label{ex:qpol-compat-borel}
For a quasi-Polish space $X$, each $\*\Sigma^0_\xi(X)$ is a finer compatible $\sigma$-topology (by \cref{it:qpol-dis}), as is their union $\@B(X)$.
\end{example}

\begin{example}
If $X, Y$ are standard Borel spaces with compatible $\sigma$-topologies $\@S(X) \subseteq \@B(X)$ and $\@S(Y) \subseteq \@B(Y)$, then the product $\sigma$-topology $\@S(X) \otimes \@S(Y) \subseteq \@B(X \times Y)$ (from \cref{def:ostar}) is a compatible $\sigma$-topology on $X \times Y$.
Indeed, given countably many countable unions of rectangles $A_i \times B_i \in \@S(X) \otimes \@S(Y)$, we may find quasi-Polish topologies $\{A_i\}_i \subseteq \@O(X) \subseteq \@S(X)$ and $\{B_i\}_i \subseteq \@O(Y) \subseteq \@S(Y)$; then $\@O(X) \otimes \@O(Y)$ contains each $A_i \times B_i$.
\end{example}

\subsection{Fiberwise topology}
\label{sec:fib}

Given any function $f : X -> Y$ between sets, we may regard $X$ as a \defn{bundle} over $Y$, i.e., as a family of sets, the \defn{fibers} $f^{-1}(y)$, indexed over $y \in Y$.
In general, when we refer to an \defn{$f$-fiberwise} concept or property, we will mean that it occurs simultaneously on each fiber.

\begin{definition}
\label{def:fib-top}
An \defn{$f$-fiberwise topology} on $X$ will mean a family of topologies, one on each fiber $f^{-1}(y)$.
We identify such a family of topologies with a global topology on $X$, namely given by the disjoint union of the fibers.
Terms like \defn{$f$-fiberwise open}, \defn{$f$-fiberwise meager}, etc., have a self-explanatory meaning.
We denote the $f$-fiberwise open sets, i.e., the corresponding global topology, by $\@O_f(X)$.
We also call $f : X -> Y$ equipped with an $f$-fiberwise topology $\@O_f(X)$ a \defn{bundle of topological spaces over $Y$}.

If $X$ is already equipped with a (global) topology $\@O(X)$, we may \defn{restrict} it to each fiber to get an $f$-fiberwise topology, whose corresponding global topology $\@O_f(X)$ \emph{refines} the original $\@O(X)$.

We say that a family $\@U \subseteq \@O_f(X)$ of $f$-fiberwise open subsets is an \defn{$f$-fiberwise open subbasis} if for each $y \in Y$, the restrictions $f^{-1}(y) \cap U$ of all $U \in \@U$ form an open subbasis for $f^{-1}(y)$.
(This does \emph{not} mean that $\@U$ is a subbasis for the global topology $\@O_f(X)$.)
Thus, a countable such $\@U$ exists iff $X$ is \defn{$f$-fiberwise second-countable} in the self-explanatory sense.

We say that $\@U$ is an \defn{$f$-fiberwise open basis} if, moreover, $\@U$ covers $X$ and the intersection of any two sets in $\@U$ is a union of other sets in $\@U$; in other words, if $\@U$ forms a basis for \emph{some} global topology on $X$, restricting to the fiberwise topology (but in general coarser than $\@O_f(X)$).
Note that this is a bit stronger than requiring $\@U$ to restrict to a basis on each fiber.
Clearly, the closure under finite intersections of any fiberwise subbasis is a fiberwise basis.
\end{definition}

\begin{definition}
\label{def:fib-pb}
Recall that for any functions $f : X -> Y$ and $g : Z -> Y$, we have the \defn{pullback} or \defn{fiber product}
\begin{equation*}
Z \times_Y X := \{(z,x) \in Z \times X \mid g(z) = f(x)\},
\end{equation*}
which fits into a commutative square
\begin{equation}
\label{diag:fib-pb}
\begin{tikzcd}
Z \times_Y X \dar["\pi_1"'] \rar["\pi_2"] & X \dar["f"] \\
Z \rar["g"] & Y
\end{tikzcd}
\end{equation}
where $\pi_1, \pi_2$ are the projections, such that each $\pi_1$-fiber $\pi_1^{-1}(z)$ is in canonical bijection (via $\pi_2$) with the $f$-fiber $f^{-1}(g(z))$.
In this situation, we also call $\pi_1$ the \defn{pullback of $f$ along $g$}.

If $X$ has an $f$-fiberwise topology, we may therefore transfer it via $\pi_2$ to a \defn{pullback $\pi_1$-fiberwise topology} on $Z \times_Y X$.
Henceforth whenever we have a pullback of a bundle of topological spaces, we will by default regard it as being equipped with the pullback fiberwise topology.
\end{definition}

The following trivial facts, as well as their various analogs (e.g., \cref{it:fib-baire-bc} below), will play a key role.
For a pullback square as above, the \defn{Beck--Chevalley condition} says that for $A \subseteq X$,
\begin{equation}
\label{eq:fib-bc}
g^{-1}(f(A)) = \pi_1(\pi_2^{-1}(A)).
\end{equation}
A special case is \defn{Frobenius reciprocity}: for $Z \subseteq Y$ and $g$ the inclusion,
\begin{equation}
\label{eq:fib-frob}
Z \cap f(A) = f(f^{-1}(Z) \cap A).
\end{equation}

\begin{definition}
\label{def:fib-baire}
Let $f : X -> Y$ be a bundle of topological spaces.
For any $A \subseteq X$, we define the \defn{Baire category quantifiers}
\begin{align*}
\exists^*_f(A) &:= \{y \in Y \mid \exists^* x \in f^{-1}(y)\, (x \in A)\}, \\
\forall^*_f(A) &:= \{y \in Y \mid \forall^* x \in f^{-1}(y)\, (x \in A)\} = \neg \exists^*_f(\neg A)
\end{align*}
(where the $\exists^*, \forall^*$ are with respect to the fiberwise topology on $f^{-1}(y)$).
For $A, B \subseteq X$, we write
\begin{align*}
A \subseteq^*_f B
&\coloniff  \forall y \in Y\, (f^{-1}(y) \cap A \subseteq^* f^{-1}(y) \cap B)
\iff  Y = \forall^*_f(A => B), \\
A =^*_f B
&\coloniff  \forall y \in Y\, (f^{-1}(y) \cap A =^* f^{-1}(y) \cap B)
\iff  Y = \forall^*_f(A <=> B).
\end{align*}
Note that these reduce to the ``absolute'' Baire category notions \cref{eq:exists*}--\cref{eq:=*} when $Y = 1$.
\end{definition}

We now record some basic facts about fiberwise Baire category, most of which are usually stated in the special case of a product bundle $X = Y \times Z$ (and $f = \pi_1$) but generalize straightforwardly; see \cite[\S8.J]{Kcdst}, \cite[\S7]{Cqpol}.
First, from the fact that $\exists^*_f$ is defined fiberwise, we clearly have:
\begin{eqenum}

\item \label{it:fib-baire-bc}
(\defn{Beck--Chevalley condition})
For a pullback square as in \cref{diag:fib-pb},
\begin{equation*}
g^{-1}(\exists^*_f(A)) = \exists^*_{\pi_1}(\pi_2^{-1}(A)).
\end{equation*}

\item \label{it:fib-baire-frob}
(\defn{Frobenius reciprocity})
In particular, for $Z \subseteq Y$,
\begin{equation*}
Z \cap \exists^*_f(A) = \exists^*_f(f^{-1}(Z) \cap A).
\end{equation*}

\item \label{it:fib-baire-surj}
In particular, for $Z \subseteq \exists^*_f(X)$,
\begin{equation*}
Z = \exists^*_f(f^{-1}(Z)).
\end{equation*}

\end{eqenum}
The next few properties are direct fiberwise translations of basic properties of ``global'' Baire category from \cref{sec:top}:
\begin{eqenum}

\item \label{it:fib-baire-im}
$\exists^*_f(A) \subseteq f(A)$, with equality if $A$ is fiberwise open and $X$ is fiberwise Baire.

\item \label{it:fib-bp-open}
Thus, if $X$ is fiberwise Baire and $A =^*_f U \in \@O_f(X)$, then $\exists^*_f(A) = \exists^*_f(U) = f(U)$.

\item \label{it:fib-baire-union}
$\exists^*_f$ preserves countable unions; $\forall^*_f$ preserves countable intersections.

\end{eqenum}
The following are fiberwise translations of the formulas \cref{eq:bp-union} and \cref{eq:bp-diff}:
\begin{eqenum}

\item \label{it:fib-bp-union}
For countably many $A_i \subseteq X$, if each $A_i =^*_f U_{A_i}$, then $\bigcup_i A_i =^*_f U_{\bigcup_i A_i} := \bigcup_i U_{A_i}$.

\item \label{it:fib-bp-diff}
If $A =^*_f U_A \in \@O_f(X)$ and $B =^*_f U_B \in \@O_f(X)$, then for any $f$-fiberwise basis $\@W \subseteq \@O_f(X)$,
\begin{align*}
A \setminus B =^*_f U_{A \setminus B} := \bigcup_{W \in \@W} (W \cap U_A \setminus f^{-1}(f(W \cap U_B))).
\end{align*}

\end{eqenum}
Applying $\exists^*_f$ to this last formula yields, assuming $X$ is fiberwise Baire and for a countable fiberwise basis $\@W$, using \cref{it:fib-bp-open} and \cref{eq:fib-frob},
\begin{align}
\label{it:fib-baire-diff}
\exists^*_f(A \setminus B)
&= \bigcup_{W \in \@W} (\exists^*_f(W \cap A) \setminus \exists^*_f(W \cap B)).
\end{align}
By induction on $\xi$, these last few properties yield the following; see \cite[22.22]{Kcdst}, \cite[7.5]{Cqpol}.

\begin{proposition}
\label{thm:fib-baire-borel}
Let $f : X -> Y$ be a continuous open fiberwise Baire map from a second-countable space $X$ to an arbitrary topological space $Y$.
Then
\begin{enumerate}[label=(\alph*)]
\item  (fiberwise Baire property)
For any $A \in \*\Sigma^0_\xi(X)$, there is a fiberwise open $U_A \subseteq X$, of the form $U_A = \bigcup_i (f^{-1}(B_i) \cap U_i)$ where $B_i \in \*\Sigma^0_\xi(Y)$ and $U_i \in \@O(X)$, such that $A =^*_f U_A$.
\item
Thus, $\exists^*_f(\*\Sigma^0_\xi(X)) \subseteq \*\Sigma^0_\xi(Y)$.
\qed
\end{enumerate}
\end{proposition}

\subsection{Borel fiberwise topology}
\label{sec:fib-bor}

We now consider bundles of spaces $f : X -> Y$ where the base space $Y$ is standard Borel, and the fiberwise topology is ``uniformly Borel''.

The key tool enabling a well-behaved theory of such ``Borel fiberwise topology'' is the following classical result.
It is usually stated for the case of a product bundle $X = Y \times Z$, with $\@S$ consisting of the cylinders $Y \times U$ for $U$ in some countable open basis for $Z$; see \cite[28.7]{Kcdst}.
However, essentially the same proof yields the general form below, which was pointed out in \cite[8.14]{Cgpd}.

\begin{theorem}[Kunugui--Novikov uniformization]
\label{thm:kunugui-novikov}
Let $f : X -> Y$ be a Borel map between standard Borel spaces, $\@S$ be a countable family of Borel subsets of $X$.
If a Borel set $A \subseteq X$ is $f$-fiberwise a union of sets in $\@S$, then $A = \bigcup_{S \in \@S} (f^{-1}(B_S) \cap S)$ for some Borel sets $B_S \subseteq Y$.
\end{theorem}

\begin{definition}
\label{def:fib-bor-qpol}
Let $f : X -> Y$ be a Borel map between standard Borel spaces, and suppose $X$ is equipped with an $f$-fiberwise topology.
We call $X$ a \defn{standard Borel bundle of quasi-Polish spaces over $Y$} if $X$ is fiberwise quasi-Polish, and is ``uniformly fiberwise second-countable'' in that it has a countable fiberwise open (sub)basis consisting of \emph{Borel} sets $\@U \subseteq \@B(X)$.
For such $X$, we let
\begin{equation*}
\@{BO}_f(X) := \@B(X) \cap \@O_f(X)
\end{equation*}
denote the $\sigma$-topology of Borel $f$-fiberwise open sets in $X$.
\end{definition}

It is perhaps not obvious that this is the ``correct'' definition of a ``uniformly quasi-Polish'' bundle of spaces.
The definition of quasi-Polish space requires not only that the topology is ``countably generated'' (i.e., second-countable), but also ``countably presented'' (i.e., $\*\Pi^0_2$; see \cref{it:qpol-pi02}); why do we require uniformity of only the former but not the latter?
The following shows that it is automatic (recall the notion of a \emph{compatible $\sigma$-topology} from \cref{def:qpol-compat}):

\begin{proposition}[topological realization of Borel bundles]
\label{thm:fib-bor-qpol}
Let $f : X -> Y$ be a standard Borel bundle of quasi-Polish spaces over a standard Borel space $Y$.
\begin{enumerate}[label=(\alph*)]
\item \label{thm:fib-bor-qpol:basis}
For any countable Borel fiberwise open basis $\@U \subseteq \@{BO}_f(X)$, $\@{BO}_f(X)$ consists of precisely all sets of the form $\bigcup_{U \in \@U} (f^{-1}(B_U) \cap U)$ for Borel sets $B_U \subseteq Y$.
\item \label{thm:fib-bor-qpol:top}
There are compatible quasi-Polish topologies $\@O(X)$ and $\@O(Y)$ making $f$ continuous, such that $\@O(X)$ restricts to the fiberwise topology on $X$.
Moreover, we may choose $\@O(X)$ to include any countably many $U_i \in \@{BO}_f(X)$; in particular, $\@{BO}_f(X)$ is a compatible $\sigma$-topology on $X$.
\end{enumerate}
\end{proposition}
\begin{proof}
\cref{thm:fib-bor-qpol:basis} is an immediate consequence of Kunugui--Novikov (\cref{thm:kunugui-novikov}).

\cref{thm:fib-bor-qpol:top}
Let $\@U \subseteq \@{BO}_f(X)$ be a countable Borel fiberwise open basis for $X$, including each given $U_i$.
First, suppose that each $U \in \@U$ is in fact fiberwise clopen.
We then have a fiberwise embedding
\begin{align*}
e : X &--> Y \times 2^\@U \\
x &|--> (f(x), (\chi_U(x))_{U \in \@U})
\end{align*}
where each $\chi_U$ is the characteristic function of $U$; this is a fiberwise embedding because $\@U$ is a fiberwise basis.
By the Lusin--Suslin theorem (see e.g., \cite[15.1]{Kcdst}), $e(X) \subseteq Y \times 2^\@U$ is Borel.
Since $e$ is a fiberwise embedding, its image is ($\pi_1$-)fiberwise $G_\delta$; thus by Saint~Raymond's uniformization theorem for Borel fiberwise $G_\delta$ sets \cite[35.45]{Kcdst}, together with Kunugui--Novikov,
\begin{align*}
e(X) = \bigcap_i \bigcup_j (B_{ij} \times V_{ij})
\end{align*}
for some countably many Borel $B_{ij} \subseteq Y$ and open $V_{ij} \subseteq 2^\@U$.
Find any compatible zero-dimensional Polish topology on $Y$ making these $B_{ij}$ clopen.
Then $e(X) \subseteq Y \times 2^\@U$ is $G_\delta$, hence we may pull back the subspace topology along $e$ to $X$, yielding a compatible zero-dimensional Polish topology.

If each $U \in \@U$ is merely fiberwise open, we may run the above argument replacing the role of $2^\@U$ with $\#S^\@U$, using the following quasi-Polish generalization of Saint~Raymond's theorem, to get a compatible quasi-Polish topology on $X$, homeomorphic to a $\*\Pi^0_2$ subspace of $Y \times \#S^\@U$.
%
\end{proof}

\begin{lemma}
Let $Y$ be a standard Borel space, $Z$ be a quasi-Polish space, $A \subseteq Y \times Z$ be a $\pi_1$-fiberwise $\*\Pi^0_2$ set.
Then there are countably many Borel $B_{ij} \subseteq Y$ and open $U_i, V_{ij} \subseteq Z$ such that
\begin{align*}
A = \bigcap_i \paren[\big]{(Y \times U_i) => \bigcup_j (B_{ij} \times V_{ij})}.
\end{align*}
\end{lemma}
\begin{proof}
By \cref{it:qpol-openquot}, let $Z'$ be Polish and $g : Z' ->> Z$ be a continuous open surjection.
Then $A' := (Y \times g)^{-1}(A) \subseteq Y \times Z'$ is $\pi_1$-fiberwise $G_\delta$, hence by Saint~Raymond's theorem
\begin{align*}
A' = \bigcap_{i'} \bigcup_j (B_{i'j} \times V'_{i'j})
\end{align*}
for countably many Borel $B_{i'j} \subseteq Y$ and open $V'_{i'j} \subseteq Z'$.
Then
\begin{align*}
A
&= \forall^*_{Y \times g}(A') &&\text{by \cref{it:fib-baire-surj}} \\
&= \bigcap_{i'} \forall^*_{Y \times g}(\bigcup_j (B_{i'j} \times V'_{i'j})) &&\text{by \cref{it:fib-baire-union}} \\
&= \bigcap_{i'} \neg \bigcup_{U \in \@U} \paren[\big]{(Y \times g)(Y \times U) \setminus \exists^*_{Y \times g}((Y \times U) \cap \bigcup_j (B_{i'j} \times V'_{i'j}))} &&\text{by \cref{it:fib-baire-diff}} \\
\intertext{(with the $(Y \times g)$-fiberwise open basis $Y \times \@U$ for $Y \times Z$, where $\@U$ is any open basis for $Z$)}
&= \bigcap_{i'} \bigcap_{U \in \@U} \paren[\big]{(Y \times g)(Y \times U) => \bigcup_j \exists^*_{Y \times g}(B_{i'j} \times (U \cap V'_{i'j}))} \\
&= \bigcap_{i'} \bigcap_{U \in \@U} \paren[\big]{(Y \times g(U)) => \bigcup_j (B_{i'j} \times g(U \cap V'_{i'j}))} &&\text{by \cref{it:fib-baire-im} and \cref{it:fib-baire-bc}}.
\end{align*}
This is clearly of the desired form, where $i$ runs over all pairs $(i',U)$.
\end{proof}

Among standard Borel bundles of quasi-Polish spaces $f : X -> Y$, the best-behaved are those for which ``fiberwise nonemptiness of a Borel fiberwise open $U \subseteq X$'' can be detected in a uniformly Borel way.
These bundles may be characterized as follows:

\begin{proposition}
\label{thm:fib-bov-qpol}
Let $f : X -> Y$ be a standard Borel bundle of quasi-Polish spaces over a standard Borel space $Y$.
The following are equivalent:
\begin{enumerate}[label=(\roman*)]
\item \label{thm:fib-bov-qpol:overt}
For any Borel fiberwise open $U \subseteq X$, $f(U) \subseteq Y$ is Borel.
\item \label{thm:fib-bov-qpol:basis}
There exists a countable Borel fiberwise open basis $\@U \subseteq \@{BO}_f(X)$ such that for every $U \in \@U$, $f(U) \subseteq Y$ is Borel.
\item \label{thm:fib-bov-qpol:open}
There are compatible quasi-Polish topologies $\@O(X)$ and $\@O(Y)$ making $f$ continuous and open, such that $\@O(X)$ restricts to the fiberwise topology on $X$.
\end{enumerate}
\end{proposition}
We call the bundle \defn{Borel-overt} if these equivalent conditions hold (borrowing a term from constructive topology; see e.g., \cite{Sovert}).
\begin{proof}
Clearly \cref{thm:fib-bov-qpol:overt} and \cref{thm:fib-bov-qpol:open} each implies \cref{thm:fib-bov-qpol:basis}.

\cref{thm:fib-bov-qpol:basis}$\implies$\cref{thm:fib-bov-qpol:overt} follows from \cref{thm:fib-bor-qpol}\cref{thm:fib-bor-qpol:basis} and Frobenius reciprocity \cref{eq:fib-frob}.

\cref{thm:fib-bov-qpol:overt}$\implies$\cref{thm:fib-bov-qpol:open}:
By \cref{thm:fib-bor-qpol}, find some compatible topologies on $X, Y$ making $f$ continuous.
For each basic open $U \subseteq X$, $f(U) \subseteq Y$ is Borel; take a finer quasi-Polish topology on $Y$ (using \cref{it:qpol-dis}) making all of these sets open, and adjoin the preimages of all new open sets in $Y$ to $\@O(X)$.
If $X', Y'$ denote $X, Y$ with these new topologies, then $X' \cong X \times_Y Y'$ whence $X'$ is still quasi-Polish; and a basic open set in $X'$ is of the form $U \cap f^{-1}(V)$ where $U \in \@O(X)$ and $V \in \@O(Y')$, whence $f(U \cap f^{-1}(V)) = f(U) \cap V \subseteq Y'$ is open, showing that $f : X' -> Y'$ is open.
\end{proof}

\begin{remark}
Not every standard Borel bundle of quasi-Polish spaces $f : X -> Y$ is Borel-overt, e.g., if $f$ is a continuous map with non-Borel image (see \cite[14.2]{Kcdst}).
\end{remark}

Borel-overt bundles are the ones for which ``fiberwise Baire category is Borel'':

\begin{corollary}[of \cref{thm:fib-baire-borel}]
\label{thm:fib-bov-baire}
Let $f : X -> Y$ be a standard Borel-overt bundle of quasi-Polish spaces over a standard Borel $Y$.
Then
\begin{enumerate}[label=(\alph*)]
\item
(Borel fiberwise Baire property)
Every $A \in \@B(X)$ is $=^*_f$ to some $U_A \in \@{BO}_f(X)$.
\item
$\exists^*_f(\@B(X)) \subseteq \@B(Y)$.
\qed
\end{enumerate}
\end{corollary}

Finally in this subsection, we recall the Kuratowski--Ulam theorem, which has the following conceptual formulation in terms of bundles:

\begin{theorem}[Kuratowski--Ulam]
\label{thm:kuratowski-ulam}
Let $X --->{f} Y --->{g} Z$ be Borel maps between Borel spaces.
Suppose that $X$ is equipped with a $(g \circ f)$-fiberwise quasi-Polish topology, and $Y$ is equipped with a $g$-fiberwise quasi-Polish topology, such that both topologies are fiberwise compatible with the subspace Borel structures on each fiber, $f$ is fiberwise continuous and fiberwise open over $Z$, and $\exists^*_f, \exists^*_g$ preserve Borel sets.
(For example, $f, g$ could both be continuous open maps between quasi-Polish spaces, or more generally $g$ could be a standard Borel-overt bundle of such.)
\begin{equation*}
\begin{tikzcd}
X \rar["f"] \drar["g \circ f"'] & Y \dar["g"] \\
& Z
\end{tikzcd}
\end{equation*}
Then
\begin{equation*}
\exists^*_g \circ \exists^*_f = \exists^*_{g \circ f} : \@B(X) --> \@B(Z).
\end{equation*}
\end{theorem}

Since $\exists^*$ is defined fiberwise, it is equivalent to just consider the case where $Z = 1$ is a singleton, where the statement becomes: for a continuous open $f : X -> Y$ between quasi-Polish spaces, a Borel set $A \subseteq X$ is nonmeager in $X$ iff for nonmeagerly many $y \in Y$, $f^{-1}(y) \cap A$ is nonmeager in $f^{-1}(y)$.
The classical case is when $f$ is a product projection; see \cite[8.41]{Kcdst}.
It was pointed out in \cite[A.1]{MTrep} that essentially the same proof works for a continuous open $f$ between Polish spaces, and in \cite[7.6]{Cqpol} that the quasi-Polish case works just as well.
See also \cref{thm:loc-kuratowski-ulam} below.

\subsection{Lower powerspaces}
\label{sec:lowpow}

\begin{definition}
\label{def:lowpow}
For a topological space $X$, its \defn{lower powerspace} $\@F(X)$ is the space of closed subsets of $X$, equipped with the \defn{lower Vietoris topology} generated by the subbasic open sets
\begin{equation*}
\Dia U := \{F \in \@F(X) \mid F \cap U \ne \emptyset\} \quad \text{for $U \in \@O(X)$}.
\end{equation*}
\end{definition}

We record some elementary properties:
\begin{eqenum}

\item \label{it:lowpow-union}
$\Dia : \@O(X) -> \@O(\@F(X))$ preserves unions; thus, restricting to a basis for $\@O(X)$ still yields a subbasis for $\@O(\@F(X))$.
In particular, if $X$ is second-countable, then so is $\@F(X)$.

\item \label{it:lowpow-unit}
If $X$ is $T_0$, we have a continuous embedding $\down : X -> \@F(X)$, $x |-> \-{\{x\}}$ (where $\-{(-)}$ denotes closure), with $\down^{-1}(\Dia U) = U$.

\item \label{it:lowpow-funct}
A continuous map $f : X -> Y$ induces the continuous image-closure map $\-f : \@F(X) -> \@F(Y)$, $F |-> \-{f(F)}$, with $\-f^{-1}(\Dia V) = \Dia f^{-1}(V)$.
Thus, if $f$ is an embedding, then so is $\-f$.

\item \label{it:lowpow-prod}
For two spaces $X, Y$, the Cartesian product map $\times : \@F(X) \times \@F(Y) -> \@F(X \times Y)$ is continuous, with $\times^{-1}(\Dia(U \times V)) = \Dia U \times \Dia V$.

\end{eqenum}


\begin{proposition}[{\cite[Theorem~5]{dBKpowsp}}]
If $X$ is a quasi-Polish space, then so is $\@F(X)$.
\end{proposition}

We will also need the following generalization:

\begin{definition}
\label{def:fib-lowpow}
For a continuous map $f : X -> Y$, regarded as a bundle of topological spaces, its \defn{fiberwise lower powerspace} $\@F_Y(X) = \@F_f(X)$ is the space of pairs $(y,F)$ where $y \in Y$ and $F \in \@F(f^{-1}(y))$, regarded as a bundle via the first projection $\pi_1 : \@F_Y(X) -> Y$, and equipped with the topology generated by the subbasic open sets
\begin{align*}
\Dia_V U :={}& \{(y,F) \in \@F_Y(X) \mid y \in V \AND F \cap U \ne \emptyset\} \\
={}& \pi_1^{-1}(V) \cap \Dia_Y U \quad \text{for $U \in \@O(X)$ and $V \in \@O(Y)$}.
\end{align*}
In other words, $\@F_Y(X)$ is equipped with the topology induced by the embedding
\begin{equation}
\label{eq:fib-lowpow}
\begin{aligned}
\@F_Y(X) &`--> Y \times \@F(X) \\
(y,F) &|--> (y,\-F).
\end{aligned}
\end{equation}
By abuse of notation, we will often refer to an element of $\@F_Y(X)$ in the fiber over a fixed $y \in Y$ as just a closed set $F \in \@F(f^{-1}(y))$ in that fiber, rather than the pair $(y,F)$.
\end{definition}

\begin{proposition}[{\cite[2.2]{Cgpd}}]
\label{thm:fib-lowpow}
If $X$ above is completely Baire while $Y$ is $T_0$ first-countable, then the image of the above embedding is
\begin{align*}
\bigcap_{U \in \@O(X), V \in \@O(Y)} ((V \times \Dia U) <=> (Y \times \Dia(f^{-1}(V) \cap U))),
\end{align*}
where the intersection may be taken over any bases of $\@O(X), \@O(Y)$.
Thus if moreover $X, Y$ are second-countable, then the image is $\*\Pi^0_2$; and if $X, Y$ are quasi-Polish, then so is $\@F_Y(X)$.
\end{proposition}
\begin{proof}
Note first that the intersection remains the same if we only consider $U, V$ in some bases of $\@O(X), \@O(Y)$, since the expressions on both sides of the $<=>$ are ``bilinear'', i.e., preserve unions in $U, V$.
Thus, we henceforth allow them to be arbitrary open sets.

It is straightforward that the image is always contained in the above intersection, and that conversely, if $(y,F)$ belongs to the intersection, then $F \subseteq f^{-1}(\-{\{y\}})$ (using the $\Leftarrow$ set for $V := \neg \-{\{y\}}$ and $U := X$).
It remains to check that if $(y,F)$ belongs to the $=>$ sets, then $f^{-1}(y) \cap F \subseteq F$ is dense, which follows from Baire category since for each basic neighborhood $V$ of $y$, the $=>$ sets yield that $f^{-1}(V) \cap F \subseteq F$ is dense.
\end{proof}

Analogously to \cref{it:lowpow-union}--\cref{it:lowpow-prod}, we have
\begin{eqenum}

\item \label{it:fib-lowpow-union}
$(V, U) |-> \Dia_V U$ preserves unions in both variables (and binary intersections in $V$).

\item \label{it:fib-lowpow-unit}
If $X$ is $f$-fiberwise $T_0$ over $Y$, we have a continuous embedding $\down_Y := (f, \down) : X -> \@F_Y(X)$ over $Y$, with $\down^{-1}(\Dia_V U) = f^{-1}(V) \cap U$.

\item \label{it:fib-lowpow-funct}
For continuous maps $X --->{f} Y --->{g} Z$, we get a continuous map $\-f : \@F_Z(X) -> \@F_Z(Y)$ over $Z$, with $\-f^{-1}(\Dia_W V) = \Dia_W f^{-1}(V)$.

\item \label{it:fib-lowpow-prod}
For continuous maps $f : X -> Z$ and $g : Y -> Z$, the fiber product map $\times_Z : \@F_Z(X) \times_Z \@F_Z(Y) -> \@F_Z(X \times_Z Y)$ is continuous, with $\times^{-1}(\Dia_W (U \times_Z V)) = \Dia_W U \times_Z \Dia_W V$.

\end{eqenum}

\subsection{Linear quantifiers}
\label{sec:prelim-lin}

In this paper, our main interest in (fiberwise) lower powerspaces stems from the conceptual link they provide between ``quantifier-like maps'' on open and Borel sets, such as the Baire category quantifier $\exists^*_f$, and topological realization.
We now make this precise.
These ideas are more-or-less well-known in the point-free topology literature, for which see \cite{JTloc}, \cite{Vpowloc}, and \cref{sec:loc-lin}.
To keep this paper accessible, we give here a classical point-based treatment.

\begin{definition}
\label{def:lin}
Let $X, Y$ be topological spaces.
A \defn{linear map} $\phi : \@O(X) -> \@O(Y)$ is one preserving arbitrary unions.
We will only be concerned with such maps for second-countable $X, Y$, for which it is equivalent to require preservation of countable unions only.

Similarly, for Borel spaces $X, Y$, a \defn{linear map} $\phi : \@B(X) -> \@B(Y)$ is one preserving countable unions.
(For an explanation of this terminology, see \cref{rmk:lin} below.)
\end{definition}

\begin{proposition}
\label{thm:lin-lowpow}
For any topological spaces $X, Y$, we have a canonical bijection
\begin{align*}
\{\text{linear maps } \@O(X) -> \@O(Y)\} &\cong \{\text{continuous maps } Y -> \@F(X)\} \\
(U |-> h^{-1}(\Dia U)) &<-| h.
\end{align*}
\end{proposition}
\begin{proof}
An open set $U \in \@O(Y)$ is equivalently a continuous map $\chi_U : Y -> \#S$;
thus, a linear map $\@O(X) -> \@O(Y)$ is equivalently a map $\@O(X) \times Y -> \#S$ linear in the first variable and continuous in the second, which is equivalently a continuous map $Y -> \{\text{linear maps } \@O(X) -> \#S\} \subseteq \#S^{\@O(X)}$; and the space of linear maps $\@O(X) -> \#S$ with the pointwise convergence topology is homeomorphic to $\@F(X)$, where $F \in \@F(X)$ corresponds to the characteristic function of $\{U \in \@O(X) \mid F \in \Dia U\}$.
\end{proof}

\begin{definition}
\label{def:fib-lin}
Let $f : X -> Z$ and $g : Y -> Z$ be continuous maps, regarded as bundles of spaces.
An \defn{$\@O(Z)$-linear map} $\phi : \@O(X) -> \@O(Y)$ is a linear map which moreover obeys
\begin{equation*}
\phi(f^{-1}(W) \cap U) = g^{-1}(W) \cap \phi(U) \quad \forall W \in \@O(Z),\, U \in \@O(X).
\end{equation*}
We are particularly interested in the case where $g = 1_Y$ is an identity, where this becomes
\begin{equation*}
\phi(f^{-1}(V) \cap U) = V \cap \phi(U) \quad \forall V \in \@O(Y),\, U \in \@O(X);
\end{equation*}
in this case, we also call $\phi$ a \defn{$\@O(Y)$-linear quantifier}.
If moreover
\begin{equation*}
\phi(f^{-1}(V)) = V \quad \text{or equivalently} \quad \phi(X) = Y
\end{equation*}
(cf.\ \cref{it:fib-baire-surj}), we call $\phi$ a \defn{$\@O(Y)$-linear retraction} of $f^{-1} : \@O(Y) -> \@O(X)$.

Similarly, if $f : X -> Y$ is a Borel map between Borel spaces, we have the notion of a \defn{$\@B(Y)$-linear quantifier} or \defn{$\@B(Y)$-linear retraction} $\phi : \@B(X) -> \@B(Y)$, defined via the same equations.
\end{definition}

\begin{example}
\label{ex:fib-lin-baire}
If $f : X -> Y$ is a continuous open map between quasi-Polish spaces, then the image map $f : \@O(X) -> \@O(Y)$ is an $\@O(Y)$-linear quantifier, and a retraction iff $f$ is surjective.

Similarly, if $f : X -> Y$ is a standard Borel-overt bundle of quasi-Polish spaces over a standard Borel space, then $\exists^*_f : \@B(X) -> \@B(Y)$ is $\@B(Y)$-linear (by \cref{thm:fib-bov-baire}, \cref{it:fib-baire-union}, and \cref{it:fib-baire-frob}).
\end{example}

\begin{remark}
\label{rmk:lin}
The terminology ``linear'' comes from viewing a ($\sigma$-)topology as analogous to a commutative ring, where $\cap$ is ``multiplication'' and $\bigcup$ is ``addition''.
For a continuous map $f : X -> Y$, we then have a ``ring homomorphism'' $f^{-1} : \@O(Y) -> \@O(X)$, via which $\@O(X)$ may be viewed as an ``algebra over $\@O(Y)$'', hence in particular as a ``$\@O(Y)$-module''; an $\@O(Y)$-linear quantifier is then a ``module homomorphism''.
For more on this perspective, see \cite{JTloc}.
\end{remark}

\begin{proposition}
\label{thm:fib-lin-lowpow}
For continuous maps $f : X -> Z$ and $g : Y -> Z$ where $X$ is completely Baire and $Z$ is $T_0$ first-countable, the bijection of \cref{thm:lin-lowpow} induces a bijection
\begin{align*}
\{\text{$\@O(Z)$-linear maps } \@O(X) -> \@O(Y)\} &\cong \{\text{continuous maps }
Y -> \@F_Z(X)
\text{ over $Z$}\} \\
(U |-> h^{-1}(\Dia_Z U)) &<-| h : Y -> \@F_Z(X).
\end{align*}
Thus in particular, for $f : X -> Y = Z$ and $g = 1_Y$, we have
\begin{align*}
\{\text{$\@O(Y)$-linear quantifiers } \@O(X) -> \@O(Y)\} &\cong \{\text{continuous sections $Y -> \@F_Y(X)$ of $\pi_1$}\} \\
(U |-> h^{-1}(\Dia_Y U)) &<-| h,
\end{align*}
with the $\@O(Y)$-linear retractions of $f^{-1}$ corresponding to the continuous sections of $\pi_1$ picking a nonempty closed set from each fiber of $f$.
\end{proposition}
\begin{proof}
A continuous map $h : Y -> \@F_Z(X)$ is the same thing as a continuous map $(h_1, h_2) : Y -> Z \times \@F(X)$ which lands in the image of the embedding \cref{eq:fib-lowpow}; to say that $h$ is ``over $Z$'' means $h_1 = \pi_1 \circ h = g$.
So the right-hand side of the first bijection equivalently consists of continuous maps $h_2 : Y -> \@F(X)$ such that $(g, h_2) : Y -> Z \times \@F(X)$ lands in the image of \cref{eq:fib-lowpow}.
By \cref{thm:fib-lowpow}, this happens iff for each $U \in \@O(X)$ and $W \in \@O(Z)$, we have
\begin{equation*}
g^{-1}(W) \cap h_2^{-1}(\Dia U) =
(g,h_2)^{-1}(W \times \Dia U)
= (g,h_2)^{-1}(Z \times \Dia(f^{-1}(W) \cap U))
= h_2^{-1}(\Dia(f^{-1}(W) \cap U));
\end{equation*}
by \cref{thm:lin-lowpow}, such $h_2$ are in bijection with $\@O(Z)$-linear maps $\phi : \@O(X) -> \@O(Y)$.

If $g = 1_Y$, to say that $h$ always picks a nonempty set is to say that $h_2$ does, i.e., $h_2^{-1}(\Dia X) = Y$, which by \cref{thm:lin-lowpow} corresponds to $\phi(X) = Y$.
\end{proof}

We now give yet a third description of linear quantifiers:

\begin{definition}
\label{def:fib-lin-supp}
Let $f : X -> Y$ be a continuous map, $\phi : \@O(X) -> \@O(Y)$ be an $\@O(Y)$-linear quantifier which corresponds via \cref{thm:fib-lin-lowpow} to a continuous section $h : Y -> \@F_Y(X)$ of $\pi_1 : \@F_Y(X) -> Y$.
The \defn{support} of $\phi$ is
\begin{align*}
\supp(\phi)
:={}& \{x \in X \mid x \in h(f(x))\} \\
={}& \{x \in X \mid \forall U \in \@O(X)\, (x \in U \implies f(x) \in \phi(U))\} \\
={}& \bigcap_{U \in \@O(X)} (U => f^{-1}(\phi(U))).
\end{align*}
Clearly this is an $f$-fiberwise closed subset of $X$, which is $\*\Pi^0_2$ if $X$ is second-countable (since it suffices to intersect over basic open $U$).
\end{definition}

\begin{example}
\label{ex:fib-lin-open-supp}
For a continuous open map $f : X -> Y$ where $X$ is completely Baire and $Y$ is $T_0$ first-countable, the image quantifier $f : \@O(X) -> \@O(Y)$ of \cref{ex:fib-lin-baire} corresponds via the above bijection to $f^{-1} : Y -> \@F_Y(X)$: indeed, $f^{-1}(y) \in \Dia_Y U \iff f^{-1}(y) \cap U \ne \emptyset \iff y \in f(U)$, whence $(f^{-1})^{-1}(\Dia_Y U) = f(U)$.
Thus the support of $f : \@O(X) -> \@O(Y)$ is all of $X$.
\end{example}

\begin{proposition}
\label{thm:fib-lin-supp}
For a continuous map $f : X -> Y$ where $X$ is completely Baire and $Y$ is $T_0$ first-countable, we have a bijection
\begin{align*}
\{\text{$\@O(Y)$-linear quantifiers } \@O(X) -> \@O(Y)\} &\cong \{\text{$f$-fiberwise closed } F \subseteq X \text{ s.t.\ $f|F$ is open}\} \\
\phi &|-> \supp(\phi) \\
(U |-> f(F \cap U)) &<-| F,
\end{align*}
with the $\@O(Y)$-linear retractions of $f^{-1}$ corresponding to the $F$ such that $f(F) = Y$.
\end{proposition}
\begin{proof}
Let $\phi : \@O(X) -> \@O(Y)$ be an $\@O(Y)$-linear quantifier, corresponding via \cref{thm:fib-lin-lowpow} to $h : Y -> \@F_Y(X)$; we must show $\phi(U) = f(\supp(\phi) \cap U)$, which will in particular show that $f|\supp(\phi)$ is open, and that $\phi(X) = Y \iff f(\supp(\phi)) = Y$.
Indeed, we have $\supp(\phi) \cap U \subseteq (U => f^{-1}(\phi(U))) \cap U \subseteq f^{-1}(\phi(U))$ by definition of $\supp(\phi)$, whence $f(\supp(\phi) \cap U) \subseteq \phi(U)$.
Conversely, for any $y \in \phi(U) = h^{-1}(\Dia_Y U)$, we have $h(y) \in \Dia_Y U$; picking any $x \in h(y) \cap U$, we have $y = f(x)$, whence $x \in h(f(x))$, whence $x \in \supp(\phi) \cap U$, whence $y = f(x) \in f(\supp(\phi) \cap U)$.

Now let $F \subseteq X$ be $f$-fiberwise closed such that $f|F$ is open.
Note that the quantifier $\phi : U |-> f(F \cap U)$, which we are claiming is the preimage of $F$ under the bijection in question, corresponds via \cref{thm:fib-lin-lowpow} to the assignment of fibers $h : y |-> f^{-1}(y) \cap F$.
From the definition of $\supp(\phi)$ in terms of $h$, clearly $\supp(\phi) = F$; we need only check that $h : Y -> \@F_Y(X)$ is continuous.
Indeed, we have $h(y) \in \Dia_Y U \iff f^{-1}(y) \cap F \cap U \ne \emptyset \iff y \in f(F \cap U)$ which is open.
\end{proof}

Using this correspondence, we may extend \cref{it:qpol-openquot} to certain non-open quotient maps:

\begin{theorem}
\label{thm:qpol-linquot}
Let $f : X -> Y$ be a continuous map from a quasi-Polish space $X$ to a $T_0$ space $Y$, and suppose there exists an $\@O(Y)$-linear retraction $\phi : \@O(X) -> \@O(Y)$ of $f^{-1}$.
Then $f$ is surjective and $Y$ is quasi-Polish.
\end{theorem}
\begin{proof}
By \cref{thm:fib-lin-supp}, $Y$ is a continuous open $T_0$ quotient of $\supp(\phi)$; apply \cref{it:qpol-openquot}.
\end{proof}

\begin{remark}
Without assuming either that $X$ is quasi-Polish, or some separation axiom on $Y$ stronger than $T_0$, the existence of an $\@O(Y)$-linear retraction $\phi$ of $f^{-1}$ as above need not imply that $f$ is surjective.
For a counterexample, consider the inclusion of the subspace $X := (0,\infty)$ into $Y := (0,\infty]$ with the (Scott) topology consisting of the sets $(r,\infty]$ for each $r$.
\end{remark}

\begin{remark}
\label{rmk:fib-lin-borel}
The Borel versions of linear quantifiers from \cref{def:fib-lin} correspond to a standard descriptive set-theoretic notion.
Let $f : X -> Y$ be a Borel map between standard Borel spaces.
For a $\@B(Y)$-linear quantifier $\phi : \@B(X) -> \@B(Y)$, we may chase through the proof of \cref{thm:lin-lowpow} to get a map $h : Y -> \{\text{linear maps } \@B(X) -> \#S\}$; now such a linear map is the same thing as a $\sigma$-ideal in $\@B(X)$, so that $\phi$ corresponds to a family of $\sigma$-ideals $(\@I_y)_{y \in Y}$, from which $\phi$ is recovered via
\begin{equation*}
\phi(A) = \{y \in Y \mid h(y)(A) = 1\} = \{y \in Y \mid A \not\in \@I_y\}.
\end{equation*}
In lieu of ``continuity'' of $h$ as in \cref{thm:lin-lowpow}, we have the requirement that $A \in \@B(X) \implies \phi(A) \in \@B(Y)$, which is a weak (non-parametrized) form of the requirement that $(\@I_y)_y$ is a \defn{Borel on Borel} family; see \cite[18.5]{Kcdst}.
And $\@B(X)$-linearity means in particular that for each $y \in Y$,
\begin{gather*}
\{y\} \cap \phi(A) = \phi(f^{-1}(y) \cap A), \\
\shortintertext{i.e.,}
A \not\in \@I_y \iff f^{-1}(y) \cap A \not\in \@I_y,
\end{gather*}
which means that each $\@I_y$ is determined by its restriction to $\@B(f^{-1}(y))$.
So we have a bijection
\begin{align*}
\{\text{$\@B(Y)$-linear quant.\ } \@B(X) -> \@B(Y)\} &\cong \{\text{weakly Borel on Borel fam.\ of $\sigma$-ideals $\@I_y \subseteq \@B(f^{-1}(y))$}\}.
\end{align*}
The $\@B(Y)$-linear retractions of $f^{-1}$ correspond to the families of \emph{proper} $\sigma$-ideals $\@I_y \subsetneq \@B(f^{-1}(y))$.
\end{remark}

Note that if $X$ is equipped with a (nice) topology, then a linear quantifier on Borel sets may be restricted to one on open sets.
Using this observation, we have

\begin{corollary}[of \cref{thm:qpol-linquot}]
\label{thm:qpol-linquot-borel}
Let $f : X -> Y$ be a Borel map from a quasi-Polish space to a standard Borel space, and let $\phi : \@B(X) -> \@B(Y)$ be a $\@B(Y)$-linear retraction of $f^{-1}$.
Suppose that $f^{-1}(\phi(\@O(X))) \subseteq \@O(X)$, and that $\phi(\@O(X))$ separates points of $Y$.
Then $\@O(Y) := \phi(\@O(X))$ is a compatible quasi-Polish topology on $Y$ making $f$ continuous.
\end{corollary}
\begin{proof}
Note first that $\@O(Y)$ is indeed a topology: it is closed under countable unions because $\phi$ preserves countable unions, hence closed under arbitrary unions by second-countability of $\@O(X)$; it contains $Y = \phi(X)$; and it is closed under binary intersections, because by $\@B(Y)$-linearity,
\begin{equation}
\label{eq:qpol-linquot-borel}
\phi(U) \cap \phi(V) = \phi(f^{-1}(\phi(U)) \cap V)
\end{equation}
for $U, V \in \@O(X)$, and $f^{-1}(\phi(U)) \in f^{-1}(\phi(\@O(X))) \subseteq \@O(X)$ by assumption.
By this same assumption, with this topology on $Y$, $f$ is continuous; and $Y$ is $T_0$ since $\@O(Y)$ separates points.
Now apply \cref{thm:qpol-linquot}.
\end{proof}

\subsection{Baire category for coarser topologies}
\label{sec:fib-lin-baire}

For our main topological realization results below, we will be applying \cref{thm:qpol-linquot-borel} above to $\@B(Y)$-linear quantifiers $\phi$ given by the Baire category quantifier $\exists^*_f$ with respect to some $f$-fiberwise topology on $X$ which is \emph{coarser} than the restriction of the global topology on $X$.
We now specialize the machinery of the preceding subsection to this case.

\begin{definition}
\label{def:baire-subtop-supp}
Let $X$ be a set with two topologies $\@S \subseteq \@T$, such that $\@T$ is second-countable.
The \defn{$\@T$-support of $\@S$} will mean the smallest $\@T$-closed $\@S$-comeager set, i.e., the intersection of all such sets, which is still $\@S$-comeager by second-countability of $\@T$.
\end{definition}

\begin{lemma}
Let $X$ be a set with two topologies $\@S \subseteq \@T$, and let $Y \subseteq X$ be the $\@T$-support of $\@S$.
Suppose that every $\@T$-open set has the $\@S$-Baire property.
Then the inclusion $(Y,\@T|Y) -> (X,\@S)$ induces an isomorphism of Baire category algebras
\begin{equation*}
\@B(Y,\@T|Y)/\text{meager} \cong \@B(X,\@S)/\text{meager}.
\end{equation*}
In particular, an $\@S$-Borel set is $\@S$-(co)meager iff its restriction to $Y$ is $\@T$-(co)meager; and $(Y,\@T|Y)$ is a Baire space.
\end{lemma}
\begin{proof}
We claim that a $\@T$-closed $F \subseteq Y$ is $\@T|Y$-nowhere dense iff it is $\@S$-meager.
Indeed, if $F$ is $\@S$-meager, then so is $Y => F$ since $\neg Y$ is $\@S$-meager by definition of support, whence the $\@T$-interior of $Y => F$ is disjoint from $Y$ again by definition of support, which means the $\@T|Y$-interior of $F$ is empty, i.e., $F$ is $\@T|Y$-nowhere dense.
Conversely, if $F$ is $\@T|Y$-nowhere dense, then letting $U \subseteq X$ be $\@S$-open such that $F \triangle U$ is $\@S$-meager, we have that $U \setminus F$ is $\@T$-open and $\@S$-meager, hence disjoint from $Y$ by definition of support, i.e., $Y \cap U \subseteq F$, whence $Y \cap U = \emptyset$ since $F$ is $\@T|Y$-nowhere dense, whence $F = F \setminus U$ is $\@S$-meager.

It follows that for $\@S$-meager $A \subseteq X$, $A$ is contained in the union of countably many $\@S$-closed nowhere dense sets, whose intersections with $Y$ are $\@T|Y$-nowhere dense, whence $Y \cap A$ is $\@T|Y$-meager.
Thus the restriction map
\begin{equation*}
Y \cap (-) : \@B(X,\@S) --> \@B(Y,\@T|Y)
\end{equation*}
descends to a well-defined map between the category algebras, which is surjective by the assumption that every $\@T$-open set has the $\@S$-Baire property.
To check that it is injective: if $A \subseteq X$ such that $Y \cap A$ is $\@T|Y$-meager, then $Y \cap A$ is contained in the union of countably many $\@T|Y$-closed nowhere dense sets, which are $\@S$-meager, whence $Y \cap A$ is $\@S$-meager, whence so is $A \subseteq Y => (Y \cap A)$ again since $\neg Y$ is $\@S$-meager by definition of support.
This implies that $(Y,\@T|Y)$ is a Baire space, because a meager open set must be the restriction of some $\@S$-meager $\@T$-open $U \subseteq X$, which is disjoint from $Y$ by definition of support.
\end{proof}

\begin{remark}
Let $f : X -> Y$ be a continuous map from a second-countable space $X$ to an arbitrary topological space $Y$, and let $\@O_f(X)$ be another fiberwise topology on $X$, \emph{coarser} than the fiberwise restriction of $\@O(X)$.
By $\exists^*_f$, we mean the Baire category quantifier for the coarser fiberwise topology $\@O_f(X)$.
Suppose that $\exists^*_f(\@O(X)) \subseteq \@O(Y)$, so that $\exists^*_f : \@O(X) -> \@O(Y)$ is an $\@O(Y)$-linear quantifier.
Then the fiberwise $\@O(X)$-support of $\@O_f(X)$, as defined in \cref{def:baire-subtop-supp}, is the same as the support of $\exists^*_f$ as defined in \cref{def:fib-lin-supp}.
Indeed, the latter definition says precisely that $\supp(\exists^*_f)$ is fiberwise the complement of the union of all $\@O_f(X)$-meager $\@O(X)$-open sets.
\end{remark}

Applying the preceding lemma fiberwise to this situation, we get

\begin{corollary}
In the situation of the preceding remark, suppose furthermore that each $\@O(X)$-open set has the fiberwise $\@O_f(X)$-Baire property.
Then for every $\@O(X)$-Borel $A \subseteq X$, we have
\begin{equation*}
\exists^*_f(A) = \exists^*_{f|\supp(\exists^*_f)}(\supp(\exists^*_f) \cap A)
\end{equation*}
where $\exists^*$ on the right-hand side is with respect to the finer topology $\@O(X)|\supp(\exists^*_f)$; and this finer topology is $f$-fiberwise Baire.
\qed
\end{corollary}

\begin{theorem}
\label{thm:qpol-baire-subtop}
Let $f : X ->> Y$ be a Borel surjection from a quasi-Polish space to a standard Borel space, and let $\@O_f(X)$ be another fiberwise topology on $X$, \emph{coarser} than the fiberwise restriction of $\@O(X)$, and making $X$ into a standard Borel-overt bundle of quasi-Polish spaces over $Y$.
Let $\exists^*_f$ denote the Baire category quantifier for $\@O_f(X)$.
Suppose that $f^{-1}(\exists^*_f(\@O(X))) \subseteq \@O(X)$.
Then $\@O(Y) := \exists^*_f(\@O(X))$ is a compatible quasi-Polish topology on $Y$ making $f$ continuous.
\end{theorem}
\begin{proof}
By \cref{ex:fib-lin-baire}, $\exists^*_f : \@B(X) -> \@B(Y)$ is a $\@B(Y)$-linear retraction of $f^{-1}$.
Thus by \cref{thm:qpol-linquot-borel}, we need only check that $\@O(Y)$ separates points of $Y$.
For that, it is enough to check that every $B \in \@B(Y)$ belongs to the $\sigma$-algebra generated by $\@O(Y)$; since $\exists^*_f : \@B(X) -> \@B(Y)$ is surjective (being a retraction of $f^{-1}$), it is enough to check that $\exists^*_f$ lands in said $\sigma$-algebra.
By the preceding corollary, this is to say that $\exists^*_{f|\supp(\exists^*_f)}$ does; but $f|\supp(\exists^*_f)$ is a continuous open fiberwise Baire map $\supp(\exists^*_f) -> Y$ from the second-countable topology $\@O(X)|\supp(\exists^*_f)$ to the topology $\@O(Y)$, so the claim follows from \cref{thm:fib-baire-borel}.
\end{proof}

\section{Polish group actions}
\label{sec:grp}

\subsection{Generalities on group actions}
\label{sec:grp-prelim}

Let $G$ be a group acting on a set $X$.
Throughout this section, we always denote the group multiplication by
$\mu : G \times G -> G$,
and the action map by
$\alpha = \alpha_X : G \times X -> X$.
Note that a special case is the left translation action $\alpha = \mu$ of $G$ on itself.

\begin{definition}
\label{def:grp-twist}
As a bundle over $X$, any action $\alpha$ is isomorphic to the product projection $\pi_2$, via the following \defn{twist involution} which we denote by $\dagger$:
\begin{equation*}
\begin{tikzcd}[column sep=4em]
G \times X \drar["\alpha"'] \ar[rr,<->,"{(g,x)\mapsto(g,x)^\dagger:=(g^{-1},gx)}","\cong"'] &&
G \times X \dlar["\pi_2"] \\
& X
\end{tikzcd}
\end{equation*}
When we apply $\dagger$ to a concept (element, subset, etc.)\ in $G \times X$, we say that it \defn{twists} to the result.
\end{definition}

Now suppose $G$ is a topological group.

\begin{definition}
\label{def:grp-fibtop}
By the \defn{$\alpha$-fiberwise topology} $\@O_\alpha(G \times X)$ on $G \times X$, we will always mean that given by twisting the product $\pi_2$-fiberwise topology given by a copy of $G$ on each $\pi_2$-fiber.
\end{definition}

\begin{remark}
\label{rmk:grp-fibtop-basis}
For any $U \subseteq G$ and $G$-invariant $A \subseteq X$, we have $(U \times A)^\dagger = U^{-1} \times A$.
Since the sets $U^{-1} \times A$ for $U \in \@O(G)$ and arbitrary $A \subseteq X$ are $\pi_2$-fiberwise open, it follows that the sets $U \times A$ for $G$-invariant $A$ are $\alpha$-fiberwise open.
Moreover, for any open basis $\@U \subseteq \@O(G)$, $\@U \times X := \{U \times X \mid U \in \@U\}$ is an $\alpha$-fiberwise open basis (in the sense of \cref{def:fib-top}).

Thus, if $X$ is also equipped with a topology (regardless of whether $\alpha$ acts continuously), then the $\alpha$-fiberwise topology $\@O_\alpha(G \times X)$ is \emph{coarser} than the fiberwise restriction of the product topology $\@O(G) \otimes \@O(X)$.
(Recall here our notations for topologies from \cref{sec:top,sec:fib}.)
\end{remark}

\begin{definition}
\label{def:grp-orbtop}
For each orbit $G \cdot x \in X/G$, we have a surjection $(-) \cdot x : G ->> G \cdot x$, via which we may equip $G \cdot x$ with the quotient topology, which does not depend on the choice of basepoint $x$ within the orbit.
We call the family of these topologies on each orbit the \defn{orbitwise topology}, which is a fiberwise topology on the quotient map $\pi : X ->> X/G$; we thus identify it as usual with a global topology on $X$, denoted $\@O_G(X)$, consisting of the \defn{orbitwise open} sets (cf.\ \cref{def:fib-top}).
We also write ${\subseteq_G^*}, {=_G^*}$ to mean containment and equality mod orbitwise meager (cf.\ \cref{def:fib-baire}).

Note that an equivalent definition of $\@O_G(X)$ is given by
\begin{eqenum}
\item \label{it:grp-orbtop}
$\alpha : G \times X ->> X$ is a continuous open surjection from the $\pi_2$-fiberwise topology to the orbitwise topology.
In particular, $A \subseteq X$ is orbitwise open
iff $\alpha^{-1}(A) = (G \times A)^\dagger \subseteq G \times X$ is $\pi_2$-fiberwise open,
iff $G \times A$ is $\alpha$-fiberwise open.
\end{eqenum}
From this it follows that
\begin{eqenum}
\item \label{it:grp-orbtop-baire}
If $A \subseteq X$ is orbitwise meager, then $\alpha^{-1}(A) = (G \times A)^\dagger$ is $\pi_2$-fiberwise meager, i.e., $G \times A$ is $\alpha$-fiberwise meager (since continuous open maps are category-preserving).
\item \label{it:grp-orbtop-invar}
If $A \subseteq X$ is $G$-invariant, then $A$ is orbitwise open.
\item \label{it:grp-orbtop-top}
If $X$ is a topological $G$-space, then $\@O(X) \subseteq \@O_G(X)$.
\item \label{it:grp-orbtop-eqvar}
If $f : X -> Y$ is an equivariant map between $G$-spaces, then $f^{-1}(\@O_G(Y)) \subseteq \@O_G(X)$.
\end{eqenum}
\end{definition}

Next, consider the associativity axiom $(g \cdot h) \cdot x = g \cdot (h \cdot x)$.
This is expressed by commutativity of the following \defn{associativity square}, in which we let $\alpha_2$ be the common composite:
\begin{equation}
\label{diag:grp-assoc}
\begin{tikzcd}
G \times G \times X \dar["\mu \times X"'] \rar["G \times \alpha"] \drar["\alpha_2"] &
G \times X \dar["\alpha"] \\
G \times X \rar["\alpha"'] &
X
\end{tikzcd}
\end{equation}
Just as $\alpha$ is the twisted version of $\pi_2$, so can this entire square be seen as a twisted version of an obviously-commuting square of projections, via the following ``higher-order twists'':
\begin{equation}
\label{diag:grp-assoc-twist}
\begin{tikzcd}
G \times G \times X \dar["\mu \times X"'] \rar["G \times \alpha"] \drar["\alpha_2"] \ar[rrr,bend left=20,shift left=2,"{(g,h,x)\mapsto(g^{-1},h^{-1}g^{-1},ghx)}"] &
G \times X \dar["\alpha"] \rar[<->,"{(g,x)\mapsto(g^{-1},gx)}"] &[4em]
G \times X \dar["\pi_2"'] &
G \times G \times X \lar["\pi_{13}"'] \dar["\pi_{23}"] \dlar["\pi_3"'] \ar[lll,bend right=20,"{(g^{-1},gh^{-1},hx)\mapsfrom(g,h,x)}"{yshift=-2pt}]
\\
G \times X \rar["\alpha"'] \ar[rrr,bend right=13,<->,"{(g,x)\mapsto(g^{-1},gx)}"'] &
X \rar[equal] &
X &
G \times X \lar["\pi_2"]
\end{tikzcd}
\end{equation}
Here $\pi_{ij}$ is the projection onto the $i$th and $j$th coordinates.
Note that the right-hand square clearly exhibits $G \times G \times X$ as (an isomorphic copy of) the pullback of $\pi_2$ with itself.
Thus, the left-hand associativity square \cref{diag:grp-assoc} exhibits $G \times G \times X$ as the pullback of $\alpha$ with itself.

Since $\alpha$ is equipped with a fiberwise topology, the pullback square \cref{diag:grp-assoc} yields both a $(\mu \times X)$-fiberwise topology and a $(G \times \alpha)$-fiberwise topology on $G \times G \times X$ (recall here \cref{def:fib-pb}).
On the other hand, the ``higher-order twist'' of \cref{diag:grp-assoc-twist} also yields an \defn{$\alpha_2$-fiberwise topology} on $G \times G \times X$, corresponding to the $\pi_3$-fiberwise topology given by the product topology on $G \times G$.

\begin{lemma}
\label{thm:grp-assoc-fibtop}
The $\alpha_2$-fiberwise topology on $G \times G \times X$ restricts to both
the $(\mu \times X)$-fiberwise topology (on each fiber $(\mu \times X)^{-1}(g,x)$, which is contained in the corresponding fiber $\alpha_2^{-1}(\alpha(g,x))$), and also
the $(G \times \alpha)$-fiberwise topology.
\end{lemma}
\begin{proof}
The $\alpha_2$-fiberwise topology restricted to the fibers of $G \times \alpha$ corresponds, via the above diagram \cref{diag:grp-assoc-twist}, to the $\pi_3$-fiberwise topology restricted to the fibers of $\pi_{13}$; in other words, it is the result of transporting the $\pi_{13}$-fiberwise topology along (the top edge of) the topmost quadrilateral in \cref{diag:grp-assoc-twist}.
This quadrilateral may be factored as follows:
\begin{equation*}
\begin{tikzcd}[column sep=8.5em]
G \times G \times X \drar["G \times \alpha"'] \rar["{(g,h,x)\mapsto(g,h^{-1},hx)}"] &
G \times G \times X \dar["\pi_{13}"] \rar["{(g,h,x)\mapsto(g^{-1},hg^{-1},gx)}"] &
G \times G \times X \dar["\pi_{13}"]
\\
& G \times X \rar["{(g,x)\mapsto(g^{-1},gx)}"] &
G \times X
\end{tikzcd}
\end{equation*}
The right square is a homeomorphism of the $\pi_{13}$-fiberwise topology on $G \times G \times X$ (it moves the fiber over $(g,x)$ to that over $(g^{-1},gx)$, followed by the fiberwise homeomorphism $h |-> hg^{-1}$), which transports along the left triangle to the $(G \times \alpha)$-fiberwise topology (since the left triangle is $G \times {}$the triangle in \cref{def:grp-twist}), which is thus equal to the restriction of the $\alpha_2$-fiberwise topology.

Similarly, the diagram
\begin{equation*}
\begin{tikzcd}[column sep=8.5em]
G \times G \times X \drar["\mu \times X"'] \rar["{(g,h,x)\mapsto(g^{-1},gh,x)}"] &
G \times G \times X \dar["\pi_{23}"] \rar["{(g,h,x)\mapsto(g,h^{-1},hx)}"] &
G \times G \times X \dar["\pi_{23}"]
\\
& G \times X \rar["{(h,x)\mapsto(h^{-1},hx)}"] &
G \times X
\end{tikzcd}
\end{equation*}
shows that the $\pi_{23}$-fiberwise topology transports to both the $(\mu \times X)$-fiberwise topology and the restriction of the $\alpha_2$-fiberwise topology.
\end{proof}

Now consider an equivariant map $f : X -> Y$ between two actions of $G$.
Equivariance means commutativity of the left square below, which ``untwists'' to the right square:
\begin{equation}
\label{diag:grp-eqvar-twist}
\begin{tikzcd}
G \times X \dar["\alpha_X"'] \rar["G \times f"{inner sep=1pt}] \ar[rrr,bend left=20,shift left=1,<->,"\dagger_X"] &
G \times Y \dar["\alpha_Y"] \rar[<->,"\dagger_Y"] &[2em]
G \times Y \dar["\pi_2"'] &
G \times X \dar["\pi_2"] \lar["G \times f"'{inner sep=1pt}] \\
X \rar["f"] \ar[rrr,bend right=15,equal] &
Y \rar[equal] &
Y &
X \lar["f"']
\end{tikzcd}
\end{equation}
Note that when $f = \alpha_X : G \times X -> X$, where $G$ acts on $G \times X$ via left translation on the first coordinate, the left equivariance square becomes the associativity square \cref{diag:grp-assoc}, while the entire diagram here is similar but not identical to the diagram \cref{diag:grp-assoc-twist} (the difference being that here, we only ``untwist'' the vertical edges).
As before, it is clear that these squares are pullbacks, whence

\begin{lemma}
\label{thm:grp-eqvar-fibtop}
For a $G$-equivariant map $f : X -> Y$,
the $\alpha_X$-fiberwise topology is the pullback of the $\alpha_Y$-fiberwise topology along $f$.
\qed
\end{lemma}

\subsection{Vaught transforms}
\label{sec:grp-vaught}

We now suppose that $G$ is a Polish group and $X$ is a standard Borel $G$-space.
Then the $\alpha$-fiberwise topology of \cref{def:grp-fibtop} turns $G \times X$ into a standard Borel-overt bundle of quasi-Polish spaces over $X$, since $\pi_2$ clearly does.
By \cref{rmk:grp-fibtop-basis}, a countable Borel fiberwise open basis is given by $\@U \times X$ for any countable open basis $\@U$ for $G$.

In the rest of this section, we will use the letters $U, V, W$ to denote arbitrary Borel subsets of $G$ (not necessarily open), and $A, B, C$ to denote arbitrary Borel subsets of $X$.

\begin{definition}
\label{def:grp-vaught}
We will use the term \defn{Vaught transform} to refer loosely to several things.
Most generally, it will refer to the Baire category quantifier
\begin{equation*}
\exists^*_\alpha : \@B(G \times X) --> \@B(X)
\end{equation*}
for the $\alpha$-fiberwise topology (which lands in $\@B(X)$ by \cref{thm:fib-bov-baire}): for Borel $D \subseteq G \times X$,
\begin{equation*}
\begin{aligned}
\exists^*_\alpha(D)
= \exists^*_{\pi_2}(D^\dagger)
:={}& \{x \in X \mid \exists^* g \in G\, ((g^{-1},gx) \in D)\} \\
={}& \{x \in X \mid \exists^* g \in G\, ((g,g^{-1}x) \in D)\}.
\end{aligned}
\end{equation*}
For a Borel rectangle $D = U \times A$, we use the notation, also called the \defn{Vaught transform},
\begin{equation*}
U * A := \exists^*_\alpha(U \times A)
= \{x \in X \mid \exists^* g \in G\, (g \in U \AND x \in gA)\}.
\end{equation*}
In the original notation of Vaught (for nonempty open $U$), this would be denoted $A^{\triangle U^{-1}}$; see \cite[16.2]{Kcdst}, \cite[5.1.7]{BKgrp}.
However, we find this binary $*$ notation more convenient, to highlight the analogy with the product set under the action $U \cdot A = f(U \times A)$ (see \cref{it:grp-vaught-im} below).

Following \cref{def:ostar}, for $\@S \subseteq \@B(X)$ (e.g., a compatible quasi-Polish topology), we write%
\footnote{\label{ft:vaught-oast}%
The symbol $\oast$ in \cite{Cscc} denotes something entirely different.}
\begin{align*}
\@O(G) \oast \@S
:= \exists^*_\alpha(\@O(G) \otimes \@S)
= \{\bigcup_i (U_i * A_i) \mid U_i \in \@O(G) \AND A_i \in \@S\}
\subseteq \@B(X)
\end{align*}
where $i$ runs over a countable set.
The notation $\@B(G) \oast \@S$ has the analogous meaning.
\end{definition}

We now record several basic properties of Vaught transforms.
These are mostly well-known, at least for open $U$; see \cite[\S16.B]{Kcdst}, \cite[5.1.7]{BKgrp}.
Our main goal here is to make clear that these can all be derived ``algebraically'' from the corresponding properties of $\exists^*_\alpha$ from \cref{sec:fib}, hence generalize essentially verbatim to the groupoid and point-free contexts in \cref{sec:gpd-vaught,sec:loc-gpd}.

By \cref{it:fib-baire-im}, \cref{rmk:grp-fibtop-basis}, \cref{it:grp-orbtop}, and \cref{it:grp-orbtop-invar},
\begin{eqenum}
\item \label{it:grp-vaught-im}
$U * A \subseteq U \cdot A$, with equality if $U$ is open and $A$ is orbitwise open.
\item \label{it:grp-vaught-invar}
Thus, if $U$ is nonempty open and $A$ is $G$-invariant, then $U * A = A$.
\end{eqenum}
By \cref{it:fib-baire-union}, $*$ distributes over countable unions:
\begin{align}
\label{it:grp-vaught-union}
(\bigcup_i U_i) * (\bigcup_j A_j) &= \bigcup_{i,j} (U_i * A_j).
\end{align}
By \cref{it:fib-baire-diff} applied to $\alpha|(U \times X)$, for open $U \subseteq G$ and any countable open basis $\@W$ for $G$ (so $W \times X$ for $\@W \ni W \subseteq U$ form an $\alpha$-fiberwise open basis for $U \times X$),
\begin{align}
\label{it:grp-vaught-diff}
U * (A \setminus B) &= \bigcup_{\@W \ni W \subseteq U} ((W * A) \setminus (W * B)).
\end{align}
Thus for a quasi-Polish $G$-space $X$, by induction, or by \cref{thm:fib-baire-borel},
\begin{align}
\label{it:grp-vaught-borel}
\@O(G) \oast \*\Sigma^0_\xi(X) \subseteq \*\Sigma^0_\xi(X).
\end{align}

By \cref{rmk:grp-fibtop-basis} and \cref{it:grp-orbtop-baire},
\begin{equation}
\label{it:grp-vaught-meager}
\begin{aligned}
U \subseteq G \text{ meager}  &\implies  U * A \subseteq U * X = \emptyset, \\
A \subseteq X \text{ orbitwise meager}  &\implies  U * A \subseteq G * A = \emptyset.
\end{aligned}
\end{equation}
Since $*$ preserves union \cref{it:grp-vaught-union}, it follows that (recalling notation from \cref{eq:subset*}, \cref{def:fib-baire})
\begin{equation}
\label{it:grp-vaught-pettis}
U \subseteq^* V \AND A \subseteq^*_G B  \implies  U * A \subseteq V * B
\end{equation}
(indeed,
$U \subseteq^* V \implies
U * A
\subseteq ((U \setminus V) \cup V) * A
= ((U \setminus V) * A) \cup (V * A) = V * A$, and similarly for $A \subseteq^*_G B$).
We will refer to this law as \defn{Pettis's theorem (for actions)}; the original Pettis's theorem for groups follows by taking $\alpha = \mu$ and $U, A \in \@O(G)$.

Note that for any $U \in \@B(G)$, letting $U =^* U' \in \@O(G)$ by the Baire property, by \cref{it:grp-vaught-pettis} we have
\begin{align}
\label{it:grp-vaught-bp}
U * A = U' * A.
\end{align}
Thus considering $U * A$ for Borel $U$ is really no more general than taking only open $U$.

Next, consider applying the Baire category quantifier to the various edges of the associativity square \cref{diag:grp-assoc}.
Note that for $U, V \in \@B(G)$ and $A \in \@B(X)$, the quantifier $\exists^*_{G \times \alpha}$ maps $U \times V \times A |-> U \times (V * A)$; here we are implicitly using that $G \times G \times X$ is equipped with the $(G \times \alpha)$-fiberwise topology given by pulling back the $\alpha$-fiberwise topology as in the square \cref{diag:grp-assoc}.
Similar considerations apply to $\exists^*_{\mu \times G}$.
Thus, the \defn{Beck--Chevalley condition} \cref{it:fib-baire-bc} applied \emph{both ways} to the pullback square \cref{diag:grp-assoc} yields, for $D \in \@B(G \times X)$,
\begin{align}
\label{it:grp-vaught-bc}
\alpha^{-1}(\exists^*_\alpha(D))
&= \exists^*_{G \times \alpha}((\mu \times X)^{-1}(D))
= \exists^*_{\mu \times X}((G \times \alpha)^{-1}(D)),
\end{align}
which for a rectangle $D = U \times A$ means (using \cref{it:grp-vaught-union})
\begin{align}
\label{it:grp-vaught-bcr}
\alpha^{-1}(U * A)
&= \exists^*_{G \times \alpha}(\mu^{-1}(U) \times A)
= \bigcup_{VW \subseteq U} (V \times (W * A)) \quad \text{for $U, V, W \in \@O(G)$} \\
\label{it:grp-vaught-bcl}
&= \exists^*_{\mu \times X}(U \times \alpha^{-1}(A))
= U * \alpha^{-1}(A).
\end{align}
Note that an immediate consequence of \cref{it:grp-vaught-bcr} and \cref{it:grp-orbtop} is
\begin{eqenum}
\item \label{it:grp-vaught-orbtop}
For $U \in \@O(G)$ and $A \in \@B(X)$, $U * A$ is orbitwise open, i.e., 
$\@O(G) \oast \@B(X) \subseteq \@O_G(X)$.
\end{eqenum}


Also, from \cref{thm:grp-assoc-fibtop}, both the $(G \times \alpha)$- and $(\mu \times X)$-fiberwise topologies on $G \times G \times X$ are restrictions of the same $\alpha_2$-fiberwise topology; and both maps are fiberwise open over $X$ to the $\alpha$-fiberwise topology, since this is clearly true of the ``untwisted'' maps $\pi_{13}, \pi_{23}$ in the right square of \cref{diag:grp-assoc-twist}.
Thus by the \defn{Kuratowski--Ulam} \cref{thm:kuratowski-ulam}, we have
\begin{align}
\label{it:grp-vaught-ku}
\exists^*_\alpha \circ \exists^*_{\mu \times X} = \exists^*_\alpha \circ \exists^*_{G \times \alpha} : \@B(G \times G \times X) -> \@B(X),
\end{align}
which on rectangles $U \times V \times A$ says
\begin{align}
\label{it:grp-vaught-assoc}
(U * V) * A &= U * (V * A). \\
\intertext{For $U, V \subseteq G$ open, we furthermore have}
\label{it:grp-vaught-assoc-im}
= (U \cdot V) * A &= U \cdot (V * A)
\end{align}
by \cref{it:grp-vaught-im} and \cref{it:grp-vaught-orbtop}.
In particular, we get
\begin{eqenum}
\item \label{it:grp-vaught-subgrp}
If $U \subseteq G$ is an open subgroup, then $U * A$ is $U$-invariant.
\end{eqenum}

Finally, for a Borel $G$-equivariant map $f : X -> Y$ between two standard Borel $G$-spaces, from the Beck--Chevalley condition \cref{it:fib-baire-bc} and \cref{thm:grp-eqvar-fibtop}, generalizing \cref{it:grp-vaught-bcl}, we have for $B \in \@B(Y)$
\begin{align}
\label{it:grp-vaught-eqvar}
f^{-1}(U * B) = U * f^{-1}(B).
\end{align}
In fact this law characterizes equivariance of $f$; see \cref{thm:grp-eqvar-vaught} below.

\subsection{Topological realization}
\label{sec:grp-realiz}

The following may be regarded as the core result underlying the Becker--Kechris topological realization theorem, generalized to the quasi-Polish context (recall the notation $\oast$ from \cref{def:grp-vaught}):

\begin{theorem}
\label{thm:grp-cts}
Let $G$ be a Polish group, $X$ be a quasi-Polish space equipped with a Borel action $\alpha$ of $G$.
Then
\begin{enumerate}[label=(\alph*)]
\item
The action is continuous iff $\@O(G) \oast \@O(X) = \exists^*_\alpha(\@O(G \times X)) = \@O(X)$, i.e., the sets $U * A$ for $U \in \@O(G)$ and $A \in \@O(X)$ (are open and) form an open (sub)basis for $X$.
\item
If $\@O(G) \oast \@O(X) \subseteq \@O(X)$, then $\@O(G) \oast \@O(X)$ forms a coarser compatible quasi-Polish topology making the action continuous.
\end{enumerate}
\end{theorem}
\begin{proof}
If the action is continuous, then $\alpha : G \times X -> X$ is a continuous open surjection, whence $\exists^*_\alpha(\@O(G \times X)) = \alpha(\@O(G \times X)) = \@O(X)$.
Conversely, suppose $\@O(G) \oast \@O(X) = \exists^*_\alpha(\@O(G \times X)) \subseteq \@O(X)$.
Then
$\alpha^{-1}(\exists^*_\alpha(\@O(G \times X)))
= \alpha^{-1}(\@O(G) \oast \@O(X))
\subseteq \@O(G) \otimes (\@O(G) \oast \@O(X))
\subseteq \@O(G \times X)$
by the Beck--Chevalley condition \cref{it:grp-vaught-bcr}, whence $\alpha$ is continuous with respect to $\@O(G) \oast \@O(X)$, which by \cref{thm:qpol-baire-subtop} is a compatible quasi-Polish topology on $X$.
\end{proof}

We now state a consequence of \cref{thm:grp-cts} that subsumes most commonly used topological realization results as special cases.
Recall the notation $\otimes$ for product $\sigma$-topology from \cref{def:ostar}, and the notion of a \emph{compatible $\sigma$-topology} from \cref{def:qpol-compat}.

\begin{theorem}
\label{thm:grp-realiz}
Let $G$ be a Polish group, $X$ be a standard Borel $G$-space, $\@S \subseteq \@B(X)$ be a compatible $\sigma$-topology such that $\@O(G) \oast \@S \subseteq \@S$.
For any $A \in \@B(X)$, the following are equivalent:
\begin{enumerate}[label=(\roman*)]
\item \label{thm:grp-realiz:open}
$A$ is open in some quasi-Polish topology $\@O(X) \subseteq \@S$ making the action continuous.
\item \label{thm:grp-realiz:vaught}
$A \in \@B(G) \oast \@S = \@O(G) \oast \@S$, i.e., $A = \bigcup_i (U_i * A_i)$ for countably many $U_i \in \@B(G)$ (or $\@O(G)$) and $A_i \in \@S$.
\item \label{thm:grp-realiz:rect}
$\alpha^{-1}(A) \in \@B(G) \otimes \@S$, i.e., $\alpha^{-1}(A) = \bigcup_i (U_i \times A_i)$ for countably many $U_i \in \@B(G)$ and $A_i \in \@S$.
\item \label{thm:grp-realiz:rect-open}
$\alpha^{-1}(A) \in \@O(G) \otimes \@S$, i.e., $\alpha^{-1}(A) = \bigcup_i (U_i \times A_i)$ for countably many $U_i \in \@O(G)$ and $A_i \in \@S$.
\item \label{thm:grp-realiz:rect-vaught}
$\alpha^{-1}(A) \in \@O(G) \otimes (\@O(G) \oast \@S)$, i.e., $\alpha^{-1}(A) = \bigcup_i (U_i \times (V_i * A_i))$ for $U_i, V_i \in \@O(G)$, $A_i \in \@S$.
\item \label{thm:grp-realiz:trans}
Every $G$-translate $g \cdot A$ for $g \in G$ is in $\@S$, and there are countably many sets in $\@S$ generating all such translates under union.
\end{enumerate}
In particular, every $G$-invariant $A \in \@S$ obeys these conditions.
Moreover, countably many $A \in \@B(X)$ obeying these conditions may be made simultaneously open in some topology as in \cref{thm:grp-realiz:open}.
\end{theorem}
\begin{proof}
First, note that in \cref{thm:grp-realiz:vaught}, indeed $\@B(G) \oast \@S = \@O(G) \oast \@S$, by \cref{it:grp-vaught-bp}.
Also, every $G$-invariant $A \in \@S$ clearly obeys \cref{thm:grp-realiz:trans}, and also obeys \cref{thm:grp-realiz:vaught} by \cref{it:grp-vaught-invar}.

We have \cref{thm:grp-realiz:open}$\implies$\cref{thm:grp-realiz:rect-vaught}, since $\alpha^{-1}(A) \in \@O(G) \otimes \@O(X) \subseteq \@O(G) \otimes (\@O(G) \oast \@S)$ by \cref{thm:grp-cts}(a).

Clearly \cref{thm:grp-realiz:rect-vaught}$\implies$\cref{thm:grp-realiz:rect-open}$\implies$\cref{thm:grp-realiz:rect}.

We have \cref{thm:grp-realiz:rect}$\implies$\cref{thm:grp-realiz:vaught} since $A = \exists^*_\alpha(\alpha^{-1}(A))$ since $G \times X$ is $\alpha$-fiberwise nonmeager \cref{it:fib-baire-surj}.

Clearly \cref{thm:grp-realiz:rect}$\implies$\cref{thm:grp-realiz:trans}, while the converse holds by the Kunugui--Novikov \cref{thm:kunugui-novikov}.

It remains to show that countably many sets obeying \cref{thm:grp-realiz:vaught} can be made simultaneously open as in \cref{thm:grp-realiz:open}; for that, it suffices to show that countably many sets $U_i * A_i \in \@O(G) * \@S$ can be made simultaneously open.
Since $\@S$ is a compatible $\sigma$-topology, there is a quasi-Polish topology $\{A_i\}_i \subseteq \@T_0 \subseteq \@S$.
Given $\@T_n$, find a finer quasi-Polish topology $\@T_n \cup (\@O(G) * \@T_n) \subseteq \@T_{n+1} \subseteq \@S$ (using compatibility of $\@S$ and countable bases for $\@O(G), \@T_n$).
Then the join $\@T$ of the $\@T_n$ is quasi-Polish \cref{it:qpol-dis}\cref{it:qpol-dis:join}, and $\{A_i\}_i \cup (\@O(G) * \@T) \subseteq \@T \subseteq \@S$.
By \cref{thm:grp-cts}(b), the topology $\@O(G) \oast \@T$ works.
\end{proof}

\begin{corollary}[{topological realization of Borel actions; cf.\ \cite[5.2.1]{BKgrp}}]
\label{thm:grp-borel-realiz}
Let $G$ be a Polish group, $X$ be a standard Borel $G$-space.
Then there is a compatible quasi-Polish topology on $X$ making the action continuous.
\end{corollary}
\begin{proof}
By \cref{thm:grp-realiz} with $\@S := \@B(X)$ and the empty collection of $A$.
\end{proof}

\begin{corollary}[{change of topology; cf.\ \cite[5.1.8]{BKgrp}, \cite{Hbeckec}}]
\label{thm:grp-realiz-dis}
Let $G$ be a Polish group, $X$ be a quasi-Polish $G$-space.
Then for any countably many sets $A_i \in \*\Sigma^0_\xi(X)$, there is a finer quasi-Polish topology containing each $\@O(G) * A_i$ and contained in $\*\Sigma^0_\xi(X)$ for which the action is still continuous.
In particular, if $A_i$ is $G$-invariant, then $A_i$ itself can be made open in such a topology.
\end{corollary}
\begin{proof}
By \cref{thm:grp-realiz}\cref{thm:grp-realiz:vaught} with $\@S := \*\Sigma^0_\xi(X)$, which is compatible (\cref{ex:qpol-compat-borel}) and obeys $\@O(G) \oast \@S \subseteq \@S$ \cref{it:grp-vaught-borel}, applied to the given sets $A_i$ as well as a countable basis for $\@O(X)$.
\end{proof}

The next result is perhaps not usually viewed as a ``topological realization theorem''.
However, we can (somewhat perversely) regard it as such: it says that if the $G$-action preserves a preexisting topology, then we can find a topological realization ``compatible with'' that preexisting topology, i.e., equal to it.
The analogous result for groupoid actions looks less perverse; see \cref{thm:gpd-top-realiz}.

\begin{corollary}[{automatic continuity of actions; cf.\ \cite[9.16(i)]{Kcdst}}]
\label{thm:grp-top-realiz}
Let $G$ be a Polish group, $X$ be a quasi-Polish space with a Borel action of $G$.
If each $g \in G$ acts via a homeomorphism of $X$, then the action is jointly continuous.
\end{corollary}
\begin{proof}
Since each $g$ acts via homeomorphisms, for each $A \in \@O(X)$, $\alpha^{-1}(A) = (G \times A)^\dagger \subseteq G \times X$ is $\pi_1$-fiberwise open, hence by Kunugui--Novikov is in $\@B(G) \otimes \@O(X)$.
Thus for any $U \in \@B(G)$, also $(U \times A)^\dagger = (U^{-1} \times X) \cap (G \times A)^\dagger \in \@B(G) \otimes \@O(X)$, and so $U * A = \exists^*_{\pi_2}((U \times A)^\dagger) \in \@O(X)$ by Frobenius \cref{it:fib-baire-frob}.
Now apply \cref{thm:grp-realiz}\cref{thm:grp-realiz:rect} with $\@S := \@O(X)$ to a countable basis for $\@O(X)$.
\end{proof}

\begin{remark}
\Cref{thm:grp-top-realiz} includes as a special case Pettis's automatic continuity theorem for Borel homomorphisms between Polish groups $f : G -> H$ (via the left translation action $G \curvearrowright H$).
This is unsurprising, since the proof of \cref{thm:grp-realiz} uses Pettis's theorem via \cref{it:grp-vaught-bp}.
\end{remark}

It is worth explicitly restating \cref{thm:grp-realiz} in the special case $\@S = \@B(X)$, to characterize all Borel sets which are ``potentially open'' in \emph{some} topology making the action continuous.
For a standard Borel $G$-space $X$, recalling the \emph{orbitwise topology} $\@O_G(X)$ from \cref{def:grp-orbtop}, let
\begin{equation}
\label{def:grp-orbtop-borel}
\@{BO}_G(X) := \@B(X) \cap \@O_G(X)
\end{equation}
denote the Borel orbitwise open sets.
By \cref{it:grp-orbtop} and Kunugui--Novikov,
\begin{equation}
\label{it:grp-orbtop-borel}
A \in \@{BO}_G(X)  \iff  \alpha^{-1}(A) \in \@O(G) \otimes \@B(X).
\end{equation}

\begin{corollary}[``potentially open'' = ``orbitwise open'']
\label{thm:grp-realiz-borel}
Let $G$ be a Polish group, $X$ be a standard Borel $G$-space.
For any $A \in \@B(X)$, the following are equivalent:
\begin{enumerate}[label=(\roman*)]
\item \label{thm:grp-realiz-borel:open}
$A$ is open in some compatible quasi-Polish topology making the action continuous.
\item \label{thm:grp-realiz-borel:vaught}
$A \in \@B(G) \oast \@B(X) = \@O(G) \oast \@B(X)$, i.e., $A = \bigcup_i (U_i * A_i)$ for countably many Borel $U_i, A_i$.
\item \label{thm:grp-realiz-borel:rect}
$\alpha^{-1}(A) \in \@B(G) \otimes \@B(X)$, i.e., $\alpha^{-1}(A) \subseteq G \times X$ is a countable union of Borel rectangles.
\item \label{thm:grp-realiz-borel:orbtop}
$\alpha^{-1}(A) \in \@O(G) \otimes \@B(X)$, i.e., $A \in \@{BO}_G(X)$, i.e., $A$ is orbitwise open.
\item \label{thm:grp-realiz-borel:rect-vaught}
$\alpha^{-1}(A) \in \@O(G) \otimes (\@O(G) \oast \@B(X)) = \@O(G) \otimes \@{BO}_G(X)$.
\item \label{thm:grp-realiz-borel:trans}
There are countably many Borel sets in $X$ generating all $G$-translates $g \cdot A$ under union.
\end{enumerate}
In particular, any $G$-invariant $A$ works.
Moreover, any countably many $A \in \@B(X)$ obeying these conditions may be made simultaneously open in some topology as in \cref{thm:grp-realiz:open}; in other words, $\@{BO}_G(X)$ is the increasing union of all compatible quasi-Polish topologies making the action continuous.
\qed
\end{corollary}

\subsection{Equivariant maps}
\label{sec:grp-eqvar}

Let $G$ be a Polish group.
By \cite[2.6.1]{BKgrp}, $\@F(G)^\#N$ is a universal standard Borel $G$-space, i.e., every other standard Borel $G$-space admits a Borel equivariant embedding into $\@F(G)^\#N$.
Here $\@F(G)$ is the Effros Borel space of $G$, i.e., the underlying standard Borel space of the lower powerspace of \cref{sec:lowpow}, equipped with the left translation action of $G$.
Since we are working in the quasi-Polish setting, where we have available the lower Vietoris topology on $\@F(G)$, we point out that in fact,

\begin{proposition}
\label{thm:grp-lowpow-univ}
For any Polish group $G$, $\@F(G)^\#N$ is a universal $T_0$ second-countable $G$-space, i.e., every $T_0$ second-countable $G$-space admits an equivariant topological embedding into $\@F(G)^\#N$.
\end{proposition}
\begin{proof}
The proof is essentially a simpler version of \cite[2.6.1]{BKgrp}.
First, we verify that

\begin{lemma}
\label{thm:grp-lowpow-trans}
For any topological group $G$ and topological $G$-space $X$, the left translation action $G \times \@F(X) -> \@F(X)$ is continuous.
In particular, $\@F(G)$ is a topological $G$-space.
\end{lemma}
\begin{proof}
The left translation action takes $(g,F) |-> \-{\alpha_X(\-{\{g\}} \times F)}$, which is a composite of continuous maps \cref{it:lowpow-unit}, \cref{it:lowpow-prod}, \cref{it:lowpow-funct}.
\end{proof}

Now let $X$ be an arbitrary $T_0$ $G$-space.
For each $A \in \@O(X)$, the map $U |-> U \cdot A : \@O(G) -> \@O(X)$ preserves unions, hence corresponds by \cref{thm:lin-lowpow} to a continuous map $h_A : X -> \@F(G)$ such that $h_A^{-1}(\Dia U) = U \cdot A$.
Then for any basis $\@A \subseteq \@O(X)$, $\@O(G) \cdot \@A$ is still a basis (because $\alpha$ is open), and so $h_\@A := (h_A)_{A \in \@A} : X -> \@F(G)^\@A$ is an embedding.
From $G$-equivariance of $U |-> U \cdot A$, it is easily seen that each $h_A$ is $G$-equivariant, whence so is $h_\@A$.
\end{proof}

\begin{corollary}[of proof]
\label{thm:grp-eqvar-prod}
For any topological group $G$ and $T_0$ $G$-spaces $X, Y$, a continuous map $f : X -> Y$ is $G$-equivariant iff for every $U \in \@O(G)$ and $B \in \@O(Y)$, we have $f^{-1}(U \cdot B) = U \cdot f^{-1}(B)$.
\end{corollary}
\begin{proof}
For any $B \in \@O(Y)$, to say that
$f^{-1}(U \cdot B) = U \cdot f^{-1}(B)$
for all $U \in \@O(G)$ means
$f^{-1}(h_B^{-1}(\Dia U)) = h_{f^{-1}(B)}^{-1}(\Dia U)$
for all subbasic $\Dia U \in \@O(\@F(G))$, i.e., the triangle
\begin{equation*}
\begin{tikzcd}
X \drar["h_{f^{-1}(B)}"'] \rar["f"] & Y \dar["h_B"] \\
& \@F(G)
\end{tikzcd}
\end{equation*}
commutes.
If this holds for all $B$, then since $h_{\@O(Y)} : Y -> \@F(G)^{\@O(Y)}$ is an equivariant embedding, and $h_{f^{-1}(\@O(Y))}$ is equivariant, we get that $f$ is equivariant.
\end{proof}

\begin{corollary}
\label{thm:grp-eqvar-vaught}
For any Polish group and standard Borel $G$-spaces $X, Y$, a Borel map $f : X -> Y$ is $G$-equivariant iff for every $U \in \@O(G)$ and $B \in \@B(Y)$, we have $f^{-1}(U * B) = U * f^{-1}(B)$; and it is enough to require this only for orbitwise open $B \in \@{BO}_G(Y)$.
\end{corollary}
\begin{proof}
$\Longrightarrow$ is by \cref{it:grp-vaught-eqvar}.
Conversely, if $f^{-1}(U * B) = U * f^{-1}(B)$ for every $U \in \@O(G)$ and $B \in \@{BO}_G(Y)$, by \cref{thm:grp-realiz-borel}, we may find a compatible quasi-Polish topology on $Y$ making the action continuous, then find a compatible quasi-Polish topology on $X$ containing the preimage of each open set in $Y$ and making the action continuous, and then apply the preceding result.
\end{proof}

\begin{remark}
\label{rmk:grp-eqvar-vaught}
In \cref{thm:grp-eqvar-vaught}, it is in fact enough to have $f^{-1}(U * B) = U * f^{-1}(B)$ for some countable separating family of $B \in \@B(Y)$, by taking the maps $h_B : Y -> \@F(G)$ as above corresponding to $U |-> U * B$ (for some compatible quasi-Polish topology on $Y$ containing each $\@O(G) * B$), which are jointly injective into $\@F(G)^\#N$ by the proof of \cite[2.6.1]{BKgrp}.
\end{remark}

We also take this opportunity to point out the following universal property enjoyed by the topological realization constructed by \cref{thm:grp-cts}:

\begin{proposition}
\label{thm:grp-realiz-univ}
Let $G$ be a Polish group, $X$ be a quasi-Polish space equipped with a Borel action of $G$ such that $\@O(G) \oast \@O(X) \subseteq \@O(X)$.
Then any continuous equivariant map $f : X -> Y$ into another quasi-Polish $G$-space is in fact continuous from the coarser topology $\@O(G) \oast \@O(X)$.
In other words, letting $X'$ be $X$ with this coarser topology, the identity map $1_X : X -> X'$ is the universal continuous map from $X$ into a quasi-Polish $G$-space, hence exhibits $X'$ as the universal continuous ``completion'' of $X$:
\begin{equation*}
\begin{tikzcd}
X \drar["f"'] \rar["1_X"] & X' \dar[dashed,"f"] \\
& Y
\end{tikzcd}
\end{equation*}
\end{proposition}
\begin{proof}
$f^{-1}(\@O(Y)) = f^{-1}(\@O(G) \oast \@O(Y)) \subseteq \@O(G) \oast f^{-1}(\@O(Y)) \subseteq \@O(G) \oast \@O(X)$, by \cref{it:grp-vaught-eqvar}.
\end{proof}

\begin{remark}
In category-theoretic terms, this says that quasi-Polish $G$-spaces form a reflective subcategory of quasi-Polish spaces with Borel $G$-action satisfying $\@O(G) \oast \@O(X) \subseteq \@O(X)$.

(Note that there are many discontinuous such actions: for example, take $X := G$, and take any compatible quasi-Polish topology finer than the group topology; then $\@O(G) \oast \@O(X) \subseteq \@O(G) \oast \@B(G) = \@O(G) \oast \@O(G) = \@O(G) \subseteq \@O(X)$ by \cref{it:grp-vaught-bp}.)
\end{remark}

\subsection{Open relations}
\label{sec:grp-struct}

For a standard Borel $G$-space $X$, \cref{thm:grp-realiz-borel} (and more generally \cref{thm:grp-realiz}) give precise characterizations of which Borel sets $A \subseteq X$ can be made open in a topological realization.
We now consider the more general problem of which $n$-ary relations for $n \ge 2$ can be made open.
For $G$-invariant relations, this amounts to topological realization of \emph{standard Borel relational $G$-structures} in the sense of first-order logic.
For ease of notation, the following discussion will focus on $n = 2$.

\begin{remark}
\label{rmk:borel-binary-open}
Even in the absence of a group action, it is not true that every Borel binary relation $R \subseteq X \times Y$ can be made open in the product of some compatible quasi-Polish topologies on $X, Y$.
Indeed, this is clearly possible iff $R \in \@B(X) \otimes \@B(Y)$, i.e., $R$ is a countable union of Borel rectangles.

More generally, if we want $R$ to be open in the product of quasi-Polish topologies contained within compatible $\sigma$-topologies $\@S(X) \subseteq \@B(X)$ and $\@S(Y) \subseteq \@B(Y)$, then we need to require $R \in \@S(X) \otimes \@S(Y)$.
\end{remark}

For standard Borel $G$-spaces $X, Y$, we have the following analogous characterization.
We adopt the following convention: $\alpha_X \times \alpha_Y$ will denote the product action $G^2 \times X \times Y -> X \times Y$ of $G^2$ (i.e., we silently swap the middle two variables of the product map $G \times X \times G \times Y -> X \times Y$), while $\alpha_{X \times Y}$ will denote the diagonal action $G \times X \times Y -> X \times Y$.
Note that $\alpha_{X \times Y}$ factors through $\alpha_X \times \alpha_Y$, via the diagonal $G -> G^2$.

\begin{theorem}
\label{thm:grp-binary-realiz}
Let $G$ be a Polish group, $X, Y$ be standard Borel $G$-spaces, $\@S(X) \subseteq \@B(X)$ and $\@S(Y) \subseteq \@B(Y)$ be compatible $\sigma$-topologies such that $\@O(G) \oast \@S(X) \subseteq \@S(X)$ and $\@O(G) \oast \@S(Y) \subseteq \@S(Y)$.
For any $R \in \@B(X \times Y)$, the following are equivalent:
\begin{enumerate}[label=(\roman*)]
\item \label{thm:grp-binary-realiz:open}
$R \in \@O(X) \otimes \@O(Y)$ for some quasi-Polish topologies $\@O(X) \subseteq \@S(X)$ and $\@O(Y) \subseteq \@S(Y)$ making the actions on $X, Y$ continuous.
\item \label{thm:grp-binary-realiz:vaught2}
$R
\in \@B(G^2) \oast (\@S(X) \otimes \@S(Y))
= \@O(G^2) \oast (\@S(X) \otimes \@S(Y))
= (\@O(G) \oast \@S(X)) \otimes (\@O(G) \oast \@S(Y))$,
i.e.,
$R
= \bigcup_i (W_i * (A_i \times B_i))
= \bigcup_i ((U_i \times V_i) * (A_i \times B_i))
= \bigcup_i ((U_i * A_i) \times (V_i * B_i))$
for countably many $U_i, V_i \in \@O(G)$, $W_i \in \@B(G^2)$, $A_i \in \@S(X)$, and $B_i \in \@S(Y)$.
\item \label{thm:grp-binary-realiz:rect2}
$(\alpha_X \times \alpha_Y)^{-1}(R) \in \@B(G^2) \otimes \@S(X) \otimes \@S(Y)$, i.e., $(\alpha_X \times \alpha_Y)^{-1}(R) = \bigcup_i (W_i \times A_i \times B_i)$ for countably many $W_i \in \@B(G)$, $A_i \in \@S(X)$, and $B_i \in \@S(Y)$.
\item \label{thm:grp-binary-realiz:rect2-open}
$(\alpha_X \times \alpha_Y)^{-1}(R) \in \@O(G^2) \otimes \@S(X) \otimes \@S(Y)$, i.e., $(\alpha_X \times \alpha_Y)^{-1}(R) = \bigcup_i (U_i \times V_i \times A_i \times B_i)$ for countably many $U_i, V_i \in \@O(G)$, $A_i \in \@S(X)$, and $B_i \in \@S(Y)$.
\item \label{thm:grp-binary-realiz:rect2-vaught}
$(\alpha_X \times \alpha_Y)^{-1}(R) \in \@O(G^2) \otimes (\@O(G) \oast \@S(X)) \otimes (\@O(G) \oast \@S(Y))$.
\end{enumerate}
Furthermore, letting $(\@O(G) \otimes \@S(X))^*_{\pi_2} \subseteq \@B(G \times X)$ consist of all Borel $D \subseteq G \times X$ which are $=^*_{\pi_2}$ to a set in $\@O(G) \otimes \@S(X)$, the following are also equivalent to the above:
\begin{enumerate}[resume*]
\item \label{thm:grp-binary-realiz:vaught-vaught*}
$R \in \@B(G) \oast (\exists^*_{\alpha_X}((\@O(G) \otimes \@S(X))^*_{\pi_2}) \otimes \exists^*_{\alpha_Y}((\@O(G) \otimes \@S(Y))^*_{\pi_2}))$.
\item \label{thm:grp-binary-realiz:vaught-vaught}
$R \in \@O(G) \oast ((\@O(G) \oast \@S(X)) \otimes (\@O(G) \oast \@S(Y)))$.
\item \label{thm:grp-binary-realiz:rect-vaught*}
$\alpha_{X \times Y}^{-1}(R) \in \@B(G) \otimes \exists^*_{\alpha_X}((\@O(G) \otimes \@S(X))^*_{\pi_2}) \otimes \exists^*_{\alpha_Y}((\@O(G) \otimes \@S(Y))^*_{\pi_2})$.
\item \label{thm:grp-binary-realiz:rect-vaught}
$\alpha_{X \times Y}^{-1}(R) \in \@O(G) \otimes (\@O(G) \oast \@S(X)) \otimes (\@O(G) \oast \@S(Y))$.
\end{enumerate}
Moreover, countably many $R$ obeying these conditions may be made simultaneously open as in \cref{thm:grp-binary-realiz:open}, while also simultaneously making open countably many $A \subseteq X$ and $B \subseteq Y$ satisfying \cref{thm:grp-realiz}.
\end{theorem}
\begin{proof}
First, note that in \cref{thm:grp-binary-realiz:vaught2}, we indeed have
\begin{equation}
\label{thm:grp-binary-realiz:vaught2-rect}
(U \times V) * (A \times B) = (U * A) \times (V * B)
\end{equation}
for $U, V \in \@B(G)$, $A \in \@B(X)$, and $B \in \@B(Y)$, since
\begin{equation*}
\begin{aligned}
(U * A) \times (V * B)
&= \exists^*_{\alpha_X}(U \times A) \times \exists^*_{\alpha_Y}(V \times B) \\
&= \exists^*_{\alpha_X \times Y}(U \times A \times \exists^*_{\alpha_Y}(V \times B)) &&\text{by Beck--Chevalley \cref{it:fib-baire-bc}} \\
&= \exists^*_{\alpha_X \times Y}(\exists^*_{G \times X \times \alpha_Y}(U \times V \times A \times B)) &&\text{by Beck--Chevalley \cref{it:fib-baire-bc}} \\
&= \exists^*_{\alpha_X \times \alpha_Y}(U \times V \times A \times B) &&\text{by Kuratowski--Ulam \cref{thm:kuratowski-ulam}} \\
&= (U \times V) * (A \times B)
\end{aligned}
\end{equation*}
(where as indicated above, we silently switch the middle two factors in $G \times X \times G \times Y$), whence
$\@O(G^2) \oast (\@S(X) \otimes \@S(Y))
= (\@O(G) \oast \@S(X)) \otimes (\@O(G) \oast \@S(Y))$;
as before, this is also equal to $\@B(G^2) \oast (\@S(X) \otimes \@S(Y))$ by \cref{it:grp-vaught-bp}.

In particular, from the assumptions $\@O(G) \oast \@S(X) \subseteq \@S(X)$ and $\@O(G) \oast \@S(Y) \subseteq \@S(Y)$, we get $\@O(G^2) \oast (\@S(X) \otimes \@S(Y)) \subseteq \@S(X) \otimes \@S(Y)$.
Now \cref{thm:grp-binary-realiz:open} clearly implies \cref{thm:grp-realiz}\cref{thm:grp-realiz:open} for $G^2 \curvearrowright X \times Y$, which by \cref{thm:grp-realiz} is equivalent to each of \cref{thm:grp-binary-realiz:vaught2}--\cref{thm:grp-binary-realiz:rect2-vaught}.
And given countably many $R$ as in \cref{thm:grp-binary-realiz:vaught2}, as well as countably many $A \subseteq X$ and $B \subseteq Y$ as in \cref{thm:grp-realiz}, by that result, we may find topologies on $X, Y$ making each of these $A, B$ as well as the sets $U_i * A_i$ and $V_i * B_i$ in \cref{thm:grp-binary-realiz:vaught2} open, whence $R$ is open as in \cref{thm:grp-binary-realiz:open}.
This proves the equivalence of \cref{thm:grp-binary-realiz:open}--\cref{thm:grp-binary-realiz:rect2-vaught}.

Since $\alpha_{X \times Y}$ factors through $\alpha_X \times \alpha_Y$, we have \cref{thm:grp-binary-realiz:rect2-vaught}$\implies$\cref{thm:grp-binary-realiz:rect-vaught}.

Since $\@O(G) \otimes \@S(X) \subseteq (\@O(G) \otimes \@S(X))^*_{\pi_2}$, and similarly for $Y$,
\cref{thm:grp-binary-realiz:vaught-vaught}$\implies$\cref{thm:grp-binary-realiz:vaught-vaught*} and
\cref{thm:grp-binary-realiz:rect-vaught}$\implies$\cref{thm:grp-binary-realiz:rect-vaught*}.

As usual, we have
\cref{thm:grp-binary-realiz:rect-vaught}$\implies$\cref{thm:grp-binary-realiz:vaught-vaught} and
\cref{thm:grp-binary-realiz:rect-vaught*}$\implies$\cref{thm:grp-binary-realiz:vaught-vaught*}
because $R = \exists^*_{\alpha_{X \times Y}}(\alpha_{X \times Y}^{-1}(R))$.

Finally, we prove \cref{thm:grp-binary-realiz:vaught-vaught*}$\implies$\cref{thm:grp-binary-realiz:vaught2}.
Let $D \in (\@O(G) \otimes \@S(X))^*_{\pi_2}$
and $E \in (\@O(G) \otimes \@S(Y))^*_{\pi_2}$,
say $D =^*_{\pi_2} \bigcup_i (U_i \times A_i) \in \@O(G) \otimes \@S(X)$
and $E =^*_{\pi_2} \bigcup_j (V_j \times B_j) \in \@O(G) \otimes \@S(Y)$.
Note that
\begin{eqenum}
\item \label{it:grp-binary-realiz:rect-baire}
If $M \subseteq G$ is meager, then $M \times G, G \times M \subseteq G^2$ are orbitwise meager for the diagonal action $G \curvearrowright G^2$, since $\alpha_{G^2}^{-1}(M \times G) = \mu^{-1}(M) \times G \subseteq G^3$ is $\pi_{23}$-fiberwise homeomorphic via $(g,h,k) |-> (gh,h,k)$ to $M \times G^2$.
\end{eqenum}
It follows that
\begin{align*}
D \times E = (D \times G \times Y) \cap (G \times X \times E)
=^*_G \bigcup_{i,j} (U_i \times V_j \times A_i \times B_j) \in \@O(G^2) \otimes \@S(X) \otimes \@S(Y)
\end{align*}
(again silently swapping the middle two factors).
Thus by Pettis's theorem \cref{it:grp-vaught-pettis} and \cref{it:grp-vaught-bc}, for any $\@B(G) \ni W =^* W' \in \@O(G)$, we have
\begin{align*}
W * (D \times E)
&= \bigcup_{i,j} ((W' * (U_i \times V_j)) \times A_i \times B_j)
\in \@O(G^2) \otimes \@S(X) \otimes \@S(Y).
\end{align*}
But now
\begin{equation*}
\begin{aligned}
\MoveEqLeft
W * (\exists^*_{\alpha_X}(D) \times \exists^*_{\alpha_Y}(E)) \\
&= W * \exists^*_{\alpha_X \times \alpha_Y}(D \times E) &&\text{by Kuratowski--Ulam as in \cref{thm:grp-binary-realiz:vaught2-rect}} \\
&= \mathrlap{\exists^*_{\alpha_{X \times Y}}(W \times \exists^*_{\alpha_X \times \alpha_Y}(D \times E))} \\
&= \exists^*_{\alpha_X \times \alpha_Y}(\exists^*_{\alpha_{G^2 \times X \times Y}}(W \times D \times E)) &&\text{by Kuratowski--Ulam as in \cref{it:grp-vaught-ku}} \\
&= \exists^*_{\alpha_X \times \alpha_Y}(W * (D \times E)) \\
&\in \mathrlap{\exists^*_{\alpha_X \times \alpha_Y}(\@O(G^2) \otimes \@S(X) \otimes \@S(Y))
= \@O(G^2) \oast (\@S(X) \otimes \@S(Y))}
\end{aligned}
\end{equation*}
satisfies \cref{thm:grp-binary-realiz:vaught2}, as desired.
\end{proof}

\begin{remark}
In contrast to \cref{thm:grp-binary-realiz:vaught2}, \cref{thm:grp-binary-realiz:rect2-vaught}, and the situation with \cref{thm:grp-realiz}, we do not know if we can add the conditions ``$R \in \@O(G) \oast (\@S(X) \otimes \@S(Y))$'' and ``$\alpha_{X \times Y}^{-1}(R) \in \@O(G) \otimes \@S(X) \otimes \@S(Y)$''.
\end{remark}

For certain $\@S$, however, we can make such a simplification:

\begin{corollary}[characterization of ``potentially open'' relations]
\label{thm:grp-binary-realiz-borel}
Let $G$ be a Polish group, $X, Y$ be standard Borel $G$-spaces.
For any $R \in \@B(X \times Y)$, the following are equivalent:
\begin{enumerate}[label=(\roman*)]
\item \label{thm:grp-binary-realiz-borel:open}
$R \in \@O(X) \otimes \@O(Y)$ for some compatible quasi-Polish topologies $\@O(X), \@O(Y)$ making the actions on $X, Y$ continuous.
\item \label{thm:grp-binary-realiz-borel:vaught2}
$R
\in \@B(G^2) \oast (\@B(X) \otimes \@B(Y))
= \@O(G^2) \oast (\@B(X) \otimes \@B(Y))
= \@{BO}_G(X) \otimes \@{BO}_G(Y)$,
i.e., $R$ is a countable union of rectangles of Borel orbitwise open sets.
\item \label{thm:grp-binary-realiz-borel:rect2}
$(\alpha_X \times \alpha_Y)^{-1}(R) \in \@B(G^2) \otimes \@B(X) \otimes \@B(Y)$.
\item \label{thm:grp-binary-realiz-borel:rect2-open}
$(\alpha_X \times \alpha_Y)^{-1}(R) \in \@O(G^2) \otimes \@B(X) \otimes \@B(Y)$.
\item \label{thm:grp-binary-realiz-borel:rect2-vaught}
$(\alpha_X \times \alpha_Y)^{-1}(R) \in \@O(G^2) \otimes \@{BO}_G(X) \otimes \@{BO}_G(Y)$.
\item \label{thm:grp-binary-realiz-borel:vaught-vaught*}
$R \in \@B(G) \oast (\@B(X) \otimes \@B(Y))$.
\item \label{thm:grp-binary-realiz-borel:vaught-vaught}
$R \in \@O(G) \oast (\@{BO}_G(X) \otimes \@{BO}_G(Y))$.
\item \label{thm:grp-binary-realiz-borel:rect-vaught*}
$\alpha_{X \times Y}^{-1}(R) \in \@B(G) \otimes \@B(X) \otimes \@B(Y)$.
\item \label{thm:grp-binary-realiz-borel:rect-vaught}
$\alpha_{X \times Y}^{-1}(R) \in \@O(G) \otimes \@{BO}_G(X) \otimes \@{BO}_G(Y)$.
\end{enumerate}
Moreover, countably many $R$ obeying these conditions may be made simultaneously open as in \cref{thm:grp-binary-realiz-borel:open}, while also simultaneously making open countably many other $A \in \@{BO}_G(X)$ and $B \in \@{BO}_G(Y)$.
\end{corollary}
\begin{proof}
By \cref{thm:grp-binary-realiz} with $\@S(X) := \@B(X)$ and $\@S(Y) := \@B(Y)$, using in \cref{thm:grp-binary-realiz-borel:vaught2,thm:grp-binary-realiz-borel:rect2-vaught,thm:grp-binary-realiz-borel:vaught-vaught,thm:grp-binary-realiz-borel:rect-vaught} that $\@O(G) \oast \@B(X) = \@{BO}_G(X)$ consists of the orbitwise open sets by \cref{thm:grp-realiz-borel} and similarly for $Y$, and in \cref{thm:grp-binary-realiz-borel:vaught-vaught*,thm:grp-binary-realiz-borel:rect-vaught*} that $(\@O(G) \otimes \@B(X))^*_{\pi_2} = \@B(G \times X)$ by the fiberwise Baire property (\cref{thm:fib-bov-baire}) and similarly for $Y$.
\end{proof}

\begin{remark}
The most substantial implication in the preceding two results is \cref*{thm:grp-binary-realiz:vaught-vaught*}$\implies$\cref*{thm:grp-binary-realiz:open}; all other implications are relatively easy consequences.
\end{remark}

As noted before, these results straightforwardly generalize to $n$-ary relations for all $n \in \#N$.
Rather than state the most general result, which would be notationally rather messy, we will only state the generalized form of the main conditions of \cref{thm:grp-binary-realiz-borel}:

\begin{corollary}
\label{thm:grp-nary-realiz-borel}
Let $G$ be a Polish group, $X_i$ be countably many standard Borel $G$-spaces, and $R_k \subseteq X_{i_{k,1}} \times \dotsb \times X_{i_{k,n_k}}$ be countably many Borel relations of arities $n_k \in \#N$.
Then there are compatible quasi-Polish topologies on each $X_i$ making the actions continuous and making each $R_k$ open, iff each $R_k \in \@B(G) \oast (\@B(X_{i_{k,1}}) \otimes \dotsb \otimes \@B(X_{i_{k,n_k}}))$, i.e., $R_k$ can be written as a countable union of sets of the form $U * (A_1 \times \dotsb \times A_{n_k})$, where $U \in \@B(G)$ and $A_j \in \@B(X_{i_{k,j}})$.

In particular, this can be done if each $R_k$ is $G$-invariant and $R_k \in \@B(X_{i_{k,1}}) \otimes \dotsb \otimes \@B(X_{i_{k,n_k}})$, i.e., $R_k$ is a countable union of Borel rectangles.
In other words, a standard Borel structure over a (multi-sorted) countable relational first-order language equipped with a Borel action of $G$ via automorphisms can be made into a quasi-Polish $G$-structure with open relations, iff each relation is a countable union of Borel rectangles.
\qed
\end{corollary}

We also have the following generalization of \cref{thm:grp-realiz-dis}:

\begin{corollary}[change of topology for relations]
\label{thm:grp-nary-realiz-dis}
Let $G$ be a Polish group, $X_i$ be countably many quasi-Polish $G$-spaces, and $R_k \subseteq X_{i_{k,1}} \times \dotsb \times X_{i_{k,n_k}}$ be countably many relations of arities $n_k \in \#N$, such that each $R_k \in \@B(G) \oast (\*\Sigma^0_\xi(X_{i_{k,1}}) \otimes \dotsb \otimes \*\Sigma^0_\xi(X_{i_{k,n_k}}))$, i.e., $R_k$ can be written as a countable union of sets of the form $U * (A_1 \times \dotsb \times A_{n_k})$, where $U \in \@B(G)$ and $A_j \in \*\Sigma^0_\xi(X_{i_{k,j}})$.
Then there are finer quasi-Polish topologies on each $X_i$ contained in $\*\Sigma^0_\xi(X_i)$ for which the action is still continuous, such that each $R_k$ becomes open.
In particular, this can be achieved if each $R_k$ is $G$-invariant and a countable union of $\*\Sigma^0_\xi$ rectangles.
\end{corollary}
\begin{proof}
As before, for simplicity of notation we only consider the case of a binary relation $R = U * (A \times B) \in \@B(G) \oast (\*\Sigma^0_\xi(X) \otimes \*\Sigma^0_\xi(Y))$, which follows from \cref{thm:grp-binary-realiz}\cref{thm:grp-binary-realiz:vaught-vaught*}$\implies$\cref{thm:grp-binary-realiz:open}, using that $\alpha_X^{-1}(A) \in \*\Sigma^0_\xi(G \times X) \subseteq (\@O(G) \otimes \*\Sigma^0_\xi(X))^*_{\pi_2}$ by the fiberwise Baire property in the form of \cref{thm:fib-baire-borel} whence $A = \exists^*_{\alpha_X}(\alpha_X^{-1}(A)) \in \exists^*_{\alpha_X}((\@O(G) \otimes \*\Sigma^0_\xi(X))^*_{\pi_2})$, and similarly for $B$.
\end{proof}

\begin{remark}
\label{rmk:grp-top-nary-realiz}
\Cref{thm:grp-top-realiz} trivially ``generalizes'' to a ``topological realization'' result for quasi-Polish $G$-structures: if $X_i$ are countably many quasi-Polish spaces, equipped with countably many (invariant) relations $R_k$ of various arities which are open in the product topology, as well as a Borel action of $G$ via homeomorphisms which are also automorphisms of the $R_k$, then $((X_i)_i, (R_k)_k)$ is already a quasi-Polish $G$-structure with open relations (and jointly continuous action).

The analogous result for actions of groupoids on bundles of structures is less trivial, and is an application of the groupoid analog of \cref{thm:grp-binary-realiz}; see \cref{thm:gpd-top-nary-realiz}.
\end{remark}

\subsection{(Zero-dimensional) Polish realizations}
\label{sec:grp-0d-reg}

\begin{remark}
\label{rmk:grp-bk-0d-reg}
Thus far, we have focused on quasi-Polish topological realizations.
To get a Polish realization, one can combine \cref{thm:grp-cts} with the first part of \cite[Proof of 5.2.1]{BKgrp}, which ensures regularity of the resulting topology by iteratively constructing a countable Boolean algebra of basic open sets closed under $U * (-)$ for each basic open $U \subseteq G$.

Note that that part of their argument can be easily formalized in a point-free manner, in the spirit of our approach in this paper.
The last part of \cite[Proof of 5.2.1]{BKgrp}, showing that the topology is strong Choquet, can still be replaced by our argument in \cref{thm:grp-cts} which instead shows that the topology is quasi-Polish, via \cref{thm:qpol-baire-subtop} which ultimately reduces to \cref{it:qpol-openquot}.
\end{remark}

In the rest of this subsection, we show that a simple variation of the above argument recovers the finer change-of-topology results of Hjorth~\cite{Hbeckec} and Sami~\cite{Sbeckec}, generalized to quasi-Polish $G$-spaces, thereby strengthening \cref{thm:grp-realiz-dis} and several other preceding results in this paper to yield Polish topologies, while also clarifying the connection between \cite{Hbeckec}, \cite{Sbeckec} and \cite{BKgrp}.

Recall that a topological space is \defn{zero-dimensional} if it has an open basis consisting of clopen sets, and that a Polish group $G$ is \defn{non-Archimedean} if $1 \in G$ has a neighborhood basis of open (hence clopen) subgroups, the cosets of which then form an open basis for $G$.

\begin{lemma}
\label{thm:grp-realiz-dis-lat}
Let $G$ be a Polish group, $\@U \subseteq \@O(G)$ be a countable basis such that $\@U = \@U^{-1}$, $X$ be a standard Borel $G$-space, $\@A \subseteq \@B(X)$ be a countable sublattice forming a basis for a compatible quasi-Polish topology $\@T$ such that $\@U * \@A \subseteq \@A$ (hence $\@O(G) \oast \@T \subseteq \@T$).
Then the sets in $\@A$ together with their complements generate a compatible zero-dimensional Polish topology $\@T'$ such that $\@O(G) \oast \@T' \subseteq \@T'$, whence $\@O(G) \oast \@T'$ is a compatible topology making the action continuous; and this topology is Polish.
If moreover $\@U$ consists of cosets (so $G$ is non-Archimedean), then $\@O(G) \oast \@T'$ is zero-dimensional.
\end{lemma}
\begin{proof}
$\@T'$ is zero-dimensional since it is generated by some sets and their complements, and is Polish since it is the result of adjoining countably many closed sets to the quasi-Polish $\@T$.
For a basic $\@T'$-open $A \setminus B$, where $A, B \in \@A$, and $U \in \@U$, by \cref{it:grp-vaught-diff} we have
\begin{align}
\label{eq:grp-realiz-dis-0d}
U * (A \setminus B)
= \bigcup_{\@U \ni W \subseteq U} ((W * A) \setminus (W * B)) \in \@T';
\end{align}
thus $\@O(G) \oast \@T' \subseteq \@T'$.
So by \cref{thm:grp-cts}, $\@O(G) \oast \@T'$ is a compatible quasi-Polish topology making the action continuous.

First, suppose $\@U$ consists of cosets.
Then for $A, B \in \@A$ and $W \in \@U$, $C := (W * A) \setminus (W * B)$ is $WW^{-1}$-invariant since $W * A, W * B$ are \cref{it:grp-vaught-subgrp}, whence $C = WW^{-1} * C \in \@O(G) \oast \@T'$ and similarly $\neg C \in \@O(G) \oast \@T'$; by \cref{eq:grp-realiz-dis-0d}, such $C$ form a basis for $\@O(G) \oast \@T'$, which is hence zero-dimensional.

In the general case, we use:

\begin{lemma}
\label{lm:grp-realiz-dis-reg}
For any $V, W \in \@U$ and $A, B \in \@A$, we have
\begin{align*}
\-{V * ((W * A) \setminus (V^{-1}VV^{-1}VW * B))}
&\subseteq VV^{-1}VW * (A \setminus B)
\end{align*}
in the topology $\@O(G) \oast \@T'$, witnessed when $V \ne \emptyset$ by the following closed set sandwiched in between:
\begin{equation*}
\neg (V * (\neg (V^{-1}VW * A) \cup (V^{-1}VW * B))).
\end{equation*}
\end{lemma}
\begin{proof}
To prove that the left set above is contained in this set, it suffices (by \cref{it:grp-vaught-im}) to show
\begin{align*}
\emptyset &= (V \cdot ((W * A) \setminus (V^{-1}VV^{-1}VW * B))) \cap (V * (\neg (V^{-1}VW * A) \cup (V^{-1}VW * B))) \\
\iff \emptyset &= ((W * A) \setminus (V^{-1}VV^{-1}VW * B)) \cap V^{-1}(V * (\neg (V^{-1}VW * A) \cup (V^{-1}VW * B))) \\
&= (W * A) \cap \neg (V^{-1}VV^{-1}VW * B) \cap ((V^{-1}V * \neg (V^{-1}VW * A)) \cup (V^{-1}VV^{-1}VW * B))
\intertext{(using \cref{it:grp-vaught-assoc-im} in the last step).
The intersection with the second term of the union is clearly empty, as is the intersection with the first term of the union, since similarly to before we have}
\emptyset &= (W * A) \cap (V^{-1}V \cdot \neg (V^{-1}VW * A)) \\
\iff \emptyset &= V^{-1}V(W * A) \cap \neg (V^{-1}VW * A).
\end{align*}
To prove the second containment: putting $W' := V^{-1}VW$, we have
\begin{align*}
\MoveEqLeft
(VV^{-1}VW * (A \setminus B)) \cup (V * (\neg (W' * A) \cup (W' * B))) \\
&= (VW' * (A \setminus B)) \cup (V * (\neg (W' * A) \cup (W' * B))) \\
&= V * ((W' * (A \setminus B)) \cup \neg (W' * A) \cup (W' * B)) \\
&\supseteq V * ((W' * A) \cup \neg (W' * A)) \\
&= V * X = X \quad \text{(by \cref{it:grp-vaught-invar})}.
\qedhere
\end{align*}
\end{proof}

Now for a basic open $U * (A \setminus B) \in \@O(G) \oast \@T'$, where $U \in \@U$ and $A, B \in \@A$, we have
\begin{align*}
U * (A \setminus B)
&= \bigcup \set{V' * U' * (A \setminus B)}{U' \in \@U \AND 1 \in V' \in \@U \AND V'U' \subseteq U} \\
&= \bigcup \set*{V' * ((W' * A) \setminus (W' * B))}{
    W' \in \@U \AND
    1 \in V' \in \@U \AND
    V'W' \subseteq U
} \quad \text{by \cref{eq:grp-realiz-dis-0d}} \\
&= \origdisplaystyle\bigcup \set*{V * ((W * A) \setminus (W' * B))}{
    \begin{aligned}
    &W, W' \in \@U \AND
    1 \in V' \in \@U \AND
    V'W' \subseteq U \\
    &\AND \@U \ni V \subseteq V' \AND
    V^{-1}VV^{-1}VW \subseteq W'
    \end{aligned}
} \\
&\subseteq \bigcup \set*{V * ((W * A) \setminus (V^{-1}VV^{-1}VW * B))}{V, W \in \@U \AND VV^{-1}VW \subseteq U}
\end{align*}
which is a union of basic open sets whose closures are contained in $U * (A \setminus B)$ by \cref{lm:grp-realiz-dis-reg}.
Thus $\@O(G) \oast \@T'$ is a regular topology, hence being quasi-Polish, is Polish by \cref{it:qpol-pol}.
\end{proof}

\begin{theorem}[{cf.\ \cite[4.3]{Sbeckec}, \cite[2.2]{Hbeckec}}]
\label{thm:grp-realiz-dis-0d-reg}
Let $G$ be a (non-Archimedean) Polish group, $X$ be a quasi-Polish $G$-space.
Then for any countably many sets $A_i \in \*\Sigma^0_\xi(X)$, $\xi \ge 2$, there is a finer (zero-dimensional) Polish topology containing each $\@O(G) * A_i$ and contained in $\*\Sigma^0_\xi(X)$ for which the action is still continuous.
In particular, if $A_i$ is $G$-invariant, then $A_i$ itself can be made open in such a topology.
\end{theorem}
\begin{proof}
Let $\@U \subseteq \@O(G)$ be a countable basis with $\@U = \@U^{-1}$, consisting of cosets if $G$ is non-Archimedean.
Since every $\*\Sigma^0_\xi(X)$ set is a countable union of differences $A_i \setminus B_i$ where $A_i \in \@O(X)$ and $B_i \in \*\Sigma^0_{\zeta_i}(X)$ for some $\zeta_i < \xi$, it suffices to show that for countably many such $A_i, B_i$, we can find a topology of the specified kind containing each $\@O(G) * (A_i \setminus B_i)$.
For each $i$, find a finer quasi-Polish topology containing $A_i, B_i, \@O(X)$ and contained in $\*\Sigma^0_{\zeta_i}(X)$ by \cref{it:qpol-dis}; then the topology $\@T_0$ generated by all of these is still quasi-Polish by \cref{it:qpol-dis}\cref{it:qpol-dis:join}, and has a countable basis $\@A_0 \subseteq \bigcup_{\zeta < \xi} \*\Sigma^0_\zeta(X)$, which we may assume to be a lattice.
Given $\@T_n, \@A_n$, similarly find a finer quasi-Polish topology $\@T_{n+1}$ generated by a countable lattice $\@A_n \cup (\@U * \@A_n) \subseteq \@A_{n+1} \subseteq \bigcup_{\zeta < \xi} \*\Sigma^0_\zeta(X)$.
Let $\@T$ be the topology generated by $\@A := \bigcup_n \@A_n \subseteq \bigcup_{\zeta < \xi} \*\Sigma^0_\zeta(X)$, which obeys $\@U * \@A \subseteq \@A$ and each $A_i, B_i \in \@A$.
Then the topology $\@O(G) \oast \@T'$ given by the preceding lemma works.
\end{proof}

\begin{remark}
\label{rmk:grp-realiz-0d-reg}
By applying \cref{thm:grp-realiz-dis-0d-reg} after
\cref{thm:grp-borel-realiz},
\cref{thm:grp-realiz-borel},
\cref{thm:grp-binary-realiz-borel},
\cref{thm:grp-nary-realiz-borel}, or
\cref{thm:grp-nary-realiz-dis} (for $\xi \ge 2$),
we get that ``quasi-Polish'' may be replaced with ``(zero-dimensional) Polish'' in those results as well.
\end{remark}

\section{Open quasi-Polish groupoids}
\label{sec:gpd}

\subsection{Generalities on groupoids and their actions}
\label{sec:gpd-prelim}

\begin{definition}
\label{def:gpd}
A \defn{groupoid} $G = (G_0, G_1, \sigma, \tau, \mu, \iota, \nu)$ consists of two sets $G_0$ (\defn{objects}) and $G_1$ (\defn{morphisms}), together with five structure maps:
\begin{itemize}
\item
$\sigma, \tau : G_1 -> G_0$ (\defn{source} and \defn{target}); if $g \in G_1$ with $\sigma(g) = x$ and $\tau(g) = y$, then we also write $g : x -> y$ or $g \in G(x,y)$ where $G(x,y)$ is the \defn{hom-set} of all morphisms from $x$ to $y$;
\item
$\iota : G_0 -> G_1$ (\defn{identity}), also denoted $\iota(x) =: 1_x$;
\item
$\nu : G_1 -> G_1$ (\defn{inverse}), also denoted $\nu(g) =: g^{-1}$;
\item
$\mu : G_1 \times_{G_0} G_1 = G_1 \tensor[_\sigma]\times{_\tau} G_1 := \{(g,h) \in G_1 \times G_1 \mid \sigma(g) = \tau(h)\} -> G_1$ (\defn{multiplication} or \defn{composition}), also denoted $\mu(g,h) =: g \cdot h =: gh$;
\end{itemize}
satisfying the usual axioms such as associativity, $\sigma(gh) = \sigma(h)$, $gg^{-1} = 1_{\tau(g)}$, etc.

We will henceforth regard $G_1$ as the ``underlying set'' of a groupoid, which we thus also denote by $G := G_1$ (so we write $G \times_{G_0} G$, etc.); objects may be identified with identity morphisms.

For two sets of morphisms $U, V \subseteq G$ ($= G_1$), $UV = U \cdot V := \mu(U \times_{G_0} V)$ denotes the set of all composites of $g \in U$ and $h \in V$ which are defined.
If $U$ or $V$ is a singleton $\{g\}$, we omit the braces.
\end{definition}

As indicated above, when working with groupoids, one frequently encounters fiber products of the form $G \times_{G_0} (-)$ or $(-) \times_{G_0} G$, for which this usual notation is potentially ambiguous, due to the presence of two canonical maps $\sigma, \tau : G_1 -> G_0$.
We therefore adopt the following

\begin{convention}
\label{def:gpd-pb}
For a groupoid $G$, $G \times_{G_0} (-)$ always means with respect to $\sigma : G -> G_0$ in the first factor, while $(-) \times_{G_0} G$ always means with respect to $\tau : G -> G_0$ in the second factor.

When we need to specify the maps used in a fiber product, because they differ from the default, or just for emphasis, we use a notation such as $\tensor[_\sigma]\times{_\tau}$ used above.

For instance,
$G \times_{G_0} G \times_{G_0} G
= G \tensor[_\sigma]\times{_\tau} G \tensor[_\sigma]\times{_\tau} G
:= \{(g,h,k) \in G^3 \mid \sigma(g) = \tau(h) \AND \sigma(h) = \tau(k)\}$;
and for another bundle $p : X -> G_0$,
$G \times_{G_0} X
= G \tensor[_\sigma]\times{_p} X
:= \{(g,x) \in G \times X \mid \sigma(g) = p(x)\}$.
\end{convention}

\begin{definition}
A \defn{topological groupoid} $G$ is one whose $G_0, G_1$ are topological spaces and $\sigma, \tau, \mu, \iota, \nu$ are continuous maps.
Note that this implies that $G_0$ is a continuous retract of $G_1$, via $\iota$ and $\sigma$, so that we may continue to regard $G = G_1$ as the underlying space.
A \defn{quasi-Polish groupoid} $G$ is a topological groupoid such that $G = G_1$ is quasi-Polish
(whence so is its retract $G_0$ by \cref{it:qpol-pi02}).
Note that a one-object quasi-Polish groupoid is the same thing as a Polish group (since topological groups are uniformizable, hence regular).


A topological groupoid $G$ is \defn{open} if $\sigma : G -> G_0$ is an open map, or equivalently (as explained in the following definition) $\tau$ or $\mu$ are.
Note that topological groups are open as one-object groupoids.
\end{definition}

\begin{definition}
\label{def:gpd-action}
An \defn{action} of a groupoid $G$ on a bundle $p : X -> G_0$ is a map
$\alpha : G \times_{G_0} X -> X$ (recalling \cref{def:gpd-pb}),
taking each $(g,a)$ where $g : x -> y$ and $a \in p^{-1}(x)$ to $g \cdot a := \alpha(g,a) \in p^{-1}(y)$, i.e., $p(\alpha(g,a)) = \tau(g)$, and satisfying the usual associativity and identity axioms.

Examples are the \defn{trivial action} $\alpha = \tau$ of $G$ on $1_{G_0} : G_0 -> G_0$, the \defn{left translation action} $\alpha = \mu$ of $G$ on $\tau : G -> G_0$, and the \defn{right translation action} $\alpha(g,h) := hg^{-1}$ on $\sigma : G -> G_0$.

The \defn{twist involution} $\dagger$ of an action is defined as in \cref{def:grp-twist}:
\begin{equation}
\label{diag:gpd-twist}
\begin{tikzcd}[column sep=4em]
G \times_{G_0} X \drar["\alpha"'] \ar[rr,<->,"{(g,a)\mapsto(g,a)^\dagger:=(g^{-1},ga)}","\cong"'] &&
G \times_{G_0} X \dlar["\pi_2"] \\
& X
\end{tikzcd}
\end{equation}

If $G$ is a topological groupoid, a \defn{topological $G$-space} is an action $\alpha$ on a bundle $p : X -> G_0$ such that $X$ is a topological space and both $p, \alpha$ are continuous.
The notion of \defn{standard Borel $G$-space} (for a quasi-Polish
groupoid $G$) is defined analogously.

An \defn{open topological $G$-space} $X$ will mean one where $p : X -> G_0$ is an open map.
This is equivalent to saying that $\pi_1 : G \times_{G_0} X -> G$ is open, since $\pi_1, p$ are pullbacks of each other (along $\sigma, \iota$).
By considering the left and right translation actions of $G$ on $G$, we thus recover the aforementioned fact that $G$ is open iff $\sigma$ is open, iff $\pi_2$ is, iff its twist $\mu$ is, iff $\tau$ is.

For a quasi-Polish groupoid $G$, a \defn{standard Borel(-overt) $G$-bundle of quasi-Polish spaces} (over $G_0$) will mean a standard Borel $G$-space $p : X -> G_0$ which is also a standard Borel(-overt) bundle of quasi-Polish spaces (cf.\ \cref{def:fib-bor-qpol,thm:fib-bov-qpol}) and such that each morphism $g : x -> y \in G$ acts via a homeomorphism $p^{-1}(x) -> p^{-1}(y)$.
\end{definition}

\begin{remark}
\label{rmk:gpd-action-open}
Being an open topological $G$-space is \emph{not} equivalent to $\alpha$ being an open map.
Indeed, if $G$ is an open topological groupoid, then for any topological $G$-space $X$, $\alpha$ is an open map, being the twist of $\pi_2$ which is a pullback of $\sigma$.

Similarly, for an open quasi-Polish groupoid $G$ and standard Borel $G$-space $X$, the action map $\alpha$ is always Borel-overt with respect to the $\alpha$-fiberwise topology as defined below.
\end{remark}

\begin{definition}
\label{def:gpd-eqvar}
For two $G$-spaces $p : X -> G_0$ and $q : Y -> G_0$, a \defn{$G$-equivariant map} $f : X -> Y$ is one such that $p = q \circ f$ and $f(g \cdot a) = g \cdot f(a)$.
\end{definition}

For the rest of this subsection, fix an open topological groupoid $G$ acting on $p : X -> G_0$.

\begin{definition}[cf.\ \cref{def:grp-fibtop}]
\label{def:gpd-fibtop}
The \defn{$\alpha$-fiberwise topology} $\@O_\alpha(G \times_{G_0} X)$ is the twist of the $\pi_2$-fiberwise topology on $G \times_{G_0} X$ given by pulling back the $\sigma$-fiberwise topology on $G$.
\end{definition}

\begin{remark}[cf.\ \cref{rmk:grp-fibtop-basis}]
\label{rmk:gpd-fibtop-basis}
For any $U \subseteq G$ and $G$-invariant $A \subseteq X$, we have $(U \times_{G_0} A)^\dagger = U^{-1} \times_{G_0} A$.
Thus if $U \subseteq G$ is $\tau$-fiberwise open, then $U^{-1} \times_{G_0} A$ is $\pi_2$-fiberwise open, whence $U \times_{G_0} A$ is $\alpha$-fiberwise open.
And for any ($\tau$-fiberwise) open basis $\@U$ for $G$,
$\@U \times_{G_0} X := \{U \times_{G_0} X \mid U \in \@U\}$ is an $\alpha$-fiberwise open basis for $G \times_{G_0} X$.
\end{remark}

\begin{definition}[cf.\ \cref{def:grp-orbtop}]
\label{def:gpd-orbtop}
The \defn{orbit} of $a \in X$ is $G \cdot a = \sigma^{-1}(p(a)) \cdot a \subseteq X$.
We denote the set of orbits by $X/G$.

The \defn{orbitwise topology} $\@O_G(X)$ is the fiberwise topology on the quotient map $\pi : X ->> X/G$ given by the quotient topology on each orbit $G \cdot x$ induced by $(-) \cdot x : \sigma^{-1}(p(x)) ->> G \cdot x$.

As in \cref{def:grp-orbtop}, the following are easily seen:
\begin{eqenum}
\item \label{it:gpd-orbtop}
$\alpha : G \times_{G_0} X ->> X$ is a continuous open surjection from the $\pi_2$-fiberwise topology to the orbitwise topology.
In particular, $A \subseteq X$ is orbitwise open iff $\alpha^{-1}(A) = (G \times_{G_0} A)^\dagger \subseteq G \times_{G_0} X$ is $\pi_2$-fiberwise open, iff $G \times_{G_0} A$ is $\alpha$-fiberwise open.
\item \label{it:gpd-orbtop-baire}
If $A \subseteq X$ is orbitwise meager, then $G \times_{G_0} A$ is $\alpha$-fiberwise meager.
\item \label{it:gpd-orbtop-invar}
If $A \subseteq X$ is $G$-invariant, then $A$ is orbitwise open.
\item \label{it:gpd-orbtop-top}
If $X$ is a topological $G$-space, then $\@O(X) \subseteq \@O_G(X)$.
\item \label{it:gpd-orbtop-eqvar}
If $f : X -> Y$ is an equivariant map between $G$-spaces, then $f^{-1}(\@O_G(Y)) \subseteq \@O_G(X)$.
\end{eqenum}
\end{definition}

The associativity of $\alpha$ is expressed by the \defn{associativity square} (cf.\ \cref{diag:grp-assoc})
\begin{equation}
\label{diag:gpd-assoc}
\begin{tikzcd}
G \times_{G_0} G \times_{G_0} X \dar["\mu \times X"'] \rar["G \times \alpha"] \drar["\alpha_2"] &
G \times_{G_0} X \dar["\alpha"] \\
G \times_{G_0} X \rar["\alpha"'] &
X
\end{tikzcd}
\end{equation}
As in \cref{diag:grp-assoc-twist}, this can be seen as a twisted version of a square of projections (note the subscripts):
\begin{equation}
\label{diag:gpd-assoc-twist}
\begin{tikzcd}
G \tensor[_\sigma]\times{_\tau} G \tensor[_\sigma]\times{_p} X \dar["\mu \times X"'] \rar["G \times \alpha"] \drar["\alpha_2"] \ar[rrr,bend left=15,"{(g,h,a)\mapsto(g^{-1},h^{-1}g^{-1},gha)}"] &
G \times_{G_0} X \dar["\alpha"] \rar[<->,"\dagger"] &[4em]
G \times_{G_0} X \dar["\pi_2"'] &
G \tensor[_\sigma]\times{_\sigma} G \tensor[_\sigma]\times{_p} X \lar["\pi_{13}"'] \dar["\pi_{23}"] \dlar["\pi_3"'] \\
G \times_{G_0} X \rar["\alpha"'] \ar[rrr,bend right=10,<->,"\dagger"] &
X \rar[equal] &
X &
G \times_{G_0} X \lar["\pi_2"]
\end{tikzcd}
\end{equation}

\begin{lemma}[cf.\ \cref{thm:grp-assoc-fibtop}]
\label{thm:gpd-assoc-fibtop}
The pullback $(\mu \times X)$-fiberwise and $(G \times \alpha)$-fiberwise topologies on $G \times_{G_0} G \times_{G_0} X$ are both restrictions of a common $\alpha_2$-fiberwise topology, namely that twisting as above to the $\pi_3$-fiberwise topology on $G \tensor[_\sigma]\times{_\sigma} G \times_{G_0} X$.
\end{lemma}
\begin{proof}
Same as \cref{thm:grp-assoc-fibtop} (being careful to replace the various products appearing in that proof with fiber products with the correct subscripts).
\end{proof}

\begin{lemma}[cf.\ \cref{thm:grp-eqvar-fibtop}]
\label{thm:gpd-eqvar-fibtop}
For a $G$-equivariant map $f : X -> Y$, the $\alpha_X$-fiberwise topology is the pullback of the $\alpha_Y$-fiberwise topology along $f$.
\qed
\end{lemma}


\begin{definition}
\label{def:gpd-actgpd}
The \defn{action groupoid} $G \ltimes X$ of a groupoid action $G \curvearrowright X$ is given by
\begin{align*}
(G \ltimes X)_0 &:= X, \\
(G \ltimes X)_1 &:= G \times_{G_0} X, \\
\sigma_{G \ltimes X} &:= \pi_2 : G \times_{G_0} X -> X, \\
\tau_{G \ltimes X} &:= \alpha : G \times_{G_0} X -> X, \\
\mu_{G \ltimes X} &:= \mu_G \times X : (G \times_{G_0} X) \tensor[_{\pi_2}]\times{_\alpha} (G \times_{G_0} X) \cong G \times_{G_0} G \times_{G_0} X -> G \times_{G_0} X, \\
\iota_{G \ltimes X} &:= (\iota_G \circ p, 1_X) : X -> G \times_{G_0} X, \\
\nu_{G \ltimes X} &:= \dagger : G \times_{G_0} X -> G \times_{G_0} X.
\end{align*}
Intuitively, we think of a morphism $(g,x) \in (G \ltimes X)_1$ as ``$g : x -> gx$''; thus a hom-set $(G \ltimes X)(x,y)$ consists of ``all $g \in G$ such that $gx = y$''.


If $G$ is an (open) topological groupoid, and $X$ is a topological $G$-space, then $G \ltimes X$ is an (open) topological groupoid.
\end{definition}

\subsection{Vaught transforms}
\label{sec:gpd-vaught}

Let $G$ be an open quasi-Polish groupoid and $p : X -> G_0$ be a standard Borel $G$-space.
By \cref{rmk:gpd-action-open}, $\alpha : G \times_{G_0} X -> X$ with the $\alpha$-fiberwise topology of \cref{def:gpd-fibtop} is then a standard Borel-overt bundle of quasi-Polish spaces.
As in \cref{sec:grp-vaught}, we henceforth use $U, V, W$ to denote Borel subsets of $G$, and $A, B, C$ to denote Borel subsets of $X$.

\begin{definition}[cf.\ \cref{def:grp-vaught}]
\label{def:gpd-vaught}
The \defn{Vaught transform} will mean the category quantifier
\begin{equation*}
\begin{aligned}
\exists^*_\alpha : \@B(G \times_{G_0} X) -->{}& \@B(X) \\
D |-->{}& \{a \in X \mid \exists^* g \in \sigma^{-1}(p(a))\, ((g^{-1},ga) \in D)\} \\
={}& \{a \in X \mid \exists^* g \in \tau^{-1}(p(a))\, ((g,g^{-1}a) \in D)\}
\end{aligned}
\end{equation*}
for the $\alpha$-fiberwise topology, as well as its restriction to Borel rectangles
\begin{equation*}
U * A := \exists^*_\alpha(U \times_{G_0} A)
= \{a \in X \mid \exists^* g \in \tau^{-1}(p(a))\, (g \in U \AND a \in gA)\}.
\end{equation*}
(For open $U$, this was denoted $A^{\triangle U^{-1}}$ in \cite{Lgpd} and \cite{Cscc}.)

Per \cref{def:ostar}, for $\@S \subseteq \@B(X)$, we write
\begin{align*}
\@O(G) \oast \@S
:= \exists^*_\alpha(\@O(G) \otimes_{G_0} \@S)
= \{\bigcup_i (U_i * A_i) \mid U_i \in \@O(G) \AND A_i \in \@S\}
\subseteq \@B(X)
\end{align*}
(where $\otimes_{G_0}$ means all countable unions of $\times_{G_0}$ of two sets from $\@O(G), \@S$); similarly for $\@B(G) \oast \@S$.
\end{definition}

We now list some basic properties of groupoid Vaught transforms, corresponding to those in \cref{sec:grp-vaught} for group actions.
These are numbered so that, excepting the last item, (\ref*{sec:gpd-vaught}.$n$) here corresponds to (\ref*{sec:grp-vaught}.$n$) from \cref{sec:grp-vaught}, and is proved in exactly the same way (using the facts from \cref{sec:gpd-prelim} analogous to those previously used from \cref{sec:grp-prelim}).

\begin{eqenum}

\matcheq{it:grp-vaught-im}
\item \label{it:gpd-vaught-im}
$U * A \subseteq U \cdot A$, with equality if $U$ is $\tau$-fiberwise open and $A$ is orbitwise open.

\matcheq{it:grp-vaught-invar}
\item \label{it:gpd-vaught-invar}
Thus, if $A$ is $G$-invariant and $U$ is $\tau$-fiberwise open with $p(A) \subseteq \tau(U)$, then $U * A = A$.

\end{eqenum}
For countably many $U_i \in \@B(G)$, $A_j \in \@B(X)$,
\begin{align}
\matcheq{it:grp-vaught-union}
\label{it:gpd-vaught-union}
(\bigcup_i U_i) * (\bigcup_j A_j) = \bigcup_{i,j} (U_i * A_j).
\end{align}
For open $U \subseteq G$ and any countable open basis $\@W$ for $G$,
\begin{gather}
\matcheq{it:grp-vaught-diff}
\label{it:gpd-vaught-diff}
U * (A \setminus B) = \bigcup_{\@W \ni W \subseteq U} ((W * A) \setminus (W * B)).
\end{gather}
For a quasi-Polish $G$-space $X$,
\begin{gather}
\matcheq{it:grp-vaught-borel}
\label{it:gpd-vaught-borel}
\@O(G) \oast \*\Sigma^0_\xi(X) \subseteq \*\Sigma^0_\xi(X).
\end{gather}

The following laws form \defn{Pettis's theorem (for groupoid actions)}:
\begin{gather}
\matcheq{it:grp-vaught-meager}
\label{it:gpd-vaught-meager}
\left\{
\begin{aligned}
U \subseteq G \text{ $\tau$-fiberwise meager}  &\implies  U * A = \emptyset, \\
A \subseteq X \text{ orbitwise meager}  &\implies  U * A = \emptyset.
\end{aligned}
\right.
\\
\matcheq{it:grp-vaught-pettis}
\label{it:gpd-vaught-pettis}
U \subseteq^*_\tau V \AND A \subseteq^*_G B  \implies  U * A \subseteq V * B.
\end{gather}
Thus for $U \in \@B(G)$, letting $U =^*_\tau U' \in \@{BO}_\tau(G)$ by the fiberwise Baire property (\cref{thm:fib-bov-baire}),
\begin{align}
\matcheq{it:grp-vaught-bp}
\label{it:gpd-vaught-bp}
U * A = U' * A.
\end{align}

By the \defn{Beck--Chevalley condition}, for $D \in \@B(G \times X)$,
\begin{align}
\matcheq{it:grp-vaught-bc}
\label{it:gpd-vaught-bc}
\alpha^{-1}(\exists^*_\alpha(D))
&= \exists^*_{G \times \alpha}((\mu \times X)^{-1}(D))
= \exists^*_{\mu \times X}((G \times \alpha)^{-1}(D)),
\end{align}
which for a rectangle $D = U \times_{G_0} A$ means
\begin{align}
\matcheq{it:grp-vaught-bcr}
\label{it:gpd-vaught-bcr}
\alpha^{-1}(U * A)
&= \exists^*_{G \times \alpha}(\mu^{-1}(U) \times_{G_0} A)
= \bigcup_{VW \subseteq U} (V \times_{G_0} (W * A)) \quad \text{for $U, V, W \in \@O(G)$} \\
\matcheq{it:grp-vaught-bcl}
\label{it:gpd-vaught-bcl}
&= \exists^*_{\mu \times X}(U \times_{G_0} \alpha^{-1}(A))
= U * \alpha^{-1}(A).
\end{align}
Consequently,
\begin{align}
\matcheq{it:grp-vaught-orbtop}
\label{it:gpd-vaught-orbtop}
\@O(G) \oast \@B(X) \subseteq \@O_G(X).
\end{align}

By the \defn{Kuratowski--Ulam theorem},
\begin{align}
\matcheq{it:grp-vaught-ku}
\label{it:gpd-vaught-ku}
\exists^*_\alpha \circ \exists^*_{\mu \times X} = \exists^*_\alpha \circ \exists^*_{G \times \alpha} : \@B(G \times_{G_0} G \times_{G_0} X) -> \@B(X),
\end{align}
which for rectangles says
\begin{align}
\matcheq{it:grp-vaught-assoc}
\label{it:gpd-vaught-assoc}
(U * V) * A &= U * (V * A) \\
\matcheq{it:grp-vaught-assoc-im}
\label{it:gpd-vaught-assoc-im}
= (U \cdot V) * A &= U \cdot (V * A) \quad \mathrlap{\text{if $U$ $\tau$-fiberwise open and $V$ open.}} \quad
\end{align}

For a Borel $G$-equivariant $f : X -> Y$ between standard Borel $G$-spaces, for $B \in \@B(Y)$,
\begin{align}
\matcheq{it:grp-vaught-eqvar}
\label{it:gpd-vaught-eqvar}
f^{-1}(U * B) = U * f^{-1}(B).
\end{align}

Finally, we record an additional fact specific to the groupoid context: by \defn{Frobenius reciprocity}, for $B \in \@B(G_0)$, $U \in \@B(G)$, and $A \in \@B(X)$,
\begin{align}
\label{it:gpd-vaught-frob}
(\tau^{-1}(B) \cap U) * A = p^{-1}(B) \cap (U * A).
\end{align}
Indeed, we have
$(\tau^{-1}(B) \cap U) * A
= \exists^*_\alpha((\tau^{-1}(B) \cap U) \times_{G_0} A)
= \exists^*_\alpha(\alpha^{-1}(p^{-1}(B)) \cap (U \times_{G_0} A))
= p^{-1}(B) \cap \exists^*_\alpha(U \times_{G_0} A)
= p^{-1}(B) \cap (U * A)$,
using \cref{it:fib-baire-frob}.

\subsection{Topological realization}
\label{sec:gpd-realiz}

\begin{theorem}[cf.\ \cref{thm:grp-cts}]
\label{thm:gpd-cts}
Let $G$ be an open quasi-Polish groupoid, $X$ be a quasi-Polish space equipped with a continuous map $p : X -> G_0$ and a Borel action $\alpha$ of $G$.
Then
\begin{enumerate}[label=(\alph*)]
\item
The action is continuous iff $\@O(G) \oast \@O(X) = \exists^*_\alpha(\@O(G \times_{G_0} X)) = \@O(X)$, i.e., the sets $U * A$ for $U \in \@O(G)$ and $A \in \@O(X)$ (are open and) form an open (sub)basis for $X$.
\item
If $\@O(G) \oast \@O(X) \subseteq \@O(X)$, then $\@O(G) \oast \@O(X)$ forms a coarser compatible quasi-Polish topology for which $p$ is still continuous and also making the action continuous.
\end{enumerate}
\end{theorem}
\begin{proof}
The proof is identical to that of \cref{thm:grp-cts}, except that we must additionally point out why $p$ is still continuous with respect to $\@O(G) \oast \@O(X)$.
This is because $p : X -> G_0$ is an equivariant map to the trivial action (\cref{def:gpd-action}) which is continuous, whence by (a) and \cref{it:gpd-vaught-eqvar},
$p^{-1}(\@O(G_0))
= p^{-1}(\@O(G) \oast \@O(G_0))
\subseteq \@O(G) \oast p^{-1}(\@O(G_0))
\subseteq \@O(G) \oast \@O(X)$.
\end{proof}

In order to apply this core result in concrete situations, compared to \cref{sec:grp-realiz}, here there is an additional subtlety.
We could ask for certain subsets of $X$ to become globally open, as in \cref{thm:grp-realiz}.
Or, we might wish only to control the topology on the individual fibers of the bundle $p : X -> G_0$.
The latter is the best we can hope for in results depending on the interchangeability of $\@O(G), \@B(G)$ as in \cref{thm:grp-realiz}, due to the \emph{$\tau$-fiberwise} meager condition in Pettis's theorem \cref{it:gpd-vaught-meager}.
We will therefore state two generalizations of \cref{thm:grp-realiz}, beginning with the global one:

\begin{theorem}[cf.\ \cref{thm:grp-realiz}]
\label{thm:gpd-realiz}
Let $G$ be an open quasi-Polish groupoid, $p : X -> G_0$ be a standard Borel $G$-space, $\@S \subseteq \@B(X)$ be a compatible $\sigma$-topology such that $p^{-1}(\@O(G_0)), \@O(G) \oast \@S \subseteq \@S$.
For any $A \in \@B(X)$, the following are equivalent:
\begin{enumerate}[label=(\roman*)]
\item \label{thm:gpd-realiz:open}
$A$ is open in some quasi-Polish topology $\@O(X) \subseteq \@S$ making $p$ and the action continuous.
\item \label{thm:gpd-realiz:vaught}
$A \in \@O(G) \oast \@S$, i.e., $A = \bigcup_i (U_i * A_i)$ for countably many $U_i \in \@O(G)$ and $A_i \in \@S$.
\item \label{thm:gpd-realiz:rect-open}
$\alpha^{-1}(A) \in \@O(G) \otimes_{G_0} \@S$, i.e., $\alpha^{-1}(A) = \bigcup_i (U_i \times_{G_0} A_i)$ for countably many $U_i \in \@O(G)$, $A_i \in \@S$.
\item \label{thm:gpd-realiz:rect-vaught}
$\alpha^{-1}(A) \in \@O(G) \otimes_{G_0} (\@O(G) \oast \@S)$.
\end{enumerate}
In particular, every $G$-invariant $A \in \@S$ obeys these conditions.
Moreover, countably many $A \in \@B(X)$ obeying these conditions may be made simultaneously open in some topology as in \cref{thm:gpd-realiz:open}.
\end{theorem}
\begin{proof}
Same as \cref{thm:grp-realiz}, except when building the topology, we start with $p^{-1}(\@O(G_0))$.
\end{proof}

The following is the fiberwise version of the above:

\begin{corollary}[cf.\ \cref{thm:grp-realiz}]
\label{thm:gpd-fib-realiz}
Let $G$ be an open quasi-Polish groupoid, $p : X -> G_0$ be a standard Borel $G$-space, $\@S \subseteq \@B(X)$ be a compatible $\sigma$-topology such that $p^{-1}(\@O(G_0)), \@O(G) \oast \@S \subseteq \@S$.
For any $A \in \@B(X)$, the following are equivalent:
\begin{enumerate}[label=(\roman*)]
\item \label{thm:gpd-fib-realiz:fibopen}
$A$ is $p$-fiberwise open in some quasi-Polish topology $\@O(X) \subseteq \@S$ making $p, \alpha$ continuous.
\item \label{thm:gpd-fib-realiz:vaught}
$A \in \@B(G) \oast \@S = \@{BO}_\tau(G) \oast \@S$, i.e., $A = \bigcup_i (U_i * A_i)$ for countably many $U_i \in \@B(G)$ (or $U_i \in \@{BO}_\tau(G)$) and $A_i \in \@S$.
\item \label{thm:gpd-fib-realiz:rect}
$\alpha^{-1}(A) \in \@B(G) \otimes_{G_0} \@S$, i.e., $\alpha^{-1}(A) = \bigcup_i (U_i \times_{G_0} A_i)$ for countably many $U_i \in \@B(G)$, $A_i \in \@S$.
\item \label{thm:gpd-fib-realiz:rect-fibopen}
$\alpha^{-1}(A) \in \@{BO}_\tau(G) \otimes_{G_0} \@S$.
\item \label{thm:gpd-fib-realiz:rect-vaught}
$\alpha^{-1}(A) \in \@{BO}_\tau(G) \otimes_{G_0} (\@O(G) \oast \@S)$.
\item \label{thm:gpd-fib-realiz:trans}
Every $G$-translate $g \cdot A$ for $g \in G$ is a $p$-fiber of some set in $\@S$, and there are countably many sets in $\@S$ generating all such translates under union and restriction to $p$-fibers.
\end{enumerate}
In particular, every $G$-invariant $A \in \@S$ obeys these conditions.
Moreover, countably many $A \in \@B(X)$ obeying these may be made simultaneously $p$-fiberwise open as in \cref{thm:gpd-fib-realiz:fibopen}, while also making open countably many sets obeying \cref{thm:gpd-realiz}.
\end{corollary}
\begin{proof}
The proofs of
$\@B(G) \oast \@S = \@{BO}_\tau(G) \oast \@S$ in \cref{thm:gpd-fib-realiz:vaught}, of
\cref{thm:gpd-fib-realiz:rect-vaught}$\implies$%
\cref{thm:gpd-fib-realiz:rect-fibopen}$\implies$%
\cref{thm:gpd-fib-realiz:rect}$\implies$%
\cref{thm:gpd-fib-realiz:vaught}, and of
\cref{thm:gpd-fib-realiz:rect}$\iff$\cref{thm:gpd-fib-realiz:trans}
are the same as in \cref{thm:grp-realiz}.

\cref{thm:gpd-fib-realiz:fibopen}$\implies$\cref{thm:gpd-fib-realiz:rect-vaught}:
By Kunugui--Novikov, $A = \bigcup_i (p^{-1}(B_i) \cap A_i)$ where $B \in \@B(G_0)$ and $A_i \in \@O(X)$; and for each $i$,
$\alpha^{-1}(p^{-1}(B_i) \cap A_i)
= (\tau^{-1}(B_i) \times_{G_0} X) \cap \alpha^{-1}(A_i)
\in \@{BO}_\tau(G) \otimes_{G_0} (\@O(G) \oast \@S)$ by \cref{thm:gpd-realiz}.

Finally, to make countably many sets $U_i * A_i$ fiberwise open, where $U_i \in \@{BO}_\tau(G)$ and $A_i \in \@S$ as in \cref{thm:gpd-fib-realiz:vaught}: by Kunugui--Novikov, $U_i = \bigcup_j (\tau^{-1}(B_{ij}) \cap V_{ij})$ where $B_{ij} \in \@B(G_0)$ and $V_{ij} \in \@O(G)$, whence by \cref{it:gpd-vaught-frob},
\begin{equation*}
\textstyle
U_i * A_i = \bigcup_j (p^{-1}(B_{ij}) \cap (V_{ij} * A_i)).
\end{equation*}
By \cref{thm:gpd-realiz}, we may make the $V_{ij} * A_i$ (as well as countably many other sets satisfying the conditions in that theorem) open, thereby making the $U_i * A_i$ $p$-fiberwise open.
\end{proof}

\begin{corollary}[{topological realization of Borel actions; cf.\ \cref{thm:grp-borel-realiz}, \cite[4.1.1]{Lgpd}}]
\label{thm:gpd-borel-realiz}
Let $G$ be an open quasi-Polish groupoid, $p : X -> G_0$ be a standard Borel $G$-space.
Then there is a compatible quasi-Polish topology on $X$ making $p, \alpha$ continuous.
\end{corollary}
\begin{proof}
By \cref{thm:gpd-realiz} with $\@S := \@B(X)$ and the empty collection of $A$.
\end{proof}

The following strengthens \cite[4.2.1]{Lgpd}, in the quasi-Polish context, to be adapted precisely to each level of the Borel hierarchy:

\begin{corollary}[change of topology; cf.\ \cref{thm:grp-realiz-dis}]
\label{thm:gpd-realiz-dis}
Let $G$ be an open quasi-Polish groupoid, $p : X -> G_0$ be a quasi-Polish $G$-space.
Then for any countably many $A_i \in \*\Sigma^0_\xi(X)$, there is a finer quasi-Polish topology containing each $\@O(G) * A_i$ and contained in $\*\Sigma^0_\xi(X)$ for which the action is still continuous.
In particular, if $A_i$ is $G$-invariant, then $A_i$ itself can be made open in such a topology.
\end{corollary}
\begin{proof}
Same as \cref{thm:grp-realiz-dis}.
\end{proof}

\begin{remark}
Presumably, one could also generalize the finer arguments of \cref{sec:grp-0d-reg} to the groupoid setting, thereby strengthening the preceding result to yield a Polish topology as in \cite[4.2.1]{Lgpd} (perhaps under an additional assumption that $G$ is locally Polish or non-Archimedean).
However, we will not pursue this in this paper.
\end{remark}

Next, recall the notion of a \emph{standard Borel $G$-bundle of quasi-Polish spaces} from \cref{def:gpd-action}.

\begin{corollary}[topological realization of Borel $G$-bundles of spaces; cf.\ \cref{thm:grp-top-realiz}]
\label{thm:gpd-top-realiz}
Let $G$ be an open quasi-Polish groupoid, $p : X -> G_0$ be a standard Borel $G$-bundle of quasi-Polish spaces.
Then there is a compatible (global) quasi-Polish topology on $X$ making $p, \alpha$ continuous and restricting to the original $p$-fiberwise topology.
Moreover, such a topology may be taken to include any countably many sets in $\@O(G) \oast \@{BO}_p(X)$, in particular $G$-invariant Borel $p$-fiberwise open sets.
\end{corollary}
\begin{proof}
This is essentially the proof of \cref{thm:grp-top-realiz}, but taking care that everything belongs to the correct fibers.
Since each $g \in G$ acts via a homeomorphism $p^{-1}(\sigma(g)) -> p^{-1}(\tau(g))$, for each $p$-fiberwise open $A \subseteq X$, $\alpha^{-1}(A) \subseteq G \times_{G_0} X$ is $\pi_1$-fiberwise open, hence by Kunugui--Novikov (applied to $G \times_{G_0}{}$any countable Borel $p$-fiberwise open basis for $X$),
\begin{align*}
\yesnumber
\label{eq:gpd-top-realiz}
\alpha^{-1}(A)
= (G \times_{G_0} A)^\dagger
&= \bigcup_i (U_i \times_{G_0} A_i)
\in \@B(G) \otimes_{G_0} \@{BO}_p(X) \\
\intertext{where $U_i \in \@B(G)$ and $A_i \in \@{BO}_p(X)$.
Thus for any $U \in \@B(G)$,}
(U \times_{G_0} A)^\dagger
&= (U^{-1} \times_{G_0} X) \cap (G \times_{G_0} A)^\dagger
= \bigcup_i ((U^{-1} \cap U_i) \times_{G_0} A) \\
\shortintertext{and so}
U * A
= \exists^*_{\pi_2}((U \times A)^\dagger)
&= \bigcup_i (p^{-1}(\exists^*_\sigma(U^{-1} \cap U_i)) \cap A) \in \@{BO}_p(X)
\end{align*}
by the Beck--Chevalley condition \cref{it:fib-baire-bc} for the pullback $G \times_{G_0} X$.
So $\@O(G) \oast \@{BO}_p(X) \subseteq \@{BO}_p(X)$; and by \cref{thm:fib-bor-qpol}, $\@{BO}_p(X)$ is a compatible $\sigma$-topology.
Now apply \cref{thm:gpd-fib-realiz} with $\@S := \@{BO}_p(X)$; by \cref{eq:gpd-top-realiz}, we may make a countable Borel $p$-fiberwise open basis for $X$ fiberwise open, while also making countably many sets in $\@O(G) \oast \@{BO}_p(X)$ open.
\end{proof}

For a \emph{Borel-overt} $G$-bundle (recall again \cref{def:gpd-action}), we would naturally hope for a topological realization making $p : X -> G_0$ an open map, generalizing \cref{thm:fib-bov-qpol} for a bundle without an action.
To achieve this for a $G$-bundle, in general, we must \emph{refine the topology of the groupoid $G$}; this can be conveniently done using the action groupoid construction (\cref{def:gpd-actgpd}).

\begin{corollary}[topological realization of Borel-overt $G$-bundles]
\label{thm:gpd-bov-realiz}
Let $G$ be an open quasi-Polish groupoid, $p : X -> G_0$ be a standard Borel-overt $G$-bundle of quasi-Polish spaces.
Then there is a finer quasi-Polish topology on $G_0$, call the resulting space $\~G_0$, for which the trivial action $G \curvearrowright \~G_0$ is still continuous, together with a compatible (global) quasi-Polish topology on $X$ making $\alpha$ continuous and $p : X -> \~G_0$ continuous and open and restricting to the original $p$-fiberwise topology.
Thus, $\~G := G \ltimes \~G_0$ is $G$ with a finer open quasi-Polish groupoid topology for which $X$ becomes an open quasi-Polish $\~G$-space.
Moreover, $\@O(X)$ may be taken to include any countably many sets in $\@O(G) \oast \@{BO}_p(X)$, in particular $G$-invariant Borel $p$-fiberwise open sets.
\end{corollary}
\begin{proof}
Start by taking any topology $\@O(X)$ given by \cref{thm:gpd-top-realiz}.
Since $p$ was Borel-overt, $p(\@O(X)) \subseteq \@B(G_0)$.
Moreover, for each $A \in \@O(X)$, $p(A) \subseteq G_0$ is orbitwise open, which for the trivial action means by \cref{it:gpd-orbtop} that $\tau^{-1}(p(A))$ is $\sigma$-fiberwise open: indeed, we have
\begin{align*}
\tau^{-1}(p(A))
&= \pi_1(\alpha^{-1}(A)) \quad \text{(by the equivariance pullback square in \cref{def:gpd-eqvar} for $p$)} \\
&= \bigcup_{UB \subseteq A} \pi_1(U \times_{G_0} B) \quad \text{(where $U \in \@O(G)$, $B \in \@O(X)$)} \\
&= \bigcup_{UB \subseteq A} (U \cap \sigma^{-1}(p(B))).
\end{align*}
Thus by \cref{thm:gpd-realiz-borel} below, there is a finer quasi-Polish topology on $G_0$ containing $p(\@O(X))$, call the resulting space $\~G_0$, for which the trivial action $G \curvearrowright \~G_0$ is still continuous.
Now adjoin $p^{-1}(\@O(\~G_0))$ to the topology of $X$, i.e., replace $X$ with $\~G_0 \times_{G_0} X$.
\end{proof}

Finally, we restate \cref{thm:gpd-realiz,thm:gpd-fib-realiz} for $\@S := \@B(X)$.
As in \cref{def:grp-orbtop-borel}, let
\begin{align}
\@{BO}_G(X) := \@B(X) \cap \@O_G(X)
\end{align}
denote the Borel orbitwise open sets.
As in \cref{it:grp-orbtop-borel}, by \cref{it:gpd-orbtop} and Kunugui--Novikov,
\begin{align}
\label{it:gpd-orbtop-borel}
A \in \@{BO}_G(X)  \iff  \alpha^{-1}(A) \in \@O(G) \otimes_{G_0} \@B(X).
\end{align}

\begin{corollary}[of \cref{thm:gpd-realiz}; cf.\ \cref{thm:grp-realiz-borel}]
\label{thm:gpd-realiz-borel}
Let $G$ be an open quasi-Polish groupoid, $p : X -> G_0$ be a standard Borel $G$-space.
For any $A \in \@B(X)$, the following are equivalent:
\begin{enumerate}[label=(\roman*)]
\item \label{thm:gpd-realiz-borel:open}
$A$ is open in some compatible quasi-Polish topology $\@O(X)$ making $p, \alpha$ continuous.
\item \label{thm:gpd-realiz-borel:vaught}
$A \in \@O(G) \oast \@B(X)$, i.e., $A = \bigcup_i (U_i * A_i)$ for countably many $U_i \in \@O(G)$, $A_i \in \@B(X)$.
\item \label{thm:gpd-realiz-borel:rect-open}
$\alpha^{-1}(A) \in \@O(G) \otimes_{G_0} \@B(X)$, i.e., $A \in \@{BO}_G(X)$, i.e., $A$ is orbitwise open.
\item \label{thm:gpd-realiz-borel:rect-vaught}
$\alpha^{-1}(A) \in \@O(G) \otimes_{G_0} (\@O(G) \oast \@B(X)) = \@O(G) \otimes_{G_0} \@{BO}_G(X)$.
\end{enumerate}
In particular, any $G$-invariant $A$ works.
Moreover, countably many $A$ obeying these conditions may be made simultaneously open in some topology as in \cref{thm:gpd-realiz-borel:open}; in other words, $\@{BO}_G(X)$ is the increasing union of all compatible quasi-Polish topologies making $p, \alpha$ continuous.
\qed
\end{corollary}

For the characterization of ``potentially $p$-fiberwise open'' Borel sets, we need to consider the $p$-fiberwise restriction of the orbitwise topology:

\begin{definition}
For a topological groupoid $G$ and $G$-space $p : X -> G_0$, we say that $A \subseteq X$ is \defn{$p$-fiberwise ($G$-)orbitwise open} if for every $a \in X$, $A \cap p^{-1}(p(a)) \cap (G \cdot a)$ is open in the subspace topology on $p^{-1}(p(a)) \cap (G \cdot a)$ induced by the orbitwise topology on $G \cdot a$.

Note that $p^{-1}(p(a)) \cap (G \cdot a) = G(p(a),p(a)) \cdot a$, where $G(x,x)$ is the automorphism group of $x \in G_0$ (recall \cref{def:gpd}).
Thus, letting $\Aut(G) \subseteq G$ be the subgroupoid of automorphisms,
\begin{eqenum}
\item
$A$ $p$-fiberwise $G$-orbitwise open
$\iff$
$A$ $\Aut(G)$-orbitwise open
$\iff A \in \@O_{\Aut(G)}(X)$.
\end{eqenum}
(Warning: if $G$ is an open topological groupoid, $\Aut(G)$ may no longer be open.)

Note also that if $A$ is $p$-fiberwise orbitwise open, then more generally, $A \cap p^{-1}(y) \cap (G \cdot a)$ is open in the subspace topology on $p^{-1}(y) \cap (G \cdot a)$ for any $y \in G_0$, since if $b \in p^{-1}(y) \cap (G \cdot a)$ then $p^{-1}(y) \cap (G \cdot a) = p^{-1}(p(b)) \cap (G \cdot b)$.
Thus similarly to \cref{it:gpd-orbtop},
\begin{eqenum}
\item
$A \subseteq X$ $p$-fiberwise orbitwise open
$\iff \alpha^{-1}(A) \subseteq G \times_{G_0} X$ is $(\tau \circ \pi_1, \pi_2)$-fiberwise open.
\end{eqenum}
If $G$ is a quasi-Polish groupoid and $X$ is a standard Borel $G$-space, then by Kunugui--Novikov,
\begin{equation}
\begin{aligned}
A \in \@{BO}_{\Aut(G)}(X)
&\iff \alpha^{-1}(A) \in \@{BO}_\tau(G) \otimes_{G_0} \@B(X) \\
&\iff \alpha^{-1}(A) = \bigcup_i ((\tau^{-1}(B_i) \cap U_i) \times_{G_0} A_i)
\end{aligned}
\end{equation}
for countably many $B_i \in \@B(G_0)$, $U_i \in \@O(G)$, and $A_i \in \@B(X)$ (cf.\ the proof of \cref{thm:gpd-fib-realiz}).
\end{definition}

\begin{corollary}[of \cref{thm:gpd-fib-realiz}; cf.\ \cref{thm:grp-realiz-borel}]
\label{thm:gpd-fib-realiz-borel}
Let $G$ be an open quasi-Polish groupoid, $p : X -> G_0$ be a standard Borel $G$-space.
For any $A \in \@B(X)$, the following are equivalent:
\begin{enumerate}[label=(\roman*)]
\item \label{thm:gpd-fib-realiz-borel:fibopen}
$A$ is $p$-fiberwise open in some compatible quasi-Polish topology $\@O(X)$ making $p, \alpha$ continuous.
\item \label{thm:gpd-fib-realiz-borel:vaught}
$A \in \@B(G) \oast \@B(G) = \@{BO}_\tau(G) \oast \@B(X)$, i.e., $A = \bigcup_i (U_i * A_i)$ for countably many Borel $U_i, A_i$.
\item \label{thm:gpd-fib-realiz-borel:rect}
$\alpha^{-1}(A) \in \@B(G) \otimes_{G_0} \@B(X)$, i.e., $\alpha^{-1}(A) \subseteq G \times_{G_0} X$ is a countable union of Borel rectangles.
\item \label{thm:gpd-fib-realiz-borel:rect-fibopen}
$\alpha^{-1}(A) \in \@{BO}_\tau(G) \otimes_{G_0} \@B(X)$, i.e., $A \in \@{BO}_{\Aut(G)}(X)$, i.e., $A$ is $p$-fiberwise orbitwise open.
\item \label{thm:gpd-fib-realiz-borel:rect-vaught}
$\alpha^{-1}(A) \in \@{BO}_\tau(G) \otimes_{G_0} (\@O(G) \oast \@B(X)) = \@{BO}_\tau(G) \otimes_{G_0} \@{BO}_G(X)$.
\item \label{thm:gpd-fib-realiz-borel:trans}
There are countably many Borel sets in $X$ generating all $G$-translates $g \cdot A$ under union and restriction to $p$-fibers.
\end{enumerate}
Moreover, countably many $A \in \@S$ obeying these may be made simultaneously $p$-fiberwise open as in \cref{thm:gpd-fib-realiz-borel:fibopen}, while also making open countably many orbitwise open sets as in \cref{thm:gpd-realiz-borel}.
\qed
\end{corollary}

\subsection{Equivariant maps}
\label{sec:gpd-eqvar}

The following extends \cite[\S6]{Lgpd} to open quasi-Polish groupoids.
Recall the \emph{fiberwise lower powerspace} construction from \cref{def:fib-lowpow}.

\begin{proposition}[cf.\ \cref{thm:grp-lowpow-univ}]
\label{thm:gpd-lowpow-univ}
For any open quasi-Polish groupoid $G$, the countable fiber product of the fiberwise lower powerspace $\@F_\tau(G)^\#N_{G_0}$ is a universal $T_0$ second-countable $G$-space, as well as a universal standard Borel $G$-space, i.e., every other such $G$-space admits an equivariant topological (resp., Borel) fiberwise embedding over $G_0$ into $\@F_\tau(G)^\#N_{G_0}$.
\end{proposition}
\begin{proof}
As in \cref{thm:grp-lowpow-univ}, we first check

\begin{lemma}
\label{thm:gpd-lowpow-trans}
For any open topological groupoid $G$ and topological $G$-space $X$, the left translation action $G \times_{G_0} \@F_p(X) -> \@F_p(X)$ is continuous.
\end{lemma}

The proof is the same as \cref{thm:grp-lowpow-trans}, using \cref{it:fib-lowpow-unit}, \cref{it:fib-lowpow-prod}, \cref{it:fib-lowpow-funct}.
Now let $p : X -> G_0$ be any $T_0$ $G$-space.
For each $A \in \@O(X)$, the map $U |-> U \cdot A : \@O(G) -> \@O(X)$ preserves unions, and is $\@O(G_0)$-linear (where $G$ is regarded as a bundle via $\tau$) by Frobenius reciprocity \cref{it:gpd-vaught-frob}, hence corresponds by \cref{thm:fib-lin-lowpow} to a continuous map $h_A : X -> \@F_\tau(G)$ over $G_0$ such that $h_A^{-1}(\Dia_{G_0} U) = U \cdot A$.
Then for any basis $\@A \subseteq \@O(X)$, $h_\@A := (h_A)_{A \in \@A} : X -> \@F_\tau(G)^\@A_{G_0}$ is easily seen to be an equivariant embedding over $G_0$, proving the topological case.
The Borel case then follows by \cref{thm:gpd-borel-realiz}.
\end{proof}

We now state the groupoid analogs of the other results in \cref{sec:grp-eqvar}, similarly numbered and proved in exactly the same way:

\matcheq{thm:grp-eqvar-prod}
\begin{corollary}
For any open quasi-Polish groupoid $G$ and $T_0$ $G$-spaces $p : X -> G_0$ and $q : Y -> G_0$, a continuous map $f : X -> Y$ over $G_0$ is $G$-equivariant iff for every $U \in \@O(G)$ and $B \in \@O(Y)$, we have $f^{-1}(U \cdot B) = U \cdot f^{-1}(B)$.
\qed
\end{corollary}

\matcheq{thm:grp-eqvar-vaught}
\begin{corollary}
\label{thm:gpd-eqvar-vaught}
For any open quasi-Polish groupoid $G$ and standard Borel $G$-spaces $p : X -> G_0$ and $q : Y -> G_0$, a Borel map $f : X -> Y$ over $G_0$ is $G$-equivariant iff for every $U \in \@O(G)$ and $B \in \@{BO}_G(Y)$, we have $f^{-1}(U * B) = U * f^{-1}(B)$.
\qed
\end{corollary}

\matcheq{rmk:grp-eqvar-vaught}
\begin{remark}
In \cref{thm:gpd-eqvar-vaught}, it is enough to have $f^{-1}(U * B) = U * f^{-1}(B)$ for a countable $q$-fiberwise separating family of $B \in \@B(Y)$, arguing as in \cref{rmk:grp-eqvar-vaught} using \cite[\S6]{Lgpd}.
\end{remark}

\matcheq{thm:grp-realiz-univ}
\begin{proposition}
\label{thm:gpd-realiz-univ}
Let $G$ be an open quasi-Polish groupoid, $X$ be a quasi-Polish space equipped with a continuous map $p : X -> G_0$ and a Borel action of $G$ such that $\@O(G) \oast \@O(X) \subseteq \@O(X)$.
Then any continuous equivariant map $f : X -> Y$ into another quasi-Polish $G$-space $q : Y -> G_0$ is in fact continuous from the coarser topology $\@O(G) \oast \@O(X)$, which is hence the universal continuous ``completion'' of $X$.
\qed
\end{proposition}

\subsection{Open relations and bundles of structures}
\label{sec:gpd-struct}

\begin{remark}[cf.\ \cref{rmk:borel-binary-open}]
For two Borel maps $f : X -> Z$ and $g : Y -> Z$ between standard Borel spaces, it is easily seen that a Borel fiberwise binary relation $R \subseteq X \times_Z Y$ can be made open in the fiber product of some compatible quasi-Polish topologies on $X, Y, Z$ making $f, g$ continuous iff $R \in \@B(X) \otimes_Z \@B(Y)$, i.e., $R$ is a countable union of Borel rectangles $\bigcup_i (A_i \times_Z B_i)$.

If we want the topologies of $X, Y$ to be contained in compatible $\sigma$-topologies $\@S(X) \subseteq \@B(X)$ and $\@S(Y) \subseteq \@B(Y)$,
then we need to require $R \in \@S(X) \otimes_Z \@S(Y)$.
\end{remark}

We now have the analogous characterization for making binary relations open in fiber products of $G$-spaces.
As in \cref{thm:grp-binary-realiz}, there are two types of conditions, referring to either the diagonal action of $G$ or the ``product action of $G^2$''; here, however, ``$G^2$'' is no longer a groupoid.

Let $G^2_\tau := G \tensor[_\tau]\times{_\tau} G$ be all pairs of morphisms with the same target, but possibly two different sources.
Given two $G$-spaces $p : X -> G_0$ and $q : Y -> G_0$, we may let $G^2_\tau$ ``act'' coordinatewise on the (ordinary, \emph{not} fiber) product $X \times Y$, landing in the fiber product $X \times_{G_0} Y$; this yields a map
\begin{equation*}
\alpha_X \times \alpha_Y : G^2_\tau \times_{G_0^2} (X \times Y) := G^2_\tau \tensor[_{\sigma^2}]\times{_{p \times q}} (X \times Y) --> X \times_{G_0} Y.
\end{equation*}
Note that
\begin{equation}
\label{eq:gpd-binary-swap}
G^2_\tau \times_{G_0^2} (X \times Y)
\cong (G \times_{G_0} X) \tensor[_{\tau\pi_1}]\times{_{\tau\pi_1}} (G \times_{G_0} Y)
\end{equation}
via switching the middle two factors; under this isomorphism, $\alpha_X \times \alpha_Y$ becomes simply the fiber product of $\alpha_X, \alpha_Y$ over $G_0$.
In particular, $\alpha_X \times \alpha_Y$ is equipped with a fiberwise topology, the product of the $\alpha_X$-fiberwise and $\alpha_Y$-fiberwise topologies, forming a standard Borel-overt bundle of quasi-Polish spaces.
Thus we have a Baire category quantifier $\exists^*_{\alpha_X \times \alpha_Y}$, whose restriction to rectangles we continue to denote by $*$, with the usual meaning for $\oast$ (\cref{def:ostar,def:gpd-vaught}).

\begin{theorem}[cf.\ \cref{thm:grp-binary-realiz}]
\label{thm:gpd-binary-realiz}
Let $G$ be an open quasi-Polish groupoid, $p : X -> G_0$ and $q : Y -> G_0$ be two standard Borel $G$-spaces, and $\@S(X) \subseteq \@B(X)$ and $\@S(Y) \subseteq \@B(Y)$ be compatible $\sigma$-topologies such that $p^{-1}(\@O(G_0)), \@O(G) \oast \@S(X) \subseteq \@S(X)$ and $q^{-1}(\@O(G_0)), \@O(G) \oast \@S(Y) \subseteq \@S(Y)$.
For any $R \in \@B(X \times_{G_0} Y)$, the following are equivalent:
\begin{enumerate}[label=(\roman*)]
\item \label{thm:gpd-binary-realiz:open}
$R \in \@O(X) \otimes_{G_0} \@O(Y)$ for some quasi-Polish topologies $\@O(X) \subseteq \@S(X)$ and $\@O(Y) \subseteq \@S(Y)$ making $p, q, \alpha_X, \alpha_Y$ continuous.
\item \label{thm:gpd-binary-realiz:vaught2}
$R \in \@O(G^2_\tau) \oast (\@S(X) \otimes \@S(Y)) = (\@O(G) \oast \@S(X)) \otimes_{G_0} (\@O(G) \oast \@S(Y))$,
i.e., we have $R = \bigcup_i ((U_i \tensor[_\tau]\times{_\tau} V_i) * (A_i \times_{G_0} B_i)) = \bigcup_i (U_i * A_i) \times_{G_0} (V_i * B_i)$ for countably many $U_i, V_i \in \@O(G)$, $A_i \in \@S(X)$, and $B_i \in \@S(Y)$, where the first $*$ refers to $\exists^*_{\alpha_X \times \alpha_Y}$.
\item \label{thm:gpd-binary-realiz:rect2-open}
$(\alpha_X \times \alpha_Y)^{-1}(R) \in \@O(G^2_\tau) \otimes_{G_0^2} (\@S(X) \otimes \@S(Y))$.
\item \label{thm:gpd-binary-realiz:rect2-vaught}
$(\alpha_X \times \alpha_Y)^{-1}(R) \in \@O(G^2_\tau) \otimes_{G_0^2} ((\@O(G) \oast \@S(X)) \otimes (\@O(G) \oast \@S(Y)))$.
\end{enumerate}
Furthermore, letting $(\@O(G) \otimes_{G_0} \@S(X))^*_{\pi_2}$ consist of all Borel $D \subseteq G \tensor[_\sigma]\times{_p} X$ which are $=^*_{\pi_2}$ to a set in $\@O(G) \otimes_{G_0} \@S(X)$, the following are also equivalent to the above:
\begin{enumerate}[resume*]
\item \label{thm:gpd-binary-realiz:vaught-vaught*}
$R \in \@O(G) \oast (\exists^*_{\alpha_X}((\@O(G) \otimes_{G_0} \@S(X))^*_{\pi_2}) \otimes_{G_0} \exists^*_{\alpha_Y}((\@O(G) \otimes_{G_0} \@S(Y))^*_{\pi_2}))$.
\item \label{thm:gpd-binary-realiz:vaught-vaught}
$R \in \@O(G) \oast ((\@O(G) \oast \@S(X)) \otimes_{G_0} (\@O(G) \oast \@S(Y)))$.
\item \label{thm:gpd-binary-realiz:rect-vaught*}
$(\alpha_{X \times_{G_0} Y})^{-1}(R) \in \@O(G) \otimes_{G_0} \exists^*_{\alpha_X}((\@O(G) \otimes_{G_0} \@S(X))^*_{\pi_2}) \otimes_{G_0} \exists^*_{\alpha_Y}((\@O(G) \otimes_{G_0} \@S(Y))^*_{\pi_2})$.
\item \label{thm:gpd-binary-realiz:rect-vaught}
$(\alpha_{X \times_{G_0} Y})^{-1}(R) \in \@O(G) \otimes_{G_0} (\@O(G) \oast \@S(X)) \otimes_{G_0} (\@O(G) \oast \@S(Y))$.
\end{enumerate}
Moreover, countably many $R$ obeying these conditions may be made simultaneously open as in \cref{thm:gpd-binary-realiz:open}, while also simultaneously making open countably many $A \subseteq X$ and $B \subseteq Y$ satisfying \cref{thm:gpd-realiz}.
\end{theorem}
\begin{proof}
Essentially the same proof as for \cref{thm:grp-binary-realiz} works.
Here, the equivalence of \cref{thm:gpd-binary-realiz:open}--\cref{thm:gpd-binary-realiz:rect2-vaught} strictly speaking does not directly follow from \cref{thm:gpd-realiz}, since $G^2_\tau$ is not a groupoid; but as there, we easily have
\cref{thm:gpd-binary-realiz:rect2-vaught}$\implies$%
\cref{thm:gpd-binary-realiz:rect2-open}$\implies$%
\cref{thm:gpd-binary-realiz:vaught2}, which implies
\cref{thm:gpd-binary-realiz:open}
since we may make each of the given sets in $\@O(G) \oast \@S(X), \@O(G) \oast \@S(Y)$ open by \cref{thm:gpd-realiz}, which in turn implies
\cref{thm:gpd-binary-realiz:rect2-vaught} by \cref{thm:gpd-realiz} since the right side of \cref{thm:gpd-binary-realiz:rect2-vaught} is the same as $(\@O(G) \otimes_{G_0} (\@O(G) \oast \@S(X))) \tensor[_{\tau\pi_1}]\otimes{_{\tau\pi_1}} (\@O(G) \otimes_{G_0} (\@O(G) \oast \@S(Y)))$ via \cref{eq:gpd-binary-swap}.
As in \cref{thm:grp-binary-realiz},
\cref{thm:gpd-binary-realiz:rect2-vaught} implies
\cref{thm:gpd-binary-realiz:rect-vaught} which implies each of
\cref{thm:gpd-binary-realiz:rect-vaught*},
\cref{thm:gpd-binary-realiz:vaught-vaught}
each of which implies
\cref{thm:gpd-binary-realiz:vaught-vaught*},
which implies
\cref{thm:gpd-binary-realiz:vaught2}
by the same proof as in \cref{thm:grp-binary-realiz}, with \cref{it:grp-binary-realiz:rect-baire} changed to
\begin{eqenum}
\item
If $M \subseteq G$ is $\sigma$-fiberwise meager, then $M \tensor[_\tau]\times{_\tau} G, G \tensor[_\tau]\times{_\tau} M \subseteq G^2_\tau$ are orbitwise meager for the diagonal action $G \curvearrowright G^2_\tau$, since $\alpha_{G^2_\tau}^{-1}(M \times G) = \mu^{-1}(M) \tensor[_\tau]\times{_\tau} G \subseteq G \tensor[_\sigma]\times{_\tau} G \tensor[_\tau]\times{_\tau} G$ is $\pi_{23}$-fiberwise homeomorphic via $(g,h,k) |-> (gh,h,k)$ to $M \tensor[_\sigma]\times{_\sigma} G \tensor[_\tau]\times{_\tau} G$.
\qedhere
\end{eqenum}
\end{proof}

\begin{remark}
There is also a fiberwise version of the above result, that we will not spell out, replacing the roles of $\@O(G^2_\tau), \@O(G)$ by $\@B$ or equivalently by $\@{BO}_\tau$ as in \cref{thm:gpd-fib-realiz}.
The versions with $\@{BO}_\tau$ follow from the above and Kunugui--Novikov, as in the proof of \cref{thm:gpd-fib-realiz}; the versions with $\@B$ then follow from the $\tau$-fiberwise Baire property (\cref{thm:fib-bov-baire}) and Pettis's theorem \cref{it:gpd-vaught-pettis} (which can also be proved for the ``action'' of $G^2_\tau$ in the same way).
\end{remark}

\begin{corollary}[characterization of ``potentially open'' relations; cf.\ \cref{thm:grp-binary-realiz-borel}]
\label{thm:gpd-binary-realiz-borel}
Let $G$ be an open quasi-Polish groupoid, $p : X -> G_0$ and $q : Y -> G_0$ be standard Borel $G$-spaces.
For any $R \in \@B(X \times_{G_0} Y)$, the following are equivalent:
\begin{enumerate}[label=(\roman*)]
\item \label{thm:gpd-binary-realiz-borel:open}
$R \in \@O(X) \otimes_{G_0} \@O(Y)$ for some compatible quasi-Polish topologies $\@O(X), \@O(Y)$ making $p, q, \alpha_X, \alpha_Y$ continuous.
\item \label{thm:gpd-binary-realiz-borel:vaught2}
$R \in \@O(G^2_\tau) \oast (\@B(X) \otimes \@B(Y)) = (\@O(G) \oast \@B(X)) \otimes_{G_0} (\@O(G) \oast \@B(Y)) = \@{BO}_G(X) \otimes_{G_0} \@{BO}_G(Y)$.
\item \label{thm:gpd-binary-realiz-borel:rect2-open}
$(\alpha_X \times \alpha_Y)^{-1}(R) \in \@O(G^2_\tau) \otimes_{G_0^2} (\@B(X) \otimes \@B(Y))$.
\item \label{thm:gpd-binary-realiz-borel:rect2-vaught}
$(\alpha_X \times \alpha_Y)^{-1}(R) \in \@O(G^2_\tau) \otimes_{G_0^2} (\@{BO}_G(X) \otimes \@{BO}_G(Y))$.
\item \label{thm:gpd-binary-realiz-borel:vaught-vaught*}
$R \in \@O(G) \oast (\@B(X) \otimes_{G_0} \@B(Y))$.
\item \label{thm:gpd-binary-realiz-borel:vaught-vaught}
$R \in \@O(G) \oast (\@{BO}_G(X) \otimes_{G_0} \@{BO}_G(Y))$.
\item \label{thm:gpd-binary-realiz-borel:rect-vaught*}
$(\alpha_{X \times_{G_0} Y})^{-1}(R) \in \@O(G) \otimes_{G_0} \@B(X) \otimes_{G_0} \@B(Y)$.
\item \label{thm:gpd-binary-realiz-borel:rect-vaught}
$(\alpha_{X \times_{G_0} Y})^{-1}(R) \in \@O(G) \otimes_{G_0} \@{BO}_G(X) \otimes_{G_0} \@{BO}_G(Y)$.
\end{enumerate}
Moreover, countably many $R$ obeying these conditions may be made simultaneously open as in \cref{thm:gpd-binary-realiz-borel:open}, while also simultaneously making open countably many other $A \in \@{BO}_G(X)$ and $B \in \@{BO}_G(Y)$.
\end{corollary}
\begin{proof}
Same as \cref{thm:grp-binary-realiz-borel}.
\end{proof}


The generalization to $n$-ary relations is straightforward; as before, we only state the main parts:

\begin{corollary}[cf.\ \cref{thm:grp-nary-realiz-borel}]
\label{thm:gpd-nary-realiz-borel}
Let $G$ be an open quasi-Polish groupoid, $p_i : X_i -> G_0$ be countably many standard Borel $G$-spaces, and $R_k \subseteq X_{i_{k,1}} \times_{G_0} \dotsb \times_{G_0} X_{i_{k,n_k}}$ be countably many Borel fiberwise relations of arities $n_k \in \#N$.
Then there are compatible quasi-Polish topologies on each $X_i$ making $p_i, \alpha_i$ continuous and each $R_k$ open, iff each $R_k \in \@O(G) \oast (\@B(X_{i_{k,1}}) \otimes_{G_0} \dotsb \otimes_{G_0} \@B(X_{i_{k,n_k}}))$.
In particular, this can be done if $R_k$ is $G$-invariant and a countable union of Borel rectangles.
\qed
\end{corollary}

The following generalizes \cref{thm:gpd-realiz-dis}:

\begin{corollary}[change of topology for relations; cf.\ \cref{thm:grp-nary-realiz-dis}]
\label{thm:gpd-nary-realiz-dis}
Let $G$ be an open quasi-Polish groupoid, $p_i : X_i -> G_0$ be countably many quasi-Polish $G$-spaces, and $R_k \subseteq X_{i_{k,1}} \times_{G_0} \dotsb \times_{G_0} X_{i_{k,n_k}}$ be countably many fiberwise relations of arities $n_k \in \#N$, such that each $R_k \in \@O(G) \oast (\*\Sigma^0_\xi(X_{i_{k,1}}) \otimes_{G_0} \dotsb \otimes_{G_0} \*\Sigma^0_\xi(X_{i_{k,n_k}}))$, i.e., $R_k$ can be written as a countable union of sets $U * (A_1 \times_{G_0} \dotsb \times_{G_0} A_{n_k})$ where $U \in \@B(G)$ and $A_j \in \*\Sigma^0_\xi(X_{i_{k,j}})$.
Then there are finer quasi-Polish topologies on each $X_i$ contained in $\*\Sigma^0_\xi(X_i)$ for which $p_i, \alpha_i$ are still continuous, such that each $R_k$ becomes open.
In particular, this can be done if $R_k$ is $G$-invariant and a countable union of $\*\Sigma^0_\xi$ rectangles.
\end{corollary}
\begin{proof}
Same as \cref{thm:grp-nary-realiz-dis}.
\end{proof}

Next, we consider applications to ``topological realizations of bundles of topological structures'', i.e., where we start with $G$-spaces $p_i : X_i -> G_0$ with each fiber equipped with a topology as well as some relations.
Recall that with just a fiberwise topology, this is addressed by \cref{thm:gpd-top-realiz}.

\begin{corollary}[cf.\ \cref{rmk:grp-top-nary-realiz}]
\label{thm:gpd-top-nary-realiz}
Let $G$ be an open quasi-Polish groupoid, $p_i : X_i -> G_0$ be countably many standard Borel $G$-bundles of quasi-Polish spaces, $R_k \subseteq X_{i_{k,1}} \times_{G_0} \dotsb \times_{G_0} X_{i_{k,n_k}}$ be countably many Borel fiberwise open relations of positive arities $n_k > 0$.
Then there are compatible (global) quasi-Polish topologies on each $X_i$ making $p_i, \alpha_i$ continuous and restricting to the original $p_i$-fiberwise topology, such that each set in $\@O(G) * R_k$ becomes open.
Hence if $R_k$ is $G$-invariant, then $R_k$ itself can be made open; in other words, a ``standard Borel $G$-bundle of quasi-Polish open relational structures'' can be realized as a quasi-Polish $G$-space with globally open relations.
\end{corollary}
\begin{proof}
For simplicity of notation we only consider two bundles $p : X -> G_0$ and $q : Y -> G_0$ with a $(p \times_{G_0} q)$-fiberwise open binary relation $R \subseteq X \times_{G_0} Y$.
By \cref{thm:gpd-top-realiz}, we can find topological realizations of $X, Y$ compatible with the fiberwise topologies.
Now we apply \cref{thm:gpd-binary-realiz} with $\@S(X) := \@{BO}_p(X)$ and $\@S(Y) := \@{BO}_q(Y)$ (which are closed under $\@O(G) * (-)$ by the proof of \cref{thm:gpd-top-realiz}) to make each set in $\@O(G) * R$ open, while also making open countably many basic open sets in these prior topological realizations, to ensure that the new realizations restrict to the original fiberwise topologies.
To see that we can apply \cref{thm:gpd-binary-realiz} to $\@O(G) * R$, note that $R \in \@{BO}_p(X) \otimes_{G_0} \@{BO}_q(Y)$ by Kunugui--Novikov.
We now claim that
\begin{equation}
\@{BO}_p(X) \subseteq \exists^*_{\alpha_X}((\@O(G) \otimes_{G_0} \@{BO}_p(X))^*_{\pi_2}),
\end{equation}
and similarly for $Y$, whence every set in $\@O(G) * R$ satisfies \cref{thm:gpd-binary-realiz}\cref{thm:gpd-binary-realiz:vaught-vaught*}.
Indeed, for $A \in \@{BO}_p(X)$,
\begin{align*}
\alpha_X^{-1}(A)
= \bigcup_i (U_i \times_{G_0} A_i)
\in \@B(G) \otimes_{G_0} \@{BO}_p(X)
\end{align*}
for $U_i \in \@B(G)$ and $A_i \in \@{BO}_p(X)$ by \cref{eq:gpd-top-realiz}; letting
\begin{align*}
U_i =^*_\sigma \bigcup_j (\sigma^{-1}(B_{ij}) \cap V_{ij})
\end{align*}
for $B_{ij} \in \@B(G_0)$ and $V_{ij} \in \@O(G)$ by the fiberwise Baire  property (\cref{thm:fib-baire-borel}),
\begin{align*}
\alpha_X^{-1}(A)
&=^*_{\pi_2} \bigcup_{i,j} ((\sigma^{-1}(B_{ij}) \cap V_{ij}) \times_{G_0} A_i) \\
&= \bigcup_{i,j} (V_{ij} \times_{G_0} (p^{-1}(B_{ij}) \cap A_i))
\in \@O(G) \otimes_{G_0} \@{BO}_p(X),
\end{align*}
whence $A = \exists^*_{\alpha_X}(\alpha_X^{-1}(A)) \in \exists^*_{\alpha_X}((\@O(G) \otimes_{G_0} \@{BO}_p(X))^*_{\pi_2})$.
\end{proof}

\begin{remark}
\label{rmk:gpd-bov-0ary-realiz}
If some of the bundles in the preceding result are Borel-overt, then applying \cref{thm:gpd-bov-realiz} afterwards, we may pass to a finer groupoid topology $\~G$ so that those bundles become open (while others remain continuous).

If we are willing to refine the topology of $G$ to $\~G$, then we can also handle a nullary relation $R$, which just means $R \subseteq G_0$ (a ``bundle of truth values''), by adding $\@O(G) * R$ to $\@O(\~G_0)$.

(The proof above fails for nullary $R$ in the step ``note that $R \in \@{BO}_p(X) \otimes \@{BO}_q(Y)$ by Kunugui--Novikov'', whose nullary analog would say that $R$ is a countable union of $G_0$ (the nullary $\times_{G_0}$), i.e., $R \in \{\emptyset, G_0\}$; the binary case implicitly absorbed a $(p \times_{G_0} q)^{-1}(B)$ into one of the factors.)
\end{remark}

One special instance of an open relation in a topological space is the equality relation ${=_X} \subseteq X^2$, which is open iff every point is isolated, i.e., $X$ is discrete.
The analogous fiberwise condition says that $X$ is ``uniformly fiberwise discrete'', or a ``bundle of \emph{sets}'', rather than \emph{spaces}:

\begin{definition}
\label{def:etale}
A continuous map $p : X -> Z$ between topological spaces is \defn{isolated} if the diagonal in $X \times_Z X$ is open, and \defn{étale} if $p$ is both open and isolated.

Equivalently, $p$ is isolated iff $X$ has an open cover of sets on which $p$ is injective, and étale iff it is a local homeomorphism, i.e., $X$ has an open cover of sets on which $p$ is an open embedding.

For a general reference on étale maps, see \cite[\S2.3]{Tsh} or \cite[\S4.1]{Cgpd}.

By \cref{it:qpol-loc}, an étale space $X -> Z$ over a quasi-Polish $Z$ is quasi-Polish iff it is second-countable.

A \defn{second-countable étale bundle of structures} over a quasi-Polish space $Z$ consists of countably many second-countable étale bundles $p_i : X_i -> Z$, countably many continuous maps over $Z$ between them of various arities $f_j : X_{i_{j,1}} \times_Z \dotsb \times_Z X_{i_{j,m_j}} -> X_{i_j}$, and countably many open relations between them of various arities $R_k \subseteq X_{i_{k,1}} \times_Z \dotsb \times_Z X_{i_{k,n_k}}$; see \cite[\S4.3]{Cgpd}.

A \defn{standard Borel bundle of countable structures} is defined the same way, except the $X_i$ are merely standard Borel spaces, the $p_i$ are countable-to-1 Borel maps, the $f_j$ are Borel maps over $Z$, and the $R_k$ are Borel relations.

For a quasi-Polish groupoid $G$, a \defn{$G$-bundle of structures} is a bundle of structures over $G_0$ together with a (continuous, resp., Borel) action of $G$ via isomorphisms between fibers.
\end{definition}

The following result was proved for open Polish $G$, and for étale spaces only (without functions or relations), in \cite[1.5]{Cscc} using \emph{ad hoc} methods.
We give here a simple proof as a direct application of \cref{thm:gpd-top-nary-realiz}.

\begin{corollary}
\label{thm:gpd-etale-realiz}
Every standard Borel $G$-bundle of countable structures over an open quasi-Polish groupoid $G$ has a topological realization as a second-countable étale $\~G$-bundle of structures over $G$ with a finer open quasi-Polish groupoid topology $\~G$.
If there are only relations of positive arities, then we can also find a realization over the original $G$ as an isolated (not necessarily étale) bundle.
\end{corollary}
\begin{proof}
To realize relations including equality on each $X_i$, apply \cref{thm:gpd-top-nary-realiz,rmk:gpd-bov-0ary-realiz} to the fiberwise discrete topology on each $X_i$, which is Borel-overt by the Lusin--Novikov uniformization theorem (see e.g., \cite[18.10]{Kcdst}).
To realize functions, realize their graphs, using the standard fact that a fiberwise map $f : X -> Y$ over $Z$ between two étale bundles $p : X -> Z$ and $q : Y -> Z$ is continuous iff its graph is open in $X \times_Z Y$ ($\Longrightarrow$ because $=_Y$ is open; $\Longleftarrow$ because $p$ is open).
\end{proof}

\section{Open localic groupoids}
\label{sec:loc}

We now generalize all of the topological realization results from the preceding two sections to localic group(oid) actions.
Because those results were proved in a point-free manner, the generalization will be nearly immediate, given the right point-free topological foundations; the bulk of this section is devoted to developing such foundations.
In \cref{sec:loc-prelim}, we quickly review some basic concepts of locale theory and localic descriptive set theory from \cite{Cborloc}.
In \cref{sec:loc-fib,sec:loc-lin}, we develop a point-free analog of the theory of fiberwise topology and Baire category quantifiers from \crefrange{sec:fib}{sec:fib-lin-baire}.
In \cref{sec:loc-gpd}, we describe the generalization of the machinery of \cref{sec:grp,sec:gpd} to the setting of localic groupoids, as well as the resulting point-free topological realization theorems.

\subsection{Generalities on locales}
\label{sec:loc-prelim}

The following is a very terse review of the main definitions of locale theory, for which see \cite[C1.1--1.2]{Jeleph}, \cite{Jstone}, \cite{PPloc}, and especially descriptive set theory for locales as developed in \cite{Cborloc} (see \cite[\S2]{Cpettis} for a more concise summary, that is however less general than what we need here).

For the reader familiar with locales, two key points should be noted about our conventions.
First, we strictly distinguish between the ``algebraic'' and ``spatial'' views of locales (like \cite{Jeleph} but unlike \cite{Jstone}, \cite{PPloc}); this allows us to unambiguously use notation and terminology on the ``spatial'' side quite close to that in the classical point-set setting.
Thus, for instance, we speak of \emph{locales} $X$ versus \emph{frames} $\@O(X)$ of \emph{open sets}; we interchangeably denote meets in $\@O(X)$ by $\wedge$ or $\cap$; we denote images by $f(A)$; etc.
Second, our ``descriptive set theory'' is fundamentally Boolean in nature (unlike that of \cite{Idst}); in fact, we make no use of Heyting algebra operations or intuitionistic logic.
Thus, for instance, for open $U, V \in \@O(X)$, $(U => V) = (\neg U \cup V)$ refers to the Boolean implication, possibly after passing to the frame of nuclei $\@N(\@O(X))$ if $U$ is not clopen; see \cref{cvt:loc-forget}.

\begin{definition}
\label{def:loc}
A \defn{suplattice} is a poset equipped with arbitrary joins.
We denote the category of suplattices (and suplattice homomorphisms, i.e., join-preserving maps) by $\!{Sup}$.

A \defn{frame} is a poset with finite meets and arbitrary joins, the former distributing over the latter.
We denote the category of frames by $\!{Frm}$.

A \defn{locale} $X$ is the same thing as a frame $\@O(X)$, whose elements $U \in \@O(X)$ we call \defn{open sets} $U \subseteq X$.
The partial order of $\@O(X)$ is also denoted $\subseteq$, and lattice operations are also denoted ${\cap} := {\wedge}$, $X := \top =$ the top element, $\emptyset := \bot$, etc.
A \defn{continuous map} $f : X -> Y$ between locales is a frame homomorphism $f^* : \@O(Y) -> \@O(X)$.
We denote the category of locales by $\!{Loc}$.

A \defn{product locale} $X \times Y$ is given by a coproduct frame%
\footnote{Strictly speaking, this notation conflicts with our earlier \cref{def:ostar} of $\otimes$ to mean product $\sigma$-topology.
But when $\@S(X), \@S(Y)$ are \emph{compatible} $\sigma$-topologies on standard Borel spaces, a straightforward chase through universal properties and \cref{thm:loc-stdsigma} shows that the product $\sigma$-topology $\@S(X) \otimes \@S(Y)$ is also the coproduct $\sigma$-frame.
We thus adopt this abuse of notation, analogous to e.g., how $\oplus$ can denote internal or external direct sum of vector spaces.}
$\@O(X \times Y) := \@O(X) \otimes \@O(Y)$; an open rectangle is denoted $U \times V := \pi_1^*(U) \cap \pi_2^*(V)$, where $\pi_1^*, \pi_2^*$ are the coproduct injections.
Similar notation is used for fiber products.
A \defn{sublocale} $X \subseteq Y$ is given by a quotient frame $\@O(Y) ->> \@O(X)$.
\end{definition}

Now let $\kappa$ be an uncountable regular cardinal.
By \defn{$\kappa$-ary}, we mean of size $<\kappa$.%
\footnote{Unlike in \cite{Cborloc}, we do not use $\sigma$ as an abbreviation for $\omega_1$, due to the potential for confusion with the source map of a groupoid.
Also, we will usually assume $\kappa < \infty$, unless otherwise noted.}

\begin{definition}
\label{def:loc-kloc}
A \defn{$\kappa$-suplattice} is a poset equipped with $\kappa$-ary joins; the category of these is denoted $\!{\kappa Sup}$.
A \defn{$\kappa$-frame} is a poset with finite meets and $\kappa$-ary joins, the former distributing over the latter; the category of these is denoted $\!{\kappa Frm}$.
A \defn{$\kappa$-locale} $X$ is a $\kappa$-frame $\@O_\kappa(X)$, whose elements are called \defn{$\kappa$-open sets} of $X$;
a \defn{$\kappa$-continuous map} $f$ is a $\kappa$-frame homomorphism $f^*$ in the opposite direction; and the category of these is denoted $\!{\kappa Loc}$.

A ($\kappa$-)locale $X$ is \defn{$\kappa$-based} if $\@O_{(\kappa)}(X)$ is $\kappa$-generated as a ($\kappa$-)frame, equivalently as a ($\kappa$\nobreakdash-)\allowbreak suplattice; these two notions, with and without the ($\kappa$-), are equivalent.
A \defn{standard $\kappa$-locale} is a ($\kappa$-)locale $X$ such that $\@O(X) = \@O_\kappa(X)$ is $\kappa$-presented.
We denote the full subcategory of $\kappa$-presented ($\kappa$-)frames by $\!{Frm_\kappa} \subseteq \!{Frm}$, and that of standard $\kappa$-locales by $\!{Loc_\kappa} \subseteq \!{Loc}$.

\defn{Product $\kappa$-locale} and \defn{$\kappa$-sublocale} are defined as in \cref{def:loc}.
Note that since a $\kappa$-ary colimit of $\kappa$-presented algebras is $\kappa$-presented, a $\kappa$-ary limit of standard $\kappa$-locales is still standard.
\end{definition}

\begin{definition}
A \defn{$\kappa$-Boolean algebra} is a $\kappa$-complete Boolean algebra; the category of these is denoted $\!{\kappa Bool}$, and the full subcategory of $\kappa$-presented algebras $\!{\kappa Bool_\kappa}$.
A \defn{$\kappa$-Borel locale} $X$ is a $\kappa$-Boolean algebra $\@B_\kappa(X)$, whose elements are called \defn{$\kappa$-Borel sets} of $X$;
a \defn{$\kappa$-Borel map} $f$ is a $\kappa$-Boolean homomorphism $f^*$; and the category of these is denoted $\!{\kappa BorLoc}$.
A \defn{standard $\kappa$-Borel locale} $X$ is one whose $\@B_\kappa(X)$ is $\kappa$-presented; the full subcategory of these is denoted $\!{\kappa BorLoc_\kappa}$.
\end{definition}

\begin{theorem}[{Heckmann \cite{Hqpol}, Loomis--Sikorski; see \cite[3.5.8]{Cborloc}}]
\label{thm:loc-stdsigma}
The canonical functors (forgetting the underlying set) are equivalences of categories between quasi-Polish spaces and standard $\omega_1$-locales, and between standard Borel spaces and standard $\omega_1$-Borel locales.
\end{theorem}

\begin{convention}
\label{cvt:loc-forget}
As noted above, between the various categories $\!{\kappa Frm}, \!{\kappa Bool}$ (for varying $\kappa$), we regard each object in one category as silently embedded inside its free completion to a ``higher'' category consisting of algebras with more structure;
and we regard these free functors as nameless \emph{forgetful} functors between the dual localic categories.

Thus, for instance, a $\kappa$-locale $X$ has an \emph{underlying $\kappa$-Borel locale}, whose $\@B_\kappa(X)$ is the free $\kappa$-Boolean algebra generated by the $\kappa$-frame $\@O_\kappa(X)$, in which we regard $\@O_\kappa(X) \subseteq \@B_\kappa(X)$ as a $\kappa$-subframe; as well as an \emph{underlying locale}, whose $\@O(X)$ is the free frame generated by $\@O_\kappa(X)$.
The following diagram depicts all such forgetful functors between the ``standard'' localic categories:
\begin{equation}
\begin{tikzcd}[column sep=1.5em]
|[label={left:\!{SBor}\simeq}]|
\!{\omega_1 BorLoc_{\omega_1}} \rar &
\!{\omega_2 BorLoc_{\omega_2}} \rar &
\dotsb \rar &
\!{\kappa BorLoc_\kappa} \rar &
\dotsb \rar &
\!{\infty BorLoc} \\
|[label={left:\!{QPol}\simeq}]|
\!{\omega_1 Loc_{\omega_1}} \uar \rar &
\!{\omega_2 Loc_{\omega_2}} \uar \rar &
\dotsb \rar &
\!{\kappa Loc_\kappa} \uar \rar &
\dotsb \rar &
\!{Loc} \uar
\end{tikzcd}
\end{equation}
Here $\!{SBor}$ is the category of standard Borel spaces, while $\!{QPol}$ is that of quasi-Polish spaces.

We also let $\@B_\infty(X) := \injlim_\kappa \@B_\kappa(X)$, the \defn{$\infty$-Borel sets} (which may form a proper class).
An \defn{$\infty$-Borel locale} is a $\kappa$-Borel locale for some $\kappa < \infty$, where we remember only $\@B_\infty(X)$, hence a class-sized complete Boolean algebra which is presented by a set; the category of these is $\!{\infty BorLoc}$.

All lattice and Boolean operations are interpreted as taking place in $\@B_\infty(X)$, and may or may not land in the original subalgebra.
For instance, for a $\kappa$-locale $X$ and $\kappa$-open $U, V \in \@O_\kappa(X)$, $(U => V) := (\neg U \cup V)$ may not land in $\@O_\kappa(X)$, but still does at least land in $\@B_\kappa(X) \subseteq \@B_\infty(X)$.
\end{convention}

\begin{definition}
\label{def:loc-borhier}
The \defn{$\kappa$-Borel hierarchy} of a $\kappa$-locale $X$ is defined by declaring
$\kappa\Sigma^0_1(X) := \@O_\kappa(X)$,
$\kappa\Sigma^0_{\xi+1}(X) := \@N_\kappa(\kappa\Sigma^0_\xi(X))$
where the functor $\@N_\kappa : \!{\kappa Frm} -> \!{\kappa Frm}$ freely adjoins a complement for every element of a $\kappa$-frame, and
$\kappa\Sigma^0_\xi(X) := \injlim_{\zeta < \xi} \kappa\Sigma^0_\zeta(X)$
for a limit ordinal $\xi$ where the colimit is taken in the category $\!{\kappa Frm}$.
Thus $\@B_\kappa(X) = \injlim_{\xi < \kappa} \kappa\Sigma^0_\xi(X)$.
Put $\kappa\Pi^0_\xi(X) := \neg(\kappa\Sigma^0_\xi(X)) \subseteq \@B_\kappa(X)$.
\end{definition}

\begin{remark}
\label{rmk:loc-sub}
For a standard $\kappa$-locale $X$, there is a canonical bijection between $\kappa\Pi^0_2(X)$ and standard $\kappa$-sublocales of $X$ (given by taking image as defined below).
Similarly, for a standard $\kappa$-Borel locale $X$, there is a canonical bijection between $\@B_\kappa(X)$ and standard $\kappa$-Borel sublocales of $X$ (the \defn{Lusin--Suslin theorem} for $\kappa$-Borel locales).
See \cite[3.4.9]{Cborloc}.
We henceforth treat these bijections as identities, i.e., we identify standard $\kappa$-(Borel) sublocales with certain $\kappa$-Borel sets.
\end{remark}

The following notion appears only implicitly in \cite{Cborloc}:

\begin{definition}
\label{def:loc-compat}
For a $\kappa$-Borel locale $X$, a \defn{$\kappa$-Borel $\kappa$-topology} on $X$ will mean a $\kappa$-subframe $\@S \subseteq \@B_\kappa(X)$, and a \defn{compatible $\kappa$-topology} on $X$ will mean one which freely generates $\@B_\kappa(X)$ as a $\kappa$-Boolean algebra.
In other words, the corresponding $\kappa$-locale $X'$ with $\@O_\kappa(X') := \@S$ is canonically $\kappa$-Borel-isomorphic to $X$, via the map $X -> X'$ induced by the inclusion $\@S `-> \@B_\kappa(X)$.

Using \cref{thm:loc-stdsigma}, this definition is easily seen to agree when $\kappa = \omega_1$ with \cref{def:qpol-compat}.

For example, for a $\kappa$-locale $X$, each $\kappa\Sigma^0_\xi(X)$ is a compatible $\kappa$-topology (cf.\ \cref{ex:qpol-compat-borel}).
If $X$ is a standard $\kappa$-locale $X$, then (cf.\ \cref{it:qpol-dis}) each $\kappa\Sigma^0_\xi(X)$ is a $\kappa$-directed union of compatible $\kappa$-presented topologies, i.e., we can ``change topology'' to make $\kappa\Sigma^0_\xi$ sets open \cite[3.3.7]{Cborloc}.
\end{definition}

\begin{definition}
\label{def:loc-im}
For a $\kappa$-Borel map $f : X -> Y$ and $A \in \@B_\kappa(X)$, the \defn{$\kappa$-Borel image} $f(A) \in \@B_\kappa(Y)$ is the smallest $B \in \@B_\kappa(Y)$ such that $A \subseteq f^*(B)$, if it exists.
The \defn{$\kappa$-Borel image of $f$} is the $\kappa$-Borel image $f(X)$; if $f(X) = Y$, $f$ is \defn{$\kappa$-Borel surjective}.
The $\kappa$-Borel image may not exist \cite[4.4.3]{Cborloc}; but if it does exist, it is automatically pullback-stable, i.e., the Beck--Chevalley condition \cref{eq:fib-bc} holds for all pullbacks in $\!{\kappa BorLoc}$.
These notions also make sense for $\kappa = \infty$.

Thus, a \defn{$\kappa$-continuous $\kappa$-open map} between $\kappa$-locales may be defined in the obvious manner, i.e., the $\kappa$-Borel image of each $\kappa$-open set exists and is $\kappa$-open.
In fact, for such a map, it is enough to require a left adjoint $f(-)$ to $f^*$ on $\@O_\kappa(Y)$, rather than all of $\@B_\kappa(Y)$, and for Frobenius reciprocity \cref{eq:fib-frob} to hold on $\@O_\kappa(Y)$; this is the more standard definition of continuous open map, found for instance in \cite[Ch.~V]{JTloc}, \cite[C3.1]{Jeleph}, and is equivalent because it is pullback-stable, hence we can pull back to a finer topology making an arbitrary $\kappa$-Borel set $\kappa$-open.
Note also that a $\kappa$-continuous $\kappa$-open map is also $\lambda$-open for $\lambda \ge \kappa$, since the free functor $\!{\kappa Frm} -> \!{\lambda Frm}$ agrees with that $\!{\kappa Sup} -> \!{\lambda Sup}$ (namely, $\lambda$-generated $\kappa$-ideal completion), hence preserves the adjunction $f \dashv f^*$.
\end{definition}

\begin{remark}
\label{rmk:loc-logic}
In generalizing notions from the classical to the localic setting, it is helpful to recall that formulas and assertions of low enough logical complexity can be interpreted in the \defn{internal logic} of a sufficiently rich category such as $\!{\kappa BorLoc_\kappa}$.
Namely, $\!{\kappa BorLoc_\kappa}$ is a $\kappa$-complete, Boolean $\kappa$-extensive category, with subobjects corresponding to $\kappa$-Borel sets by \cref{rmk:loc-sub}; thus we can interpret terms as well as quantifier-free formulas in $\kappa$-infinitary first-order logic, and also certain quantifiers once we know the corresponding $\kappa$-Borel images exist.
Thus, simple definitions such as the associativity law for group actions ``$(g \cdot h) \cdot x = g \cdot (h \cdot x)$'' automatically make sense also in the localic setting.
See \cite[\S3.6]{Cborloc} for a detailed overview of this technique.
\end{remark}

\begin{theorem}[{Baire category theorem \cite[1.5]{Iloc}}]
\label{thm:loc-baire}
For any locale $X$, the intersection of all dense $\infty\Pi^0_2$ sets in $X$ is still dense.
Thus, for any $\kappa$-locale $X$, the intersection of $<\kappa$-many dense $\kappa\Pi^0_2$ sets in $X$ is still a dense $\kappa\Pi^0_2$ set.
\end{theorem}

We recall that in a topological space or locale, if a set of the form $(U => V) = (\neg U \cup V)$ is dense, for open $U, V$, then so is the open subset $(\neg U)^\circ \cup V$; see e.g., \cite[7.1]{Cqpol}.

\begin{definition}
For a locale $X$, we call an $\infty$-Borel set $A \in \@B_\infty(X)$ \defn{comeager} if it contains a dense $\infty\Pi^0_2$ set, or equivalently an intersection of dense open sets.
Thus for a $\kappa$-locale $X$, the comeager $\kappa$-Borel sets form a $\kappa$-filter in $\@B_\kappa(X)$.
\end{definition}

\subsection{Fiberwise topology in locales}
\label{sec:loc-fib}

Whereas in \cref{sec:fib} it was natural to import standard topological notions to the fiberwise context, and then show in \cref{sec:fib-bor} that things can be done in a uniformly Borel manner, in the point-free setting only the uniform Borel notions make sense to begin with.

\begin{definition}
\label{def:loc-fib}
Let $Y$ be a $\kappa$-Borel locale.
A \defn{$\kappa$-Borel bundle of $\kappa$-locales} over $Y$ will mean an arbitrary $\kappa$-locale $X$ equipped with a $\kappa$-continuous map $f : X -> Y$, where $Y$ is regarded as a ``discrete $\kappa$-locale'' with $\@O_\kappa(Y) := \@B_\kappa(Y)$.
In this situation, we denote the $\kappa$-frame of $X$ by $\@{BO}_{\kappa,f}(X)$ instead of $\@O_\kappa(X)$, and call its elements the \defn{$\kappa$-Borel $f$-fiberwise $\kappa$-open} sets of $X$.

If $Y$ is a standard $\kappa$-Borel locale, and $\@{BO}_{\kappa,f}(X)$ is $\kappa$-presented \emph{as a $\kappa$-frame equipped with a homomorphism $f^* : \@B_\kappa(Y) -> \@{BO}_{\kappa,f}(X)$}, i.e., as an ``algebra over $\@B_\kappa(Y)$'' (cf.\ \cref{rmk:lin}), then we call $X$ a \defn{standard $\kappa$-Borel bundle of standard $\kappa$-locales} over $Y$.
Note that this implies that the underlying $\kappa$-Borel locale of $X$ is standard.
In this case, we also write
$\@B_\kappa\@O_f(X) := \@{BO}_{\kappa,f}(X)$,
and call its elements \defn{$\kappa$-Borel $f$-fiberwise open} (instead of ``$\kappa$-open'').

A \defn{$\kappa$-Borel $f$-fiberwise $\kappa$-open subbasis} $\@U \subseteq \@{BO}_{\kappa,f}(X)$ is a generator of $\@{BO}_{\kappa,f}(X)$ as a $\@B_\kappa(X)$-algebra (i.e., $f^*(\@B_\kappa(Y)) \cup \@U$ generates it as a $\kappa$-frame).
If the closure of $\@U$ under $\kappa$-ary joins is already closed under finite meets, then $\@U$ is a \defn{$\kappa$-Borel $f$-fiberwise $\kappa$-open basis}.
Clearly, the closure under finite meets of a subbasis is a basis; thus every standard $\kappa$-Borel bundle of standard $\kappa$-locales has a $\kappa$-ary fiberwise basis, i.e., is \defn{fiberwise $\kappa$-based}.

If $f : X -> Y$ is a $\kappa$-continuous map between $\kappa$-locales, then we may \defn{$f$-fiberwise restrict} the global $\kappa$-topology $\@O_\kappa(X)$ to get a $\kappa$-Borel bundle of $\kappa$-locales (over the underlying $\kappa$-Borel locale of $Y$), with $\@{BO}_{\kappa,f}(X) :=$ the $\kappa$-subframe of $\@B_\kappa(X)$ generated by $f^{-1}(\@B_\kappa(X)) \cup \@O_\kappa(X)$, or equivalently the pushout of $f^* : \@O_\kappa(Y) -> \@O_\kappa(X)$ and the inclusion $\@O_\kappa(Y) -> \@B_\kappa(Y)$.
Note that $\@{BO}_{\kappa,f}(X)$ is indeed a compatible $\kappa$-topology on $X$, since the free functor $\!{\kappa Frm} -> \!{\kappa Bool}$ preserves pushouts.
If $X, Y$ are standard $\kappa$-locales, then the $f$-fiberwise restriction is a standard $\kappa$-Borel bundle.

If $f : X -> Y$ is a $\kappa$-Borel bundle of $\kappa$-locales, and $\lambda \ge \kappa$, we may regard $f$ also as a $\lambda$-Borel bundle of $\lambda$-locales, by taking $\@{BO}_{\lambda,f}(X)$ to be the free $\@B_\lambda(X)$-algebra generated by $\@{BO}_{\kappa,f}(X)$ as a $\@B_\kappa(X)$-algebra.
In other words, we complete $\@{BO}_{\kappa,f}(X)$ under $\lambda$-ary joins, while also adjoining all $\lambda$-ary Boolean combinations of elements in $f^*(\@B_\kappa(Y))$.
Per \cref{cvt:loc-forget}, we regard this as a nameless forgetful functor.
\end{definition}

The next few results justify that this is the ``correct'' point-free generalization of \cref{def:fib-bor-qpol}:

\begin{theorem}[Kunugui--Novikov uniformization for locales]
\label{thm:loc-kunugui-novikov}
Let $f : X -> Y$ be a $\kappa$-Borel map between standard $\kappa$-Borel locales, $\@S \subseteq \@B_\kappa(X)$ be a $\kappa$-ary family.
Let $g : A' -> X$ and $h : A -> X$ be two $\kappa$-Borel maps from standard $\kappa$-Borel locales, such that ``the $f$-fiberwise $\@S$-closure of $g(A')$ is disjoint from $h(A)$'': the $\kappa$-Borel set defined in the internal logic of $\!{\kappa BorLoc}$ by
\begin{align*}
\{((a_S)_{S \in \@S}, a) \in A^{\prime\@S} \times A \mid \bigwedge_{S \in \@S} (h(a) \in S \implies g(a_S) \in S \AND f(g(a_S)) = f(h(a)))\},
\end{align*}
of ``$\@S$-nets in $A'$ whose $g$-image converges in the same $f$-fiber to the $h$-image of $a \in A$'', is empty.
Then there are $\kappa$-Borel $B_S \in \@B_\kappa(Y)$ such that the ``$f$-fiberwise $\@S$-closed'' set
$C := \bigcap_{S \in \@S} (S => f^*(B_S))$
``contains $g(A')$ and is disjoint from $h(A)$'', i.e., $g^*(C) = A'$ and $h^*(C) = \emptyset$.

In particular, if $A \in \@B_\kappa(X)$ is ``$f$-fiberwise $\@S$-open'', i.e., ``disjoint from the $f$-fiberwise closure of $\neg A$'' expressed internally as above, then there are $B_S \in \@B_\kappa(Y)$ such that $A = \bigcup_{S \in \@S} (f^*(B_S) \cap S)$.
\end{theorem}

The proof is a straightforward point-free transcription of the usual proof (see \cite[8.14]{Cgpd}), using the Novikov separation theorem from \cite[4.2.1]{Cborloc}.
We will not give the details, since we do not need this result, except as informal motivation for \cref{def:loc-fib}:

\begin{remark}
It follows that the dropping of the prefix ``$\kappa$-'' in \cref{def:loc-fib} for $\kappa$-Borel $f$-fiberwise open sets in \emph{standard} $\kappa$-Borel bundles $f : X -> Y$ is justified: if $A \in \@B_\kappa(X)$, and also $A \in \@{BO}_{\lambda,f}(X)$ for some $\lambda \ge \kappa$, then the internal definition of ``$f$-fiberwise open'' in \cref{thm:loc-kunugui-novikov} (for $\@S :=$ any $\kappa$-ary $\kappa$-Borel fiberwise basis) holds, whence in fact $A \in \@B_\kappa\@O_f(X)$.
\end{remark}

\begin{proposition}[cf.\ \cref{thm:fib-bor-qpol}]
\label{thm:loc-fib}
Let $f : X -> Y$ be a standard $\kappa$-Borel bundle of standard $\kappa$-locales over a standard $\kappa$-locale $Y$.
Then there is a compatible standard $\kappa$-topology $\@O(X) \subseteq \@B_\kappa\@O_f(X)$ making $f$ continuous and $f$-fiberwise restricting to $\@B_\kappa\@O_f(X)$.
\end{proposition}
\begin{proof}
Take a $\kappa$-ary presentation of $\@B_\kappa\@O_f(X) = \ang{G \mid R}$ as an algebra over $\@B_\kappa(Y)$.
Let $Y'$ be $Y$ with a finer compatible standard $\kappa$-topology $\@O(Y) \subseteq \@O(Y') \subseteq \@B(Y)$, making open all of the $<\kappa$-many elements of $\@B_\kappa(Y)$ appearing in some relation in $R$.
Then that same presentation presents an algebra over the $\kappa$-presented $\kappa$-frame $\@O(Y')$, hence this algebra is also $\kappa$-presented as a $\kappa$-frame; it is easily seen that letting $\@O(X)$ be this $\kappa$-frame works.
\end{proof}

\begin{definition}[cf.\ \cref{thm:fib-bov-qpol}]
\label{def:loc-fib-bov}
A $\kappa$-Borel bundle of $\kappa$-locales $f : X -> Y$ will be called \defn{$\kappa$-Borel-overt} if every $A \in \@{BO}_{\kappa,f}(X)$ has a $\kappa$-Borel image $f(A) \in \@B_\kappa(Y)$.
In other words, regarded as a $\kappa$-continuous map between the $\kappa$-topologies $\@{BO}_{\kappa,f}(X)$ and $\@B_\kappa(Y)$ as in \cref{def:loc-fib}, $f$ is a $\kappa$-open map.
Note that this notion is stable under increasing $\kappa$ (cf.\ \cref{def:loc-im}).
\end{definition}

\begin{proposition}
\label{thm:loc-fib-bov}
Let $f : X -> Y$ be a standard $\kappa$-Borel-overt bundle of standard $\kappa$-locales.
Then there are compatible standard $\kappa$-topologies $\@O(X) \subseteq \@B_\kappa\@O_f(X)$ and $\@O(Y) \subseteq \@B_\kappa(Y)$ making $f$ continuous open, such that $\@O(X)$ $f$-fiberwise restricts to $\@B_\kappa\@O_f(X)$.
\end{proposition}
\begin{proof}
Same as \cref{thm:fib-bov-qpol}.
\end{proof}

\begin{lemma}[cf.\ \cref{it:fib-bp-diff}]
\label{thm:loc-fib-bov-psneg}
Let $f : X -> Y$ be a fiberwise $\kappa$-based $\kappa$-Borel-overt bundle of $\kappa$-locales.
Then for any $U \in \@{BO}_{\kappa,f}(X)$ and $\kappa$-ary $\kappa$-Borel fiberwise basis $\@W \subseteq \@{BO}_{\kappa,f}(X)$,
\begin{align*}
(\neg U)^\circ_f := \bigcup_{W \in \@W} (W \setminus f^*(f(W \cap U)))
\end{align*}
is the \defn{$f$-fiberwise interior} of $\neg U$, i.e., the largest $\infty$-Borel $f$-fiberwise open set disjoint from $U$.
\end{lemma}
\begin{proof}
It is disjoint from $U$, since for each $W$, we have $W \cap U \setminus f^*(f(W \cap U)) = \emptyset$ since $W \cap U \subseteq f^*(f(W \cap U))$.
Let $V \in \@{BO}_{\infty,f}(X)$ also be disjoint from $U$.
Then $V = \bigcup_{W \in \@W} (f^*(B_W) \cap W)$ for some $B_W \in \@B_\infty(Y)$, so each
$f^*(B_W) \cap W \cap U = \emptyset$, i.e.,
$W \cap U \subseteq f^*(\neg B_W)$, i.e.,
$f(W \cap U) \subseteq \neg B_W$, i.e.,
$B_W \subseteq \neg f(W \cap U)$, whence
$V = \bigcup_{W \in \@W} (f^*(B_W) \cap W)
\subseteq \bigcup_{W \in \@W} (\neg f^*(f(W \cap U)) \cap W)$.
\end{proof}

\begin{definition}
Let $f : X -> Y$ be a $\kappa$-Borel bundle of $\kappa$-locales.
We say that $A \in \@B_\kappa(X)$ is \defn{$f$-fiberwise dense} if it is dense in the usual sense with respect to the global $\kappa$-topology $\@{BO}_{\kappa,f}(X)$, i.e., for every $\emptyset \ne U \in \@{BO}_{\kappa,f}(X)$, we have $A \cap U \ne \emptyset$.
\end{definition}

\begin{remark}
\label{rmk:loc-fib-dense-pb}
By \cite[2.12]{Cpettis} (which is stated for $\kappa = \infty$ but works equally well for all $\kappa$), if $A$ as above is $f$-fiberwise dense, then it remains so after pulling back along any $\kappa$-Borel map $Z -> Y$.

A different, but related, notion of ``fiberwise density'', for a continuous locale map $f : X -> Y$ and sublocale $A \subseteq X$, is defined by Johnstone in \cite{Jgpd}.
In the case $\kappa = \infty$, our notion is precisely the pullback-stable strengthening of Johnstone's; see \cite[\S2]{Cpettis}.
\end{remark}

\begin{remark}
\label{rmk:loc-fib-dense-abs}
Our notion of ``$f$-fiberwise dense'' is \emph{not} stable under increasing $\kappa$ for general $\kappa$-Borel sets.
Indeed, the related notion of ``$\kappa$-Borel surjection'' is not stable \cite[4.4.5]{Cborloc}: there exist $\kappa < \lambda$ and a continuous map $f : X -> Y$ between $\kappa$-locales which is $\kappa$-Borel surjective, but not $\lambda$-Borel surjective.
By ``fiberwise adjoining a least element $\bot$ in the specialization order'' to $X$, i.e., passing to the scone $X^\bot_Y$ over $Y$ (see \cite[C3.6.3]{Jeleph}), we obtain a fiberwise dense $\kappa$-Borel set $X \subseteq X^\bot_Y$ which is no longer fiberwise dense when regarded as a $\lambda$-Borel set.

However, for a $\kappa$-Borel fiberwise \emph{$\kappa$-open} set in a \emph{fiberwise $\kappa$-based} $\kappa$-Borel-\emph{overt} bundle $X -> Y$, being fiberwise dense or not \emph{is} stable under increasing $\kappa$, by \cref{thm:loc-fib-bov-psneg}.
\end{remark}

By the usual Baire category \cref{thm:loc-baire}, applied to $\@{BO}_{\kappa,f}(X)$,

\begin{theorem}[fiberwise Baire category theorem]
For any $\kappa$-Borel bundle of $\kappa$-locales $f : X -> Y$, the intersection of $<\kappa$-many $f$-fiberwise dense $U \in \@{BO}_{\kappa,f}(X)$ is still $f$-fiberwise dense.
\qed
\end{theorem}

\begin{definition}
For a $\kappa$-Borel bundle of $\kappa$-locales $f : X -> Y$, we call a $\kappa$-Borel set $A \in \@B_\kappa(X)$ \defn{$f$-fiberwise comeager} if it contains a $\kappa$-ary intersection of $\kappa$-Borel $f$-fiberwise dense open sets, and \defn{$f$-fiberwise meager} if $\neg A$ is $f$-fiberwise comeager.

The notations $\subseteq^*_f$ and $=^*_f$ have their usual meanings (\cref{def:fib-baire}).

The \defn{$\kappa$-Borel Baire-categorical image} $\exists^*_f(A)$ of $A \in \@B_\kappa(X)$ is the smallest $B \in \@B_\kappa(Y)$ such that $A \subseteq^*_f f^*(B)$, if it exists.
We also put $\forall^*_f(A) := \neg \exists^*_f(\neg A)$, if it exists.
\end{definition}

We now verify that $\exists^*_f$ obeys the obvious $\kappa$-localic analogs of all of the properties from \cref{sec:fib}, at least for a $\kappa$-Borel-overt bundle $f : X -> Y$, which we assume $f$ to be in the following discussion.

\begin{proof}[Proof of \cref{it:fib-baire-im}]
If $A \in \@B_\kappa(X)$ and $f(A) \in \@B_\kappa(Y)$ exists, then $A \subseteq^*_f f^*(f(A))$, whence $\exists^*_f(A) \subseteq f(A)$ assuming $\exists^*_f(A)$ exists (in fact it always does, by \cref{thm:loc-fib-baire-borel} below).
If $A \in \@{BO}_{\kappa,f}(X)$, then for any $B \in \@B_\kappa(Y)$ such that $A \subseteq^*_f f^*(B)$, $A \setminus f^*(B)$ is $f$-fiberwise open and $f$-fiberwise meager, hence empty, whence $f(A) \subseteq B$; this shows $f(A) = \exists^*_f(A)$.
\end{proof}

\cref{it:fib-bp-open} now follows.
We have \cref{it:fib-baire-union} (for $\kappa$-ary unions), since $\exists^*_f$ is defined as a left adjoint (with respect to the preorder $\subseteq^*_f$ on $\@B_\kappa(X)$).
We clearly have \cref{it:fib-bp-union},
and \cref{it:fib-bp-diff} holds for the same reason as before, using the formula for fiberwise interior from \cref{thm:loc-fib-bov-psneg};
whence \cref{it:fib-baire-diff} follows, using \cref{it:fib-baire-im} and Frobenius reciprocity for images.
As before, by induction we now have

\begin{proposition}[cf.\ \cref{thm:fib-baire-borel}]
\label{thm:loc-fib-baire-borel}
Let $f : X -> Y$ be a $\kappa$-continuous $\kappa$-open map between $\kappa$-locales, such that $\@O_\kappa(X)$ is $\kappa$-generated as an $\@O_\kappa(Y)$-algebra.
Then
\begin{enumerate}[label=(\alph*)]
\item (fiberwise Baire property)
For any $A \in \kappa\Sigma^0_\xi(X)$, there is a $U_A = \bigcup_i (f^*(B_i) \cap U_i) \in \@{BO}_{\kappa,f}(X)$, where $B_i \in \kappa\Sigma^0_\xi(Y)$ and $U_i \in \@O(X)$, such that $A =^*_f U_A$.
\item
Thus for any $A \in \kappa\Sigma^0_\xi(X)$, $\exists^*_f(A)$ exists and is in $\kappa\Sigma^0_\xi(Y)$.
\end{enumerate}
In particular, this holds for a fiberwise $\kappa$-based $\kappa$-Borel-overt bundle of $\kappa$-locales $f : X -> Y$.
\qed
\end{proposition}

\begin{corollary}[Beck--Chevalley condition]
\label{thm:loc-fib-baire-bc}
Let $f : X -> Y$ be as above.
For a pullback as in \cref{diag:fib-pb} along a $\kappa$-Borel map $g : Z -> Y$, for $A \in \@B_\kappa(X)$, we have
\begin{equation*}
g^*(\exists^*_f(A)) = \exists^*_{\pi_1}(\pi_2^*(A)).
\end{equation*}
\end{corollary}
\begin{proof}
If $A =^*_f U_A \in \@{BO}_{\kappa,f}(X)$, then $\pi_2^*(A) =^*_{\pi_1} \pi_2^*(U_A)$ by \cref{rmk:loc-fib-dense-pb}; now apply $\exists^*_{\pi_1}$.
\end{proof}

We thus get \cref{it:fib-baire-bc}, from which Frobenius reciprocity \cref{it:fib-baire-frob} and \cref{it:fib-baire-surj} follow.

\begin{remark}
For a fiberwise $\kappa$-based $\kappa$-Borel-overt bundle $f : X -> Y$, the notion of $f$-fiberwise meager $\kappa$-Borel $A \in \@B_\infty(X)$ is stable under increasing $\kappa$, by \cref{rmk:loc-fib-dense-abs}.
Hence so are $=^*_f$, $\subseteq^*_f$.
It follows that $\exists^*_f$ is also preserved under increasing $\kappa$, since it clearly is for fiberwise $\kappa$-open sets.
\end{remark}

\begin{theorem}[Kuratowski--Ulam; cf.\ \cref{thm:kuratowski-ulam}]
\label{thm:loc-kuratowski-ulam}
Let $X --->{f} Y --->{g} Z$ be $\kappa$-Borel maps between $\kappa$-Borel locales, such that $g \circ f : X -> Z$ and $g : Y -> Z$ are $\kappa$-Borel-overt bundles of $\kappa$-locales, $f$ is fiberwise $\kappa$-continuous and $\kappa$-open over $Z$ (i.e., $f, f^*$ map between $\@{BO}_{\kappa,g \circ f}(X), \@{BO}_{\kappa,g}(Y)$), and $\@{BO}_{\kappa,g \circ f}(X)$ has a $\kappa$-ary fiberwise basis $\@W$.
Then
\begin{equation*}
\exists^*_g \circ \exists^*_f = \exists^*_{g \circ f} : \@B_\kappa(X) --> \@B_\kappa(Z).
\end{equation*}
\end{theorem}
\begin{proof}
We follow the proof of \cite[7.6]{Cqpol}.
First, we show that if $A \in \@B_\kappa(X)$ is $(g \circ f)$-fiberwise meager, then $\exists^*_f(A)$ is $g$-fiberwise meager.
By \cref{it:fib-baire-union}, we may assume $A$ is $(g \circ f)$-fiberwise $\kappa$-closed nowhere dense, i.e., $\neg A \in \@{BO}_{\kappa,g \circ f}(X)$, and the $(g \circ f)$-fiberwise interior $A^\circ_f$ (which exists by \cref{thm:loc-fib-bov-psneg}) is empty.
It follows that for each $W \in \@W$, $f(W \setminus A)$ is $g$-fiberwise dense in $f(W)$, since if $V \in \@{BO}_{\kappa,g}(Y)$ with $\emptyset = V \cap f(W \setminus A) = f(f^*(V) \cap W \setminus A)$, then $f^*(V) \cap W \subseteq A$ is fiberwise open, hence empty, whence $V \cap f(W) = f(f^*(V) \cap W) = \emptyset$.
Thus by \cref{it:fib-baire-diff},
\begin{align*}
\exists^*_f(A) = \bigcup_{W \in \@W} (f(W) \setminus f(W \setminus A))
\end{align*}
is a $g$-fiberwise meager $\kappa\Sigma^0_2$ set.
Now to complete the proof: let $A \in \@B_\kappa(X)$ be arbitrary, and let $A =^*_{g \circ f} U \in \@{BO}_{\kappa,g \circ f}(X)$; then by the first part of the proof, $\exists^*_f(A) =^*_g \exists^*_f(U) = f(U) \in \@{BO}_{\kappa,g}(Y)$, whence $\exists^*_g(\exists^*_f(A)) = \exists^*_g(f(U)) = g(f(U)) = \exists^*_{g \circ f}(A)$, using \cref{it:fib-baire-im} three times.
\end{proof}

\subsection{Linear quantifiers}
\label{sec:loc-lin}

\begin{definition}[cf.\ \cref{def:lin,def:fib-lin}; \cite{JTloc}]
\label{def:loc-lin}
A \defn{linear map} between $\kappa$-suplattices is another name for a $\kappa$-suplattice homomorphism.
If $\@R, \@S, \@T$ are $\kappa$-suplattices with bilinear actions $\@R \times \@S -> \@S$ and $\@R \times \@T -> \@T$, an \defn{$\@R$-linear map} $\@S -> \@T$ is an $\@R$-equivariant linear map.

In particular, if $f : X -> Z$ and $g : Y -> Z$ are $\kappa$-continuous maps between $\kappa$-locales, so that $f^* : \@O_\kappa(Z) -> \@O_\kappa(X)$ and $g^* : \@O_\kappa(Z) -> \@O_\kappa(Y)$ are ``algebras'', hence ``modules'', over $\@O_\kappa(Z)$, then we have the notion of $\@O_\kappa(Z)$-linear map $\phi : \@O_\kappa(X) -> \@O_\kappa(Y)$, where equivariance amounts to the Frobenius reciprocity law $\phi(f^*(W) \cap U) = g^*(W) \cap \phi(U)$.
\end{definition}

\begin{remark}[\cite{JTloc}]
\label{rmk:loc-lin-surj}
If $f : X -> Y$ is a $\kappa$-continuous map with an $\@O_\kappa(Y)$-linear retraction $\phi : \@O_\kappa(X) -> \@O_\kappa(Y)$ of $f^*$, then $f$ is a pullback-stable epimorphism in $\!{\kappa Loc}$, i.e., a $\kappa$-Borel surjection.
This is because every pullback $\pi_1 : Z \times_Y X -> Z$ of $f$ along $g : Z -> Y$ also has an $\@O_\kappa(Z)$-linear retraction $\phi_1$ of $\pi_1^*$, defined via the Beck--Chevalley condition
\begin{equation*}
\phi_1(W \times_Y U) := W \cap g^*(\phi(U))
\end{equation*}
using the universal property of $\@O_\kappa(Z \times_Y X)$ as the $\kappa$-suplattice tensor product of $\@O_\kappa(Z), \@O_\kappa(X)$ over $\@O_\kappa(Y)$.
Taking $\@O_\kappa(Z) := \@B_\kappa(Y)$ shows that $f^* : \@B_\kappa(Y) -> \@B_\kappa(X)$ is injective, i.e., $f(X) = Y$.
\end{remark}

\begin{remark}[cf.\ \cref{thm:lin-lowpow,thm:fib-lin-lowpow,thm:fib-lin-supp}]
A linear map $\@O(X) -> \@O(Y)$ extends uniquely to a frame homomorphism from the free frame over $\@O(X)$ \emph{as a suplattice} to $\@O(Y)$; this frame corresponds to the \defn{lower powerlocale} $\@F(X)$, thus the linear map is equivalently a continuous map $Y -> \@F(X)$.
See \cite{Vlog}, \cite{Spowloc} for more information on powerlocales.

Similarly, define the \defn{fiberwise lower powerlocale} $\@F_Z(X)$ of $f : X -> Z$ by taking $\@O(\@F_Y(X))$ to be the free $\@O(Y)$-algebra generated by $\@O(X)$ as an $\@O(Y)$-module.
Then an $\@O(Z)$-linear map $\@O(X) -> \@O(Y)$ is the same thing as a continuous map $Y -> \@F_Z(X)$ over $Z$.
By \cite[3.3]{BFpowloc} (see also \cite{Vpowloc}), for $Y = Z$, the $\@O(Y)$-linear maps $\phi : \@O(X) -> \@O(Y)$ also correspond to fiberwise closed sublocales of $X$ to which the restriction of $f$ is open.
\end{remark}

Using these correspondences, we may give a localic transcription of the proof of \cref{thm:qpol-linquot} and hence \cref{thm:qpol-linquot-borel}.
However, it is easier to bypass powerlocales and just reason algebraically:

\begin{proposition}[cf.\ \cref{thm:qpol-linquot-borel}]
Let $f : X -> Y$ be a $\kappa$-Borel map from a standard $\kappa$-locale to a standard $\kappa$-Borel locale, and let $\phi : \@B_\kappa(X) -> \@B_\kappa(Y)$ be a $\@B_\kappa(Y)$-linear retraction of $f^*$.
Suppose that $f^*(\phi(\@O(X))) \subseteq \@O(X)$, and that $\phi(\@O(X))$ generates $\@B_\kappa(Y)$.
Then $\@O(Y) := \phi(\@O(X))$ is a compatible standard $\kappa$-topology on $Y$ making $f$ continuous.
\end{proposition}
\begin{proof}
$\phi(\@O(X))$ is a $\kappa$-generated $\kappa$-subsuplattice of $\@B_\kappa(Y)$, hence is a subsuplattice, and is closed under finite meets because $\phi$ is a $\@B_\kappa(Y)$-linear retraction (see \cref{eq:qpol-linquot-borel}), hence is a subframe.
Since $\phi(\@O(X))$ is a suplattice retract of the $\kappa$-presented $\@O(X)$, it is $\kappa$-presented as a suplattice, hence also as a frame (because of the posite construction; see e.g., \cite[2.6.8]{Cborloc}).
Since $f^* : \phi(\@O(X)) -> \@O(X)$ has a $\phi(\@O(X))$-linear retraction $\phi$, it extends to an injective homomorphism between the free $\kappa$-Boolean algebras by \cref{rmk:loc-lin-surj}.
Thus the frame inclusion $\phi(\@O(X)) -> \@B_\kappa(Y)$ also extends to an injective $\kappa$-Borel homomorphism from the free $\kappa$-Boolean algebra over $\phi(\@O(X))$, since its composite with $f^* : \@B_\kappa(Y) -> \@B_\kappa(X)$ is injective; but it is also surjective by assumption.
\end{proof}

\begin{corollary}[cf.\ \cref{thm:qpol-baire-subtop}]
Let $f : X ->> Y$ be a $\kappa$-Borel surjection from a standard $\kappa$-locale to a standard $\kappa$-Borel locale, and suppose $X$ is a standard $\kappa$-Borel-overt bundle of standard $\kappa$-locales over $Y$ with another compatible fiberwise $\kappa$-topology $\@B_\kappa\@O_f(X) \subseteq \@B_\kappa(X)$, which has a $\kappa$-ary fiberwise basis $\@W \subseteq \@O(X)$ (thus $\@B_\kappa\@O_f(X)$ is coarser than the fiberwise restriction of $\@O(X)$).
If $f^*(\exists^*_f(\@O(X))) \subseteq \@O(X)$, then $\@O(Y) := \exists^*_f(\@O(X))$ is a compatible standard $\kappa$-topology on $Y$.
\end{corollary}
\begin{proof}
Since $\phi := \exists^*_f$ is $\@B_\kappa(Y)$-linear by \cref{it:fib-baire-frob} (which holds by \cref{thm:loc-fib-baire-bc}), and a retraction of $f^*$ by \cref{it:fib-baire-im}, by the preceding result, we need only check that $\exists^*_f(\@O(X))$ generates $\@B_\kappa(Y)$.
For that, it is enough to check that $\exists^*_f : \@B_\kappa(X) ->> \@B_\kappa(Y)$ lands in the $\kappa$-Boolean subalgebra generated by $\exists^*_f(\@O(X))$.
This follows from the inductive proof of \cref{thm:loc-fib-baire-borel}, using the formulas \cref{it:fib-baire-union,it:fib-baire-diff} and the fiberwise basis $\@W$ for $\@B_\kappa\@O_f(X)$ contained in $\@O(X)$.
\end{proof}

\subsection{Localic groupoid actions}
\label{sec:loc-gpd}

By a \defn{standard $\kappa$-localic groupoid}, we mean an internal groupoid $G = (G_0, G_1, \sigma, \tau, \mu, \iota, \nu)$ in the category $\!{Loc}_\kappa$ of standard $\kappa$-locales.
When $\kappa = \omega_1$, this reduces to the notion of quasi-Polish groupoid by \cref{thm:loc-stdsigma}.
Other notions from \cref{sec:gpd-prelim} such as \defn{open} groupoid, \defn{standard $\kappa$-(Borel) $G$-locale} $p : X -> G_0$, and the \defn{$\alpha$-fiberwise topology} $\@{BO}_\alpha(X)$, as well as the \defn{Vaught transform} $U * A$ from \cref{sec:gpd-vaught}, can now be straightforwardly internalized in $\!{Loc}_\kappa$ or $\!{\kappa BorLoc}_\kappa$, and obey the same properties as before, with proofs using the preceding subsections in place of \crefrange{sec:fib}{sec:fib-lin-baire}.
(Of course, countable unions before are here replaced with $\kappa$-ary ones.)

\Cref{def:gpd-orbtop} of the \defn{orbitwise topology} $\@O_G(X)$ refers to points.
However, all uses of this notion in \cref{sec:gpd} were via \cref{it:gpd-orbtop}, or equivalently its Borel version \cref{it:gpd-orbtop-borel}, since all sets we were working with were Borel.
We can therefore take this as the point-free definition: a $\kappa$-Borel $A \in \@B_\kappa(X)$ is \defn{orbitwise open} if $\alpha^*(A) \in \@B_\kappa(G \times_{G_0} X)$ is $\pi_2$-fiberwise open, which means (by definition of pullback locale) that $\alpha^*(A) \in \@O(G) \otimes_{G_0} \@B_\kappa(X)$.

Every result from \cref{sec:gpd-realiz,sec:gpd-eqvar,sec:gpd-struct} now generalizes essentially verbatim, with the same proof.
In particular, we obtain each result from \cref{sec:grp-realiz,sec:grp-eqvar,sec:grp-struct} generalized to localic groups; the results of \cref{sec:grp-0d-reg} also generalize to localic groups.
We will not repeat every statement in the localic context.
Rather, we only point out the minor changes and new subtleties that arise:
\begin{itemize}

\item
Of course, ``countable'' should be replaced everywhere with ``$\kappa$-ary''.

\item
In \cref{thm:gpd-fib-realiz,thm:gpd-fib-realiz-borel}, \cref*{thm:gpd-fib-realiz:trans} refers to translates by individual groupoid elements $g \in G$, hence should be omitted.
(It is possible to interpret ``generating all $G$-translates $g \cdot A$'' internally as in \cref{thm:loc-kunugui-novikov}; but that theorem says that \cref*{thm:gpd-fib-realiz:rect} is an equivalent point-free condition.)

\item
Likewise, in \cref{def:gpd-action} of \defn{standard $\kappa$-Borel(-overt) $G$-bundle of standard $\kappa$-locales}, the condition ``each morphism $g \in G$ acts via a homeomorphism'' should be interpreted internally, to mean that $\alpha^*(\@B_\kappa\@O_p(X)) \subseteq \@B_\kappa\@O_{\pi_1}(G \times_{G_0} X)$, i.e., ``$(g,x) |-> gx$ is $\pi_1$-fiberwise continuous''.
With this definition, the proof of \cref{thm:gpd-top-realiz} remains the same.

\item
In \cref{thm:gpd-lowpow-univ}, instead of a ``universal $T_0$ second-countable $G$-space $\@F_\tau(G)^\#N_{G_0}$'', we get a universal $\kappa$-based $G$-locale $\@F_\tau(G)^{\kappa_0}_{G_0}$, assuming $\kappa = \kappa_0^+$ is a successor cardinal.
If $\kappa$ is a limit cardinal (hence, being also regular, is weakly inaccessible), then we instead get a $\kappa$-ary family of universal $\kappa$-based $G$-locales, one for each $\kappa_0$ in some cofinal family below $\kappa$.


\item
In \cref{thm:gpd-lowpow-trans}, rather than prove that ``the left translation action is continuous'' (which refers to translating an individual closed set by an individual groupoid element), we \emph{define} the left translation action as a locale map via the formula in \cref{thm:grp-lowpow-trans}, and then check that it is an action with the desired properties.

\item
In \cref{def:etale}, ``countable-to-1'' should of course be replaced by ``$\kappa$-ary-to-1''.
To make sense of this in a point-free manner, we can simply require the $\kappa$-ary analog of the conclusion of the Lusin--Novikov uniformization theorem, i.e., there is a $\kappa$-ary cover of the domain of $p_i$ by $\kappa$-Borel sets to which $p_i$ restricts to a monomorphism.

\item
In \cref{sec:grp-0d-reg}, the results work as stated, yielding zero-dimensional, respectively regular, topological realizations of standard $\kappa$-Borel $G$-locales.
However, we point out that for non-second-countable locales, complete regularity is a more natural separation axiom than mere regularity.
We would thus expect a strengthening of \cref{thm:grp-realiz-dis-0d-reg} yielding a completely regular topological realization.
We will not pursue such a strengthening here, for it seems likely that it would require something akin to the Birkhoff--Kakutani metrization theorem, thus placing it closer to \cite{Hbeckec} and further from the point-free spirit of this paper.

\end{itemize}

\bigskip\noindent
Department of Mathematics \\
University of Michigan \\
Ann Arbor, MI 48109, USA \\
Email: \nolinkurl{ruiyuan@umich.edu}

\end{document}